\documentclass[reqno, 11pt, twoside]{amsart}

\usepackage[margin=3.5cm]{geometry}

\usepackage{amsmath,amssymb,amsthm,bm,colonequals,graphicx,mathrsfs,mathtools,microtype,stmaryrd, amsfonts, dsfont, comment, multicol}
\usepackage[shortlabels]{enumitem}
\usepackage[colorinlistoftodos]{todonotes}

\numberwithin{equation}{section}

\theoremstyle{plain}

\newtheorem{theorem}{Theorem}[section] 

\newtheorem{lemma}[theorem]{Lemma} 

\newtheorem{proposition}[theorem]{Proposition} 

\newtheorem{proposition-definition}[theorem]{Proposition-Definition} 

\newtheorem{corollary}[theorem]{Corollary}

\theoremstyle{definition}

\newtheorem{definition}[theorem]{Definition}

\newtheorem{example}[theorem]{Example}

\theoremstyle{remark}

\newtheorem{remark}[theorem]{Remark}

\newcommand{\abs}[1]{\left\lvert #1 \right\rvert} 
\newcommand{\card}[1]{\#{ #1 }}
\newcommand{\norm}[1]{\lvert #1 \rvert} 
\newcommand{\floor}[1]{\lfloor #1 \rfloor}
\newcommand{\ceil}[1]{\lceil #1 \rceil}

\newcommand{\belongs}{\subseteq}

\newcommand{\eps}{\epsilon}

\newcommand{\defeq}{\colonequals}

\newcommand{\maps}{\colon}

\newcommand{\NN}{\mathbb{N}}
\newcommand{\ZZ}{\mathbb{Z}}
\newcommand{\QQ}{\mathbb{Q}}
\newcommand{\CC}{\mathbb{C}}
\newcommand{\FF}{\mathbb{F}}

\newcommand{\PP}{\mathbb{P}}

\DeclareMathOperator{\supp}{supp}

\DeclareMathOperator{\Br}{Br}

\DeclareMathOperator{\Pic}{Pic}

\def\RR{\mathbb{R}}

\newcommand{\ol}{\overline}

\renewcommand{\hat}{\widehat}

\renewcommand{\leq}{\leqslant}
\renewcommand{\le}{\leqslant}
\renewcommand{\geq}{\geqslant}
\renewcommand{\ge}{\geqslant}
\newcommand{\ve}{\varepsilon}
\renewcommand{\eps}{\varepsilon}
\newcommand{\bve}{\boldsymbol{\varepsilon}}
\newcommand{\lcm}{\mathrm{lcm}}
\newcommand{\ZZp}{\mathbb{Z}_{\text{prim}}}

\usepackage[pagebackref=true]{hyperref}
\usepackage[alphabetic,initials]{amsrefs}

\newcommand{\x}{\mathbf{x}}
\renewcommand{\a}{\mathbf{a}}
\renewcommand{\b}{\mathbf{b}}
\newcommand{\dd}{\mathbf{d}}
\newcommand{\y}{\mathbf{y}}

\newcommand{\m}{\mathbf{m}}
\newcommand{\n}{\mathbf{n}}

\newcommand{\z}{\mathbf{z}}
\newcommand{\0}{\mathbf{0}}
\newcommand{\sumstar}{\sideset{}{^\star}\sum}
\newcommand{\1}{\mathbf{1}}

\usepackage{color}

\title[Beyond the square-root barrier] 
{Beyond the square-root barrier:\\ 
cubic forms of Perazzo type}
\author{Tim Browning}
\author{Ritabrata Munshi}
\author{Victor Y. Wang}

\address{IST Austria\\
Am Campus 1\\
3400 Klosterneuburg\\
Austria}
\email{tdb@ist.ac.at}

\address{Stat-Math Unit\\ Indian Statistical Institute\\ 203 B.T. Road\\ Kolkata 700108\\ India}
\email{ritabratamunshi@gmail.com}

\address{Institute of Mathematics\\
Academia Sinica\\
Taipei 106319\\
Taiwan}
\email{vywang@as.edu.tw}

\subjclass{11D25 (11D45, 11G35, 11P55, 14D10, 14J35)}

 \date{\today}

\begin{document}

\makeatletter
\def\subjclassname{$2020$ Mathematics Subject Classification}
\makeatother

\begin{abstract}
We show how the circle method can be used to study rational points on a certain  cubic fourfold, going beyond the square-root barrier.
\end{abstract}

\maketitle

\thispagestyle{empty}
 \setcounter{tocdepth}{1}
 \tableofcontents

\section{Introduction}

The circle method has long proved an important ally in the quantitative study 
of cubic hypersurfaces $X\subset \PP^{n-1}$ defined 
over $\QQ$.  Let $N_X(B)$ denote the number of rational points of height $B$ on $X$. 
Assuming that $X$ is non-conical and that 
the singular locus has sufficiently large codimension, 
the Manin conjecture \cite{manin}
predicts that $N_X(B)$  should have order $B^{n-3}$ for the standard exponential height function on $\PP^{n-1}(\QQ)$.

Qualitatively,  as outlined by Colliot-Th\'el\`ene  \cite[Appendix]{13}, 
if $n\geq 5$ and $X$ is not a cone, then one
expects  the Hasse principle and weak approximation to hold for the 
smooth locus $X_{\mathrm{smooth}}$ of  $X$, provided that 
the singular locus is  either  empty or has codimension at least $4$ in $X$.
It follows from work of Heath-Brown \cite{14} that $X(\QQ)\neq \emptyset$ for $n\geq 14$, without any  hypotheses on $X$. Likewise, Swarbrick Jones \cite{SJ} has shown that $X_{\mathrm{smooth}}$   satisfies weak approximation provided that 
$n\geq 19$, when   $X$ is   geometrically integral and  not equal to a cone.
When $X$ is smooth, 
Hooley \cites{hooley1, hooley3} has shown that the Hasse principle and weak approximation hold 
provided that $n\geq 9$. 
All of these results rely on the circle method. 
When the cubic hypersurface $X\subset\PP^{n-1}$ contains a set of 
three conjugate singular points,  however, work of 
Colliot-Th\'el\`ene and Salberger \cite{c-t-s} uses descent machinery to show that the smooth Hasse principle holds if $n\geq 4$.
Moreover, if $X$ is geometrically integral and not equal to a cone,  they further show that  weak approximation holds for $X_\mathrm{smooth}$ if $n\neq 5$. (They also show that weak approximation fails when $n=5$, in which case the singular locus has codimension  $3$.)

The well-known {\em square-root barrier} in the circle method seems to preclude handling cubic hypersurfaces in $\PP^{n-1}$ via the circle method unless $n>6$.  One notable exception to this arises in work of the third author \cite{wang2023ratios}, which provides a highly conditional treatment of a smoothly weighted version of the counting function $N_X(B)$ for the smooth cubic fourfold
\begin{equation}\label{eq:6}
x_1^3+x_2^3+x_3^3=x_4^3+x_5^3+x_6^3,
\end{equation}
which
% by Weyl differencing
is $\QQ$-linearly isomorphic to
$\sum_{i=1}^{3} (3x_iy_i^2 + x_i^3) = 0$,
% a union of ternary quadratic equations.
a family of quadrics.
% after blowing up x_1=x_2=x_3=0.
In this paper we study the  fourfold 
  $X\subset \PP^5$ defined by the cubic form
\begin{align}\label{eq:FF}
F(\x, \y)
= x_1y_1^2+x_2y_2^2+x_3y_3^2.
\end{align}
In 1900, Perazzo \cite{perazzo} initiated a programme to classify non-conical cubic hypersurfaces 
with vanishing Hessian, whose singular locus often  contains a linear space.
Perazzo's work was put on a firmer footing by
Gondim and Russo \cite{gondim}, who proved that in some
cases the existence of a linear space of sufficiently high dimension in the singular locus assures
that the cubic hypersurface has vanishing Hessian.  
To illustrate their result,  in \cite[Remark 3.2]{gondim} it is observed that 
the hypersurface $X$ defined by \eqref{eq:FF}
contains the double plane $y_1=y_2=y_3=0$ but has non-vanishing Hessian $H_F=y_1^2y_2^2y_3^2$. 

We shall use the smooth $\delta$-method variant of the circle method to assess the asymptotic behaviour of the counting function
$$
N(B)=\sum_{\substack{\x,\y\in \ZZ^3\\ F(\x, \y)=0}}W\left(\frac{(\x,\y)}{B}\right),
$$
as $B\to \infty$,
for  any smooth weight  function 
$W:\RR^6\to\RR_{\geq 0}$ 
that is compactly supported on 
$\{(\x,\y)\in \RR^6: y_1y_2y_3\neq 0\}$.
Let 
\begin{equation}\label{eq:sigma-inf}
\sigma_\infty =\lim_{\epsilon\to 0} (2\epsilon)^{-1} \int_{|F(\x,\y)|\leq \epsilon} W(\x,\y)\mathrm d\x\mathrm d \y
\end{equation}
be the (weighted) real density associated to $X$. 
Then 
$H_F\neq 0$ 
for any $(\x,\y)\in \supp(W)$, and furthermore, 
it follows from \cite[Thm.~3]{HB} that $\sigma_\infty>0$
provided we choose $W$ so that $F(\x,\y)=0$ for some $(\x,\y)\in \supp(W)$.
The following is our main result. 

\begin{theorem}\label{t:main}
Let $\eps>0$.
Let 
$W:\RR^6\to\RR_{\geq 0}$ 
be any smooth weight  function 
that is compactly supported on 
$\{(\x,\y)\in \RR^6: y_1y_2y_3\neq 0\}$.
Then 
$$
N(B)=
\frac{\sigma_ \infty}{\zeta(3)}
B^3\log B +O(B^3(\log{B})^\eps).
$$
\end{theorem}

The cubic fourfold $X\subset \PP^5$ contains
the double plane $y_1=y_2=y_3=0$ and in  Section~\ref{s:peyre} 
we shall confirm that this asymptotic formula establishes the Manin conjecture \cite{manin} for 
a crepant resolution $\rho:\widetilde X \to X$. We shall  also check that the leading constant agrees with Peyre's prediction \cite{peyre} for $\widetilde X$.

Theorem \ref{t:main} appears to be the first unconditional treatment of a reduced and irreducible cubic fourfold via the circle method. 
As we shall discuss shortly, one of the key steps to circumventing the  square-root barrier arises from 
our analysis of  certain finite field  exponential sums  in the method, which satisfy better than square-root cancellation for the cubic form \eqref{eq:FF}.
It would be interesting to determine exactly which cubic forms
in the $20$-dimensional moduli space of cubic fourfolds
admit such cancellation; for now, we list some examples below.
This cancellation property $(\diamond)$ is closely related, or perhaps equivalent,
to the vanishing of the $A$-number defined by Katz \cite{katz}.

\begin{definition}
Let $N$ be a non-zero integer.
We say that \emph{the property $(\diamond)$ holds}
for a cubic form $F\in \ZZ[\frac1N][x_1,\dots,x_6]$
% for a cubic form $F\in \QQ[x_1,\dots,x_6]$ with coefficients in $\ZZ[\frac1N]$
if there exists a non-zero polynomial $G\in \ZZ[b_1,\dots,b_6]$
and a constant $C>0$ such that
uniformly over vectors $\b\in \ZZ^6$ and primes $p\nmid N\,G(\b)$,
we have $|S_p(\b)| \le Cp^3$, where
$$S_p(\b)\defeq \sum_{a\in \FF_p^\times} \sum_{x_1,\dots,x_6\in \FF_p}\psi(aF(x_1,\dots,x_6) + b_1x_1+\dots+b_6x_6),
$$
for any non-trivial additive character  $\psi:\FF_p\to \CC^*$.
\end{definition}

\begin{example}
\label{prelim-examples}
Preliminary calculations
indicate that $(\diamond)$ holds for
sufficiently general members of the following families,
the first and fourth of which contain \eqref{eq:FF}.
\begin{enumerate}
\item Cubic fourfolds singular along a rational plane, say $\y=\0$.
These take the form $\sum_{i=1}^{3} x_iQ_i(\y) = C(\y)$,
where $Q_1,Q_2,Q_3\in \QQ[\y]$ are ternary quadratic forms
and $C\in \QQ[\y]$ is a ternary cubic form.
The expression $\sum_{i=1}^{3} x_iQ_i(\y)$ defines a \emph{net of conics},
whose classification over $\RR$ and $\CC$
can be found in work of Wall \cite{WallNets}.

\item Fourfolds corresponding to pencils of quadric surfaces.
Their equations are of the form $x_1Q_1(y_1,\dots,y_4)=x_2Q_2(y_1,\dots,y_4)$.
In a biprojective setting,
such equations were studied by
% https://arxiv.org/abs/2204.09322
Bonolis--Browning--Huang \cite{BBH}.

\item Fourfolds defined by trilinear equations
$\sum_{f,g,h\in \{x,y\}} c_{f,g,h} f_1 g_2 h_3 = 0$
in three pairs of variables,
with coefficients $c_{f,g,h}\in \QQ$.
These are specializations of (2).
Dividing by $y_1y_2y_3$, they
may be interpreted in terms of fractions $\frac{x_i}{y_i}$.
Equations of this sort are the subject of many works, including
\cites{BBS,BloomKuperberg,nFracs}.

\item Fourfolds $X$ given by
$\sum_{i=1}^{3} x_iy_i^2 = c_1x_1x_2x_3 + c_2y_1y_2y_3$,
where $c_1,c_2\in \QQ$.
If $c_1=0$, this fits in (1).
% If $c_1=c_2=0$, then $X$ is the subject of Theorem~\ref{t:main}.
If $c_1c_2^2=4$, then $X$ was studied by
Schmidt \cite{S} and Derenthal \cite{ulrich}.
% If $c_1(c_1c_2^2-4) = 0$, then $X$ has a two-dimensional singular locus.
Now suppose $c_1(c_1c_2^2-4)\ne 0$.
Then the singular locus of $X$ consists of three conics,
$y_i=y_j=x_k = y_k^2 - c_1x_ix_j = 0$, where $\{i,j,k\} = \{1,2,3\}$.
A generic hyperplane section of $X$ is a cubic threefold
with six isolated singularities.
% which can be checked to be nodal
% and to have the property that any five nodes are linearly independent.
By methods of \cites{6nodal,dichotomous},
% https://arxiv.org/abs/1501.06432
% https://arxiv.org/abs/2202.10427
it can be shown that $(\diamond)$ holds.
Alternatively, an explicit derivation via Gauss and Sali\'{e} sums
is given for $(c_1,c_2) = (-1,0)$ in Appendix~\ref{elementary}.
% which might be possible to extend to the general case.
\end{enumerate}
The case $c_1(c_1c_2^2-4)\ne 0$ of (4) may be especially deep,
because it seems to have relatively little linear structure.
% and it seems unlikely
% for existing torsor-theoretic or group-theoretic methods to apply.
Theorem~\ref{t:main} might extend to
(4) when $c_1(c_1c_2^2-4)$ is a non-zero square,
% for example if $(c_1,c_2) = (-1,0)$,
in which case we also expect the main term to have order 
$B^3\log B$. 
\end{example}

We proceed to compare  our result with the literature.
Using multiplicative harmonic analysis (as in work of Batyrev and Tschinkel \cite{toric})
or universal torsors
(as in work of Salberger \cite{salberger}),
Manin's conjecture has been proven for toric varieties,
including the singular cubic fourfold in $\PP^5$ given by the equation
% https://zbmath.org/32.0651.04
\begin{equation}
\label{primal}
x_1x_2x_3 = y_1y_2y_3.
\end{equation}
The cubic \eqref{primal} is sometimes called the \emph{Perazzo primal},
although its Hessian is non-zero
and it originates from \cite{primal} rather than from \cite{perazzo}.
Recent 
work  of  Blomer, Br\"udern and Salberger \cite{BBS} is  concerned with another singular cubic fourfold in $\PP^5$, with defining  equation
\begin{equation}\label{eq:BBS}
x_1y_2y_3 + x_2y_1y_3 + x_3y_1y_2=0. 
\end{equation}
Their main result shows that there is a degree $4$ polynomial $P$ and a constant $\delta>0$ 
such that the number of 
rational points of height at most $B$   on the open subset where $y_1y_2y_3\neq 0$ is 
 $B^3P(\log B)+O(B^{3-\delta}).$ 
The proof of this result does not involve  the circle method. Instead it relies on a combination of  elementary lattice point considerations, analytic counting by multiple Mellin integrals, and an Euler product identity for certain multiple Dirichlet series.
As expounded by Derenthal \cite{ulrich}, 
a  further cubic fourfold whose quantitative arithmetic has yielded to investigation 
is cut out by the  senary cubic form arising as the determinant of a symmetric $3\times 3$ matrix, corresponding to taking $(c_1,c_2)=(4,1)$ in Example \ref{prelim-examples}(4).
It is shown in \cite{ulrich} how to relate this counting problem
to a problem about counting  quadratic points in $\PP^2$, which was handled by Schmidt \cite{S} using  the geometry of numbers.

\begin{remark}
    A possible measure of difficulty that can be attributed to cubic fourfolds $X\subset \PP^5$ is the degree of the associated dual  variety $X^*$, which is the image of $X$ under the Gauss map. When $X$ is smooth  it is well-known that $X^*$ is a hypersurface with 
    $\deg (X^*)=48$. 
        In our case, for the cubic defined by the polynomial \eqref{eq:FF}, we shall see that the dual is a hypersurface of degree $6$. On the other hand, the examples \eqref{primal} and \eqref{eq:BBS} are self-dual, so that their duals have degree $3$. Finally, the chordal cubic studied in \cite{ulrich} is dual to the degree $4$ Veronese surface.
\end{remark}

Even for cubic hypersurfaces $X\subset \PP^{n-1}$ in $n=7$ variables, there 
are  relatively few successful treatments via the circle method.  The first of these concerns diagonal cubic hypersurfaces, for which   Baker \cite{baker} has shown that $X(\QQ)\neq \emptyset$ as soon as
$n\geq 7$. Next, consider the family of cubic hypersurfaces
$$
N_{K_1/\QQ}(x_1,x_2,x_3)+N_{K_2/\QQ}(x_4,x_5,x_6)+cx_7^3=0,
$$
where $c\in \QQ^*$ and $N_{K_1/\QQ}$ and $N_{K_2/\QQ}$ are norm forms associated to cubic extensions $K_1,K_2$ of $\QQ$. Birch, Davenport and Lewis \cite{BDL} applied the circle method to prove the smooth Hasse principle for  this family. Building on these examples,  Harvey \cite{harvey} has shown how to handle a hybrid family of cubic hypersurfaces in $\PP^6$ involving a norm form and a diagonal form. (Note that the latter  results 
are special cases of the hypersurfaces studied in 
\cite{c-t-s}.)

\begin{remark}
The cubic form \eqref{eq:FF} 
has  bidegree $(1,2)$ and can also be  viewed as a Fano threefold in 
$\PP^2\times \PP^2$. It would be  interesting to try and confirm the Manin conjecture for this threefold, for which one would need to count with respect to the  anticanonical height function $H(x_1:x_2:x_3)^2H(y_1:y_2:y_3)$, where $H$ is the exponential height function on $\PP^2(\QQ)$. 
Upper and lower bounds 
of  matching order
  were established for this counting function by Le Boudec \cite{Pierre}.
\end{remark}

\medskip

Let us close this introduction with some comments about the proof of Theorem \ref{t:main} and the reason we are able to go beyond the square-root barrier. 
Our work draws 
inspiration from some of the arguments found in \cite{wang2023ratios}, but our work is completely unconditional. 
The starting point is the version of the circle method 
developed by Heath-Brown \cite[Thm.~1]{HB}, which is based on the smooth $\delta$-function technology of  Duke, Friedlander and
Iwaniec \cite{DFI}.  After an application of Poisson summation, the method leads to an analysis of the complete exponential sums
\begin{align*}
S_q(\m,\n)
&=\sumstar_{a\bmod{q}}~~\sum_{\x,\y\bmod{q}} e_q\left(aF(\x,\y)+\m.\x+\n.\y\right),
\end{align*}
for $q\in \NN$ and  $\m,\n\in \ZZ^3$, where  $\sum^\star$ denotes a sum over residues coprime to the modulus.
We might expect  $S_q(\m,\n)$ to have order 
$q^{7/2}$ for typical $q$ and $\m,\n$, if the sum exhibits 
square-root cancellation. 
The exponential sums $S_q(\m,\n)$ are multiplicative functions of $q$ and
%by evaluating 
%the $\x$-sum using orthogonality of characters,
we will be able to show that 
$S_p(\m,\n)=O(p^{3})$, for any prime $p$
% provided that   $q$ is  square-free 
 and suitably generic $\m,\n$,
thereby ensuring that property $(\diamond)$ holds for $F(\x,\y)$.
    This is critical to the success of our proof but it is not in itself enough. 
 Instead, we shall need to exploit additional cancellation obtained through sign changes in these exponential sums, via an analysis of the Dirichlet series
 $$
\xi(s;\m,\n)=\sum_{q=1}^\infty q^{-s}S_q(\m,\n),
$$
for $s\in \CC$, 
when $\m,\n\in \ZZ^3$ satisfy $m_1m_2m_3D(\m,\n)\neq 0$, where 
$D$ is an explicit sextic polynomial given in 
\eqref{eq:Dmn},   defining the dual hypersurface $X^*$. 
For such $\m,\n$, the  series $\xi(s;\m,\n)$ is absolutely convergent
in the half-plane  $\Re(s)>4$.  
In order to prove Theorem~\ref{t:main} it is vital to establish an
analytic continuation of $\xi(s;\m,\n)$ to the left of this
line.

A further difficulty arises when analysing the contribution from square-free $q\in \NN$ and $\m,\n\in \ZZ^3$ such that 
$D(\m,\n)\neq 0$ and $q\mid D(\m,\n)$. In this case $S_q(\m,\n)=c_qq^4$ for a non-negative arithmetic function $c_q$ (with constant average order) and we are precluded from obtaining any additional cancellation in the summation over $q$. 
Morally speaking, we need to estimate sums of the form
$$
\frac{1}{M^6}
\sum_{\substack{|\m|,|\n|\leq M \\D(\m, \n)\neq 0}}
\sum_{\substack{q\sim R\\ q\mid D(\m,\n)}}\mu^2(q),
$$
for $M,R\geq 1$, which we might expect to have order $O(\log M)$.
Unfortunately we cannot afford to lose this factor of $\log M$ in our argument and we shall instead make use of the 
Hooley $\Delta$-function, together with  upper bounds for the average of this function along polynomial 
sequences  provided in work of la Bret\`eche--Tenenbaum \cite{regis} and 
Chan--Koymans--Pagano--Sofos \cite{chan}.
 
One final feature of our work that is worth highlighting 
concerns how the circle method pieces fit together to 
yield the statement of Theorem \ref{t:main}. 
When applying Poisson summation in the smooth $\delta$-function version of the circle method, it  is common for the main term to arise from the 
trivial character. However, in our setting, it  transpires that only  $\frac{3}{4}$ of the main term arises from the vectors $\m,\n$ that both vanish. For the remaining contribution, which amounts to  
$\frac{1}{4}$ of the main term in  Theorem \ref{t:main}, 
we must look to the 
vectors $\m,\n\in \ZZ^3$ with $(\m,\n)\neq (\0,\0)$ for which $D(\m,\n)=0$, where $D$ is the dual form from  above. Comparable phenomena can be found in the work of Wang 
\cite{wang2023ratios}, where the dual variety encodes the contribution from rational planes on the fourfold \eqref{eq:6}, or in work of Heath-Brown \cite{hb-3/2}, where the dual variety handles 
the rational lines on the Fermat cubic surface,
or in works of Vaughan--Wooley \cite{VW} and Br\"udern--Wooley \cite{BW}, that  demonstrate 
how the minor arcs can play a similar role for varieties related to  the Segre cubic. 
In Section \ref{s:dual},
 we shall offer some 
 numerical intuition for why the $\frac{3}{4}:\frac{1}{4}$ split of contributions  is plausible in our setting.
It would be interesting to understand what happens in the circle method for
the cubic investigated in \cite{BBS}, which features a main term of order $B^3(\log{B})^4$, and which may reveal additional features.

\subsection*{Acknowledgements}
The second author was supported by J.C. Bose Fellowship JCB/2021/000018 from ANRF DST
and the third author was supported by the European Union's Horizon~2020 research and innovation programme under the Marie Sk\l{}odowska-Curie Grant Agreement No.~101034413,
and by the National Science and Technology Council Project Grant 114-2115-M-001-010-MY2.

\section{The Manin conjecture}\label{s:peyre}

 Let 
 $X\subset \PP^5$ be the 
 cubic fourfold  defined by the equation 
\[ F(\mathbf{x}, \mathbf{y}) = x_1 y_1^2 + x_2 y_2^2 + x_3 y_3^2 = 0. \] 
 This variety is 
 singular along the plane $\Pi = \{y_1 = y_2 = y_3 = 0\} \cong \mathbb{P}^2$. 
We shall verify that Theorem \ref{t:main} yields a resolution of the Manin conjecture 
 \cite{manin}  for the blow-up  $\rho:\widetilde X\to X$ along $\Pi$,
which  resolves the singularities of $X$.
  This section was written with the help of ChatGPT 5.2 and 
Gemini 3 Pro.
 
 Because $X$ is a cubic hypersurface, its canonical class is $K_X = \mathcal{O}_X(-3)$ 
 and we take the anticanonical height $H(z)=\|\z\|^{3}$, 
 where $\|\cdot\|$ is the Euclidean norm on $\RR^6$ and where
 $\z\in \ZZ^6$ is a primitive non-zero vector representing the rational point 
$z\in \PP^5(\QQ)$, modulo the action of $\pm 1$. 
Pick a smooth weight function $\eta:\PP^5(\RR)\to \RR_{\geq 0}$ which is supported on  a compact subset of 
$\{(\x:\y)\in \PP^5(\RR): y_1y_2y_3\ne 0\}$.
Then we will concern ourselves with the asymptotic behaviour of the counting function
$$
N_X(B)\defeq 
\sum_{\substack{z\in X(\QQ)\\
H(z)\leq  B
}}
\eta(z) 
$$
as $B\to \infty$.
This corresponds
to a smoothly weighted 
variant of the usual counting function with respect to the anticanonical height function. 

Let $\delta>0$ and 
pick smooth weight functions 
$w_\delta^\pm:(0,\infty)\to \RR_{\geq 0}$ such that 
$$
\chi_{(1/2+\delta,1-\delta]}(t)\leq  
w_\delta^-(t)\leq 
\chi_{(1/2,1]}(t)\leq w_\delta^+(t)\leq \chi_{(1/2-\delta,1+\delta]}(t),
$$ 
for all $t>0$.
Define the weight functions $
W^\pm(\z)=\eta\left(\z/\|\z\|\right) w_\delta^\pm(\|\z\|),
$
for any $\z=(\x,\y)\in \RR^6$.  Then $W^\pm:\RR^6\to \RR_{\geq 0}$ are smooth weight functions that are compactly supported on
$\{(\x,\y)\in \RR^6: y_1y_2y_3\neq 0\}$.
A straightforward application of M\"obius inversion and Theorem \ref{t:main} now yields
\begin{align*}
\sum_{z\in X(\QQ)}
\eta(z) w_\delta^\pm \left(\frac{H(z)^{1/3}}{B^{1/3}}\right)
&=\frac{1}{2}\cdot \frac{\sigma_\infty^{(\pm,\delta)}}{\zeta(3)^2} B\log (B^{1/3}) +O(B(\log B)^{\ve})\\
&= \frac{\sigma_\infty^{(\pm,\delta)}}{6\zeta(3)^2} B\log B +O(B(\log B)^{\ve}),
\end{align*}
for any $\ve>0$, 
where in the light of 
\eqref{eq:sigma-inf}, we have 
$$
\sigma_\infty^{(\pm,\delta)} =\lim_{\epsilon\to 0} (2\epsilon)^{-1} \int_{|F(\z)|\leq \epsilon} 
\eta\left(\frac{\z}{\|\z\|}\right) w_\delta^\pm(\|\z\|)
\mathrm d\z.
$$
A simple change of variables yields
$$
\lim_{\epsilon\to 0} (2\epsilon)^{-1} \int_{|F(\z)|\leq \epsilon} 
\eta\left(\frac{\z}{\|\z\|}\right) \chi_{(\alpha,\beta]}(\|\z\|)
\mathrm d\z=(\beta^3-\alpha^3)\lim_{\epsilon\to 0} (2\epsilon)^{-1} \int_{\substack{\|\z\|\leq 1\\  |F(\z)|\leq \epsilon}}
\eta\left(\frac{\z}{\|\z\|}\right) 
\mathrm d\z,
$$
for any $0<\alpha<\beta$, so that 
$
\lim_{\delta\to 0} \sigma_\infty^{(\pm,\delta)}=\left(1-\frac{1}{2^{3}}\right)\sigma_\infty,
$
with 
\begin{equation}\label{eq:sigma-inf'} 
\sigma_\infty=
\lim_{\epsilon\to 0} (2\epsilon)^{-1}
\int_{\{\|\z\|\leq 1: |F(\z)|\leq \epsilon\}}
\eta\left(\frac{\z}{\|\z\|}\right) 
\mathrm d\z.
\end{equation}
Taking $\delta\to 0$, we may conclude  that 
\begin{align*}
\sum_{\substack{z\in X(\QQ)\\
2^{-1}B^{1/3} < H(z)^{1/3}\leq  B^{1/3}
}}
\eta(z) 
= (1+o(1)) \frac{(1-\frac{1}{2^3})\sigma_\infty}{6\zeta(3)^2} B\log B,
\end{align*}
as $B\to \infty$.  It now follows that 
$$
N_X(B)
= (1+o(1)) \frac{\sigma_\infty}{6\zeta(3)^2} B\log B,
$$
on summing over dyadic intervals.
We proceed to compare this result with the Manin prediction for $\widetilde{X}$ in \cite{manin}, together with Peyre's prediction \cite{peyre} for the leading constant. 

 The variety  $\widetilde{X}$ is a smooth variety that admits the structure of a $\mathbb{P}^2$-bundle over $\mathbb{P}^2$,  given by $\widetilde{X} \cong \mathbb{P}_{\mathbb{P}^2}(K \oplus \mathcal{O}_{\mathbb{P}^2}(-1))$, where $K$ is a rank-2 vector bundle over $\mathbb{P}^2$. 
The morphism $\rho$ is crepant, meaning that $K_{\widetilde{X}} = \rho^* K_X$. 
It follows that the exponents of $B$ and $\log B$ are correct, since 
 $\Pic( \widetilde{X})\cong \ZZ^2$.
The predicted leading constant is 
\begin{equation}\label{eq:peyre}
\alpha(\widetilde{X}) \beta(\widetilde{X}) \tau_\infty(\widetilde{X}) \tau_{\mathrm{fin}}(\widetilde{X}). 
\end{equation}
The effective cone $C_{\mathrm{eff}}$ of $\widetilde{X}$ is generated by the exceptional divisor $E$ and the pullback of the hyperplane class from the base, $M = \pi^* \mathcal{O}_{\mathbb{P}^2}(1)$.
Let $L$ be the hyperplane class of $X$. Since $\rho$ is a blow-up, we have $\rho^*L = M + E$ and 
it follows from  adjunction that $K_X = -3L$. 
We deduce that 
$-K_{\widetilde{X}} = -\rho^* K_X = 3(M + E) = 3M + 3E$, 
since $\rho$ is crepant.
Then, with the Lebesgue measure on $N_1(\widetilde{X})_{\mathbb{R}}$ induced by the dual lattice to $\Pic( \widetilde{X})$, the constant $\alpha(\widetilde{X})$ equals the volume of the slice
\[ \{ (y_1, y_2) \in C_{\mathrm{eff}}^\vee : 3y_1 + 3y_2 = 1, y_i \ge 0 \}, \]
which is $(3 \cdot 3)^{-1}$.
Furthermore, 
the constant $\beta(\widetilde{X})$ is defined as the order of the 
Brauer group $\Br(\widetilde{X})/\Br(\QQ)$. 
But then it follows that 
  \begin{equation}\label{eq:ab}
 \alpha(\widetilde{X}) = \frac{1}{9}, \quad 
\beta(\widetilde{X}) = 1 ,
\end{equation}
since $\widetilde X$ is clearly rational.

\begin{lemma}\label{lem1}
The finite Tamagawa number is 
$$
\tau_{\mathrm{fin}}(\widetilde X) = \frac{1}{\zeta(3)^2}.
$$
\end{lemma}
\begin{proof}
We have $\tau_{\mathrm{fin}}(\widetilde X)=\prod_p \tau_p$, where 
$$
\tau_p = \frac{1}{L_p(1,\Pic (\widetilde{X}))} \lim_{k \to \infty} \frac{\#\widetilde{X}(\mathbb{Z}/p^k\mathbb{Z})}{p^{k \dim X}},
$$
with
 $L_p(1,\Pic (\widetilde{X}))=
  (1 - p^{-1})^{-2}$  the convergence factor associated to the split Picard group  of rank $2$.

The original singular variety $X$ has bad reduction at $p=2$,
since the partial derivatives $2x_iy_i$ vanish modulo 2.
Nonetheless, the variety $\widetilde X$ has
a smooth integral model $$\mathcal{\widetilde{X}}
:= \mathbb{P}_{\mathbb{P}^2_{\mathbb{Z}}}(K \oplus \mathcal{O}_{\mathbb{P}^2_{\mathbb{Z}}}(-1))$$
over $\mathbb{Z}$,
constructed via the bundle $K$, defined over $\mathbb{P}^2_{\mathbb{Z}}$ as the kernel of the map
\[ \mathcal{O}_{\mathbb{P}^2_{\mathbb{Z}}}^{\oplus 3} \xrightarrow{(u_1^2, u_2^2, u_3^2)} \mathcal{O}_{\mathbb{P}^2_{\mathbb{Z}}}(2). \]
Because $u_1, u_2, u_3$ are projective coordinates on the base $\mathbb{P}^2_{\mathbb{Z}}$, they share no common zeros in any characteristic.
% Consequently, their squares $u_1^2, u_2^2, u_3^2$ also cannot simultaneously vanish.
% This guarantees that the map is universally surjective over $\mathbb{Z}$.
Thus, $K$ is a well-defined rank-2 vector bundle over $\mathbb{P}^2_{\mathbb{Z}}$,
and $\mathcal{\widetilde{X}}$ is a projective bundle over $\mathbb{P}^2_{\mathbb{Z}}$.

Over any finite field $\mathbb{F}_p$, the number of points on this $\mathbb{P}^2$-bundle over $\mathbb{P}^2$ is exactly the product of the number of points on the base and the fibre, whence
\[ 
\#\widetilde{X}(\mathbb{F}_p) = (\#\mathbb{P}^2(\mathbb{F}_p))^2 = (p^2 + p + 1)^2. 
\]
Because the reduction $\widetilde{X}_{\mathbb{F}_p}$ is smooth, Hensel's lemma implies that 
$$
\#\widetilde{X}(\mathbb{Z}/p^k\mathbb{Z}) = p^{\dim X (k-1)} \#\widetilde{X}(\mathbb{F}_p),
$$
for any $k\geq 1$. 
 The dimension of $X$ is 4, so that 
\[ \tau_p = \left(1 - \frac{1}{p}\right)^2 \frac{(p^2 + p + 1)^2}{p^4} = \left( \frac{p-1}{p} \right)^2 \left( \frac{p^3-1}{p^2(p-1)} \right)^2 = 
\left( 1-\frac{1}{p^3} \right)^2.
\]
Taking the product over all primes easily yields the statement of the lemma.
\end{proof}

\begin{lemma}\label{lem2}
The smoothly weighted archimedean Tamagawa number is 
 $\tau_\infty(\widetilde{X}) = \frac{3}{2} \sigma_\infty$.
\end{lemma}

\begin{proof}
The density \eqref{eq:sigma-inf'} can be  written 
\[ \sigma_\infty = \int_{\{\|\z\|\leq 1:F(\mathbf{z}) = 0 \}}
\eta\left(\frac{\z}{\|\z\|}\right) 
 \omega_{L}, \]
where 
$\omega_{L}$ is the Leray form, defined via  $\mathrm{d}F \wedge \omega_{L} = \mathrm{d}\mathbf{z}$.
By contrast, in Peyre's formalism \cite{peyre}, 
the smoothly weighted archimedean factor 
 $\tau_\infty(\widetilde{X})$ 
 is obtained by integrating the smooth weight $\eta$ against 
 the canonical measure $\omega_X$ on  $X(\mathbb{R})$.
Let $\pi: \mathcal{C}_X \setminus \{0\} \to X(\mathbb{R})$ be the standard projection, 
where $\mathcal{C}_X$ is the affine cone over $X$. 
Setting $t$ as the radial fibre coordinate, $n=6$ as the number of variables, and $d=3$ as the degree, the measures satisfy the identity
\[ 
\omega_{L} = |t|^{n - d - 1} \mathrm{d}t \wedge \pi^* \omega_X = |t|^2 \mathrm{d}t \wedge \pi^* \omega_X. 
\]
By Fubini's theorem, we integrate over the base $X(\mathbb{R})$ and the fiber $t \in [-1, 1]$ to obtain the relation
\[ \sigma_\infty = \int_{X(\mathbb{R})} \eta(z)\left( \int_{-1}^1 |t|^2 \mathrm{d}t \right) \omega_X =\frac{2}{3}
\tau_\infty(\widetilde{X}),
 \]
which completes the proof. 
\end{proof}

Combining Lemmas \ref{lem1} and \ref{lem2} with \eqref{eq:ab} in
\eqref{eq:peyre}, we finally deduce that 
\begin{align*}
\alpha(\widetilde{X}) \beta(\widetilde{X}) \tau_\infty(\widetilde{X}) \tau_{\mathrm{fin}}(\widetilde{X})
&=
\frac{1}{9}\cdot 1\cdot \frac{3}{2}\sigma_\infty \cdot \frac{1}{\zeta(3)^2}
=\frac{\sigma_\infty}{6\zeta(3)^2},
\end{align*}
as required.

\section{Enter the circle method}\label{sec:initial}

We shall free the sum $N(B)$ from the arithmetic  restriction by detecting the equation using the smooth
 $\delta$-method version of the circle method in the form that was developed by 
 Heath-Brown \cite[Thm.~1]{HB}.  
Let $e_q(x)=\exp(2\pi ix/q)$. 
Then 
$$
N(B)=\sum_{\x, \y\in \ZZ^3}W\left(\frac{(\x,\y)}{B}\right)\frac{c_Q}{Q^2}\sum_{q=1}^{\infty}\;\sideset{}{^{\star}}\sum_{a \bmod{q}}e_q\left(aF(\x, \y)\right)h\left(\frac{q}{Q},\frac{F(\x, \y)} {Q^2}\right),
$$
for any $Q\geq 1$, 
for a suitable smooth function $h(x,y)$ and a constant $c_Q$ that satisfies $c_Q=1+O_N(Q^{-N}).$ The notation $\sum^\star_{a \bmod{q}}$ means that the sum is restricted to $a\bmod{q}$ for which $\gcd(a,q)=1$.
Some useful   properties of $h(x,y)$ are recorded 
in \cite[Lemma 4]{HB}, which we proceed to recall here. Firstly, 
 $h(x,y)\neq 0$ only if  $x\leq \max\{1,2|y|\}.$ Moreover, 
\begin{eqnarray}
x^{i} \frac{\partial^i}{\partial x^i}h(x,y)\ll_i x^{-1} & \textnormal{and} & \frac{\partial}{\partial y}h(x,y)=0 ,\label{hbound1}
\end{eqnarray}
for $x\leq 1$ and $|y|\leq x/2$. Also for $|y|\geq x/2$, we have
\begin{equation}
x^i |y|^j \frac{\partial^{i+j}}{\partial x^i\partial y^j}h(x,y)\ll_{i,j} x^{-1}. \label{hbound2}
\end{equation}
In our work we shall  choose $Q=B^{3/2}$.

\subsection*{Poisson summation}
It follows from  Poisson summation that 
$$
\sum_{\x, \y\in \ZZ^3}W\left(\frac{(\x,\y)}{B}\right)e_q\left(aF(\x, \y)\right)h\left(\frac{q}{Q},\frac{F(\x, \y)} {Q^2}\right)=\frac{B^6}{q^6}\sum_{\m, \n\in \ZZ^3} U_q(a;\m,\n)
 I_q(\m,\n),
$$
where 
$$
U_q(a;\m,\n)=\sum_{\x, \y\bmod{q}}e_q(aF(\x,\y)+\m.\x+\n.\y)
$$
and 
\begin{align}\label{eq:int-trans}
I_q(\m,\n)=\int_{\mathbb{R}^6}W(\x,\y)h\left(\frac{q}{Q},F(\x, \y)\right)e_q(-B\m.\x-B\n.\y)\mathrm{d}\x\mathrm{d}\y.
\end{align}
Hence we arrive at the expression
\begin{align}\label{eq:NB-1}
N(B)=c_QB^3\sum_{\m, \n\in \ZZ^3}\sum_{q=1}^{\infty}\frac{1}{q^6} S_q(\m,\n)I_q(\m,\n),
\end{align}
where
\begin{align}\label{eq:char-sum-in}
S_q(\m,\n)=\sideset{}{^{\star}}\sum_{a \bmod{q}}\sum_{\x, \y\bmod{q}}e_q(aF(\x,\y)+\m.\x+\n.\y).
\end{align}
As commented upon in the introduction, a surprising feature of our work is that we shall get main term contributions both from the zero frequency $(\m,\n)=\0$, and  from the vectors $(\m,\n)\neq (\0,\0)$
which vanish on the dual hypersurface. 

\subsection*{The dual form}
A key role will be played by 
the dual form, which  in our setting is given by the  sextic form
\begin{align}\label{eq:Dmn}
D(\m,\n)=\sum_{1\leq i\leq 3} m_i^2n_i^4-2\sum_{1\leq i<j\leq 3} m_im_jn_i^2n_j^2.
\end{align}
Using elementary algebra it is clear that it  satisfies the following factorization property.
\begin{equation}
\label{eq:factorize-dual-form}
D(\m,\n) = (x - y - z) (x + y - z) (x - y + z) (x + y + z),
\end{equation}
where $x,y,z$ are square roots of $m_1n_1^2,m_2n_2^2,m_3n_3^2$, respectively.
In particular $D(\m,\n)$ is irreducible over $\QQ$.
% Later in the proof we shall have need of the upper bound
% % \begin{equation}\label{eq:NT}
% % \sum_{
% % \substack{
% % \m,\n\in \ZZ^3\\
% % |\m|,|\n|\leq T}} \tau(m_1m_2m_3D(\m,\n))^\ve\ll T^6 (\log T)^{16^\ve-1},
% % \end{equation}
% % for any $T\geq 2$ and $\ve>0$.
% % To see this we use the inequality $\tau(mn)\leq \tau(m)\tau(n)$ 
% % and H\"older's inequality 
% % to deduce that 
% % $$
% % \left(\sum_{
% % |\m|,|\n|\leq T} \tau(m_1m_2m_3D(\m,\n))^\ve
% % \right)^4
% % \ll \left(T^6 (\log T)^{2^{4\ve}-1}\right)^3
% % \sum_{
% % |\m|,|\n|\leq T} \tau(D(\m,\n))^{4\ve}.
% % $$
% % A variant of the remaining sum is analysed in 
% % work of Chan et al.\ \cite[Cor.~1.14]{chan}, which readily yields
% % $$
% % \sum_{
% % |\m|,|\n|\leq T} \tau(D(\m,\n))^{4\ve}\ll T^6 (\log T)^{2^{4\ve}-1},
% % $$
% % as required to complete the proof of  \eqref{eq:NT}.
% \begin{equation}\label{eq:NT}
% \sum_{
% \substack{
% \m,\n\in \ZZ^3\\
% |\m|,|\n|\leq T}} \tau(D(\m,\n))
% \ll T^6 \log T,
% \end{equation}
% for any $T\geq 2$,
% which is a special case of
% work of Chan et al.\ \cite[Cor.~1.14]{chan}.

\subsection*{Oscillatory integrals}%\label{s:integrals}

Recalling \eqref{eq:FF}, the  Hessian of $F(\x,\y)$ is easily calculated to be $H_F(\x,\y)=y_1^2y_2^2y_3^2$.
Our assumption on the support of $W$  therefore implies that 
\begin{equation}\label{eq:assume-hess}
H_F(\x,\y)\gg 1 \quad \text{ for all $(\x,\y)\in \supp(W)$}.
\end{equation}
The following result summarises what we need to know about the 
integral $I_q(\m,\n)$ and its partial derivatives with respect to $q$.

\begin{lemma}\label{int-estimate}
Let $j,k\geq 0$ be integers. Then 
$$
q^j\frac{\partial^j}{\partial q^j}
I_q(\m,\n)
\ll_{j,k} \frac{(1+B\max\{|\m|,|\n|\}/q)^{-2}}
{(1+\max\{|\m|,|\n|\}/B^{1/2})^k
\, (1+B\,|\hat{D}(\m,\n)|\max\{|\m|,|\n|\}/q)^k},
$$
where
\begin{equation}
\label{normalize-D}
\hat{D}(\m,\n)
\defeq \frac{D(\m,\n)}{(1+\max\{|\m|,|\n|\})^6}
\ll 1.
\end{equation}
\end{lemma}

\begin{proof}
We have $\nabla{F}(\x,\y) = (y_1^2,y_2^2,y_3^2,2x_1y_1,2x_2y_2,2x_3y_3)$.
So by \eqref{eq:assume-hess},
we have
\begin{equation}
\label{smoothness-on-W}
|\nabla{F}(\x,\y)| \ge \max(y_1^2,y_2^2,y_3^2) \gg 1
\end{equation}
for all $(\x,\y)\in \supp(W)$.
Although the cubic form $F$ is singular,
the inequality \eqref{smoothness-on-W} will be a suitable replacement for smoothness.

Using \eqref{eq:factorize-dual-form},
we easily verify the polynomial divisibility relation
\begin{equation}
\label{poly-div}
F(\x,\y)\mid D(\nabla{F}(\x,\y)).
\end{equation}
Recall the  properties 
\eqref{hbound1} and \eqref{hbound2}
of the $h$ function occurring in definition \eqref{eq:int-trans}
 of the integral $I_q(\m,\n)$.
Using \eqref{eq:assume-hess}, \eqref{smoothness-on-W}, and \eqref{poly-div}, 
as in \cite{wang2026zeta}*{proof of Proposition~8.1},
we obtain the bound in the lemma. 
(A similar, and more transparent, argument is used in  \cite{BGW2024positive}*{\S5}
over a function field.)
The fact that we can take arbitrarily many $q$-derivatives is due to a recursion
recorded in \cite{wang2026zeta}*{Lemma~8.5},
which was observed for $j=1$ by Heath-Brown \cite[Lemma~14]{HB}.
\end{proof}

\subsection*{Summary}

Substituting $c_Q=1+O_N(Q^{-N})$ in 
\eqref{eq:NB-1}, it now follows from Lemma \ref{int-estimate} that 
$$
N(B)=B^3 \left(M(B)+E_1(B)+E_2(B)+E_3(B)\right)+O(1),
$$
where
\begin{equation}\label{eq:main}
M(B)\defeq
\sum_{q=1}^{\infty}\frac{1}{q^6} S_q(\0,\0)I_q(\0,\0)
\end{equation}
and 
\begin{align}\label{eq:E1B}
E_1(B)&\defeq
\sum_{\substack{\m, \n\in \ZZ^3\\
m_1m_2m_3D(\m,\n)\neq 0}}
\sum_{q=1}^{\infty}\frac{1}{q^6} S_q(\m,\n)I_q(\m,\n),\\
E_2(B)&\defeq \label{eq:E2B}
\sum_{\substack{\m, \n\in \ZZ^3\\
m_1m_2m_3=0\\
D(\m,\n)\neq 0}}
\sum_{q=1}^{\infty}\frac{1}{q^6} S_q(\m,\n)I_q(\m,\n),\\
E_3(B)&\defeq \label{eq:E3B}
\sum_{\substack{\m, \n\in \ZZ^3\\
D(\m,\n)=0\\
(\m,\n)\neq (\0,\0)
}}
\sum_{q=1}^{\infty}\frac{1}{q^6} S_q(\m,\n)I_q(\m,\n).
\end{align}
Our primary tasks are to prove that 
$M(B) + E_3(B)$ satisfies an asymptotic formula with main term of order
 $\log B$, as $B\to \infty$, and to prove that $E_i(B)=o(\log B)$, for $1\leq i\leq 2$.
This will be achieved in 
Propositions
\ref{PROP:main},
\ref{p:E1B},
\ref{PROP:E2B},
and \ref{final-E3-estimate}.

\subsection*{Notation and basic facts}

Throughout our work we 
shall let $\ve>0$ be a small parameter, following common convention and allowing it to change value from  appearance to appearance. In any  estimate involving $\ve$ the implied constant will be allowed to depend on $\ve$ in any way. 
Given any non-negative quantities $A,B$,
we shall write $A\asymp B$ when there exist positive constants $c_1<c_2$ such that 
$c_1A\leq B\leq c_2 A$,
and we write $A\sim B$ when $A < B \le 2A$.
We put
$\kappa(q)=\prod_{p\mid q}p,
$ 
for  the square-free kernel of $q\in \NN$. 
We will make frequent use of the  estimate
\begin{equation}\label{eq:kernel-divisor-bound}
\#\{1\le n\leq T: \kappa(n)\mid r\}\ll (rT)^\eps,
\end{equation}
for any $\eps>0$, that follows from a simple application of Rankin's trick. 
Indeed the left hand side is at most 
$$
\sum_{\kappa(n)\mid r} (T/n)^\eps
= T^\eps \prod_{p\mid r} (1 - p^{-\eps})^{-1}
\ll T^\eps r^\eps.
$$
We will also use the standard inequality
\begin{equation}
\label{gcd-1-average}
\sum_{1\le n\le T} \gcd(n,a)
\le \sum_{d\mid a} d \frac{T}{d}
\ll T a^\eps,
\end{equation}
valid for any integer $a\ge 1$.

At various stages of our work we shall need to estimate the number of zeros of polynomials modulo prime powers. 
Let $G\in \ZZ[x_1,\dots,x_n]$ be a polynomial of degree $d$ with content $c(G)$. Then, for any prime power $p^k$, we have  
\begin{equation}\label{eq:regis}
\#\{\x\in (\ZZ/p^k\ZZ)^n: G(\x)\equiv 0\bmod{p^k}\}\leq d^n(k+1)^{n-1}p^{k(n-1/d)+v_p(c(G))/d}.
\end{equation}
There are	upper bounds of   this shape in the literature, but the one we have given is found in work of la Bret\`eche and Tenenbaum 
 \cite{regis}*{Lemma~4.2}.

Finally, we need versions of the prime number theorem
for the Riemann zeta function
and for quadratic Dirichlet $L$-functions.
\begin{lemma}
\label{twisted-PNT}
Let $k,Z,N,P,B\ge 1$ be integers and 
let 
$$
S(N,z) = \sum_{\substack{1\le n\le N \\ \gcd(n,z)=1}} \mu(n).
$$
Then
\begin{equation*}
\sum_{ z\le Z} \tau(z)^k \abs{S(N,Pz)}
\ll_{k,B} \frac{(1+\log{Z})^{2^k-1}}{(1+\log{N})^B} ZN
+ (PZN)^\eps ZN^{1/2}.
\end{equation*}
% old Davenport bound: \ll_{k,B} \frac{(1+\log{Z})^{2^k-1}}{(1+\log{N})^B} \tau(P) ZN.
\end{lemma}

\begin{proof}
The result is standard.
The following argument was found with the help of
Gemini 3 Pro,
improving on a variant that we had based on a result of 
Davenport \cite{davPNT}.
Let $M(x) = \sum_{n\le x} \mu(n)$ be the Mertens function.
Then
\begin{equation*}
S(N,Pz)
= \sum_{\substack{de\le N \\ \kappa(d)\mid Pz}} \mu(e)
= \sum_{\substack{d\le N \\ \kappa(d)\mid Pz}} M(N/d),
\end{equation*}
via the Dirichlet series factorization
$\prod_{p\nmid Pz} (1-p^{-s})
= \zeta(s)^{-1} \prod_{p\mid Pz} (1-p^{-s})^{-1}$.
Trivially $M(N/d)\ll N/d\ll N^{1/2}$
for $d\ge N^{1/2}$, say.
On the other hand, by the prime number theorem,
$M(N/d)\ll_A (N/d)/(1+\log(N/d))^A \ll_A (N/d)/(1+\log{N})^A$
for $d\le N^{1/2}$, for any fixed $A>0$.
Since $\sum_{d\le N^{1/2}} 1/d \ll 1+\log{N}$,
we get 
$$S(N,Pz)\ll_A \frac{N}{(1+\log{N})^{A-1}}
+ \#\{N^{1/2}\le d\le N: \kappa(d)\mid Pz\}\, N^{1/2}.$$
The bounds \eqref{eq:kernel-divisor-bound}
and $\sum_{z\le Z} \tau(z)^k \ll_k Z(1+\log{Z})^{2^k-1}$
complete the proof.
\end{proof}

% \begin{proof}
% Let $T(N,d) = \sum_{\substack{1\le n\le N \\ n\equiv 0\bmod{d}}} \mu(n)$.
% Trivially, $T(N,d)\ll N/d$.
% On the other hand,
% $$
% T(N,d)
% =
% \frac{1}{d}\sum_{h\bmod{d}}
%  \sum_{1\le n\le N}
% \mu(n) e(hn/d),
% $$
% using additive characters to detect the condition $d\mid n$.
% But we can estimate the inner sum using 
% work of  Davenport \cite{davPNT}, 
% which provides the 
% uniform bound
% % Note: A modern reference to Davenport is Green-Tao, The Mobius function is strongly orthogonal to nilsequences, https://arxiv.org/abs/0807.1736
% $\sum_{n\le x} e(n\alpha)\mu(n) \ll_A x/(\log{x})^A$,
% for $x\ge 2$ and $\alpha\in \RR$.
% Thus $T(N,d)\ll_A N/(1+\log{N})^A$.
% By M\"{o}bius inversion,
% $S(N,Pz)
% % = \sum_{d\mid Pz} \mu(d) T(N,d)
% \ll \sum_{d\mid Pz} |T(N,d)|$.
% Letting $g = \gcd(d,P)$ and $d' = d/g$,
% and writing $z = d'y$,
% we find that the sum over $z$ in the lemma is
% \begin{equation*}
% \ll_A \sum_{d\le Z} \sum_{y\le Z/d'} \frac{\tau(d'y)^k N}{\max(d,(\log{N})^A)}
% \ll_k \sum_{d\le Z} \frac{Z(\log{Z})^{2^k-1}}{d'}
% \frac{\tau(d')^k N}{\max(d,(\log{N})^A)},
% \end{equation*}
% since $\tau(d'y)\leq \tau(d')\tau(y)$.
% Since $\max(d,(\log{N})^A)\ge d^{1/2} (\log{N})^{A/2}
% \ge (d')^{1/2} (\log{N})^{A/2}$,
% and the number of possible values of $g$ is $\le \tau(P)$,
% the lemma now follows from the divisor bound $\tau(d')\ll (d')^\eps$.
% \end{proof}

The Jacobi symbol $(\frac{a}{b})\in \{-1,0,1\}$,
where $a,b\in \ZZ$ with $b>0$ and $b\equiv 1\bmod{2}$,
is by definition completely multiplicative in $b$ when $a$ is fixed.
% https://en.wikipedia.org/wiki/Jacobi_symbol#Properties
It is also completely multiplicative in $a$ when $b$ is fixed.
Moreover, we have $(\frac{a^2}{b}) = (\frac{a}{b^2})
= (\frac{a}{b})^2 = \1_{\gcd(a,b)=1}$.
% For each $m\in \ZZ$,
% let $\zeta_{x^2=m}$ be the Hasse--Weil zeta function of the $\QQ$-variety $x^2=m$,
% and let $\chi_m(r)$ be the $r$th coefficient of
% the Dirichlet series $\zeta_{x^2=m}(s)/\zeta(s)$.
% Concretely, if $m$ is not a square,
% then $\zeta_{x^2=m}$ is the Dedekind zeta function of the quadratic field $\QQ(\sqrt{m})$;
% if $m=0$, then $\zeta_{x^2=m}$ is the Riemann zeta function $\zeta$;
% and if $m$ is a non-zero square, then $\zeta_{x^2=m}(s) = \zeta(s)^2$.
% Thus in particular, $\chi_m = \chi_{d^2m}$ for all $d\ge 1$,
% and $(\frac{m}{p}) = \chi_m(p) \bm{1}_{p\nmid m}$ for all odd primes $p$.
The following result is based on Heath-Brown's large sieve for real characters \cite{realLS},
and we  will eventually apply it
with coefficients $c(q)\in \{\mu(q),1\}$.

\begin{lemma}
\label{cheap-large-sieve}
Let $H,M,Q_1\ge 1$
and $I\belongs \{q\sim Q_1\}$.
Let $t\in 2\ZZ$ with $t\ne 0$.
Let $c\maps I\to \CC$ be a function
such that $|c(q)|\le 1$ for all $q\in I$.
Then
\begin{equation*}
\sum_{0\ne h\ll H}
\sum_{0\ne m\ll M}
\left|\sum_{q\in I} c(q) \left(\frac{hm}{q}\right) \1_{\gcd(q,t)=1}\right|
\ll (HMQ_1)^\eps (HMQ_1^{1/2} + (HM)^{1/2}Q_1).
\end{equation*}
\end{lemma}

\begin{proof}
By the divisor bound we may glue $h$ and $m$ into a new variable $m'=hm\ll HM$,
and so reduce to the case where $H=1$, with $h=1$.
Write $m = m_0 m_1 m_2^2$ where $m_1$ is odd, positive, and square-free,
$m_2$ is positive,
and $m_0\in \{\pm 1, \pm 2\}$.
Write $q = q_1 q_2^2$ where $q_1$ is positive and square-free,
and $q_2$ is positive.
Then
\begin{align*}
\left(\frac{m}{q}\right)\1_{\gcd(q,t)=1}
&= \left(\frac{m_0m_1m_2^2}{q_1q_2^2}\right)\1_{\gcd(q,t)=1}\\
&= \left(\frac{m_0}{q_1}\right) \left(\frac{m_1}{q_1}\right)
\1_{\gcd(q_1,tm_2)=1}
\1_{\gcd(q_2,tm_1m_2)=1}.
\end{align*}
The triangle inequality yields
\begin{align*}
\sum_{0\ne m\ll M}
\left|\sum_{q\in I} c(q) \left(\frac{m}{q}\right) \1_{\gcd(q,t)=1}\right|
\le \sum_{m_0,m_2,q_2}
\sum_{m_1\ll M/m_2^2}
\left|\Sigma_{m_0,m_2,q_2}(m_1)\right|
\end{align*}
where
$$
\Sigma_{m_0,m_2,q_2}(m_1)=
\sum_{q_1\in q_2^{-2}I} \left(\frac{m_1}{q_1}\right) c(q_1q_2^2)
\left(\frac{m_0}{q_1}\right) \1_{\gcd(q_1,tm_2)=1}.
$$
The interval $q_2^{-2}I$
and the coefficient $c(q_1q_2^2) (\frac{m_0}{q_1}) \1_{\gcd(q_1,tm_2)=1}$
are independent of $m_1$. Thus
\begin{equation*}
\sum_{m_1\ll M/m_2^2} \left|\Sigma_{m_0,m_2,q_2}(m_1)\right|
\ll (M/m_2^2)^{1/2}
((MQ_1)^\eps (M/m_2^2 + Q_1/q_2^2) (Q_1/q_2^2))^{1/2},
\end{equation*}
by \cite{realLS}*{Theorem~1}
and the Cauchy--Schwarz inequality over $m_1$.
Thus
\begin{equation*}
\sum_{m_0,m_2,q_2}
\sum_{m_1\ll M/m_2^2}  \left|\Sigma_{m_0,m_2,q_2}(m_1)\right|
\ll (MQ_1)^\eps (MQ_1^{1/2} + M^{1/2}Q_1),
\end{equation*}
since $\sum_{m_2\ll M^{1/2}} 1/m_2\ll M^\eps$
and $\sum_{q_2\ll Q_1^{1/2}} 1/q_2\ll Q_1^\eps$.
\end{proof}

\section{Exponential sums}

\subsection*{Auxiliary estimates}

It will be convenient to define 
\begin{equation}\label{eq:sq-gcd}
\{a,b\}=\prod_{p^j\| \gcd(a,b)} p^{2\floor{j/2}}
\end{equation}
for the largest square dividing the greatest common divisor of two integers $a,b$.
For any $q\in \NN$ and $m\in \ZZ$, 
let   $\eta_q(m)$  be the number of $y\in \ZZ/q\ZZ$ such that 
$y^2\equiv m\bmod{q}$. 

\begin{lemma}\label{lem:eta}
Let $q\in \NN$ and $m\in \ZZ$. Then 
 $\eta_q(m)\leq \gcd(q,2) 2^{\omega(q)}\sqrt{\{q,m\}}$. 
\end{lemma}

\begin{proof}
By the  Chinese remainder theorem, it suffices to study $\eta_{q}(m)$ when $q=p^r$ is a prime power.  Let $p^j=\gcd(p^r,m)$. Then the congruence 
$y^2\equiv m\bmod{p^r}$ implies that $p^{\lceil j/2 \rceil}\mid y$.
Writing $m'=m/p^j$, it follows that 
$$
\eta_{p^r}(m)=\#\left\{
y \bmod{p^{r-\lceil j/2 \rceil}} : p^{2\lceil j/2 \rceil-j}y^2\equiv m' \bmod{p^{r-j}}
\right\}.
$$
If $j=r$ then we get 
$$
\eta_{p^r}(m)=p^{r-\lceil r/2 \rceil}
= p^{\floor{r/2}},
$$
which is satisfactory. 
If $j<r$ then $j$ must be even, since  $\gcd(p^{r-j},m')=1$. But then 
$$
\eta_{p^r}(m)=\#\left\{
y \bmod{p^{r- j/2}} : y^2\equiv m' \bmod{p^{r-j}}
\right\}\leq \gcd(p,2) 2p^{j/2},
$$
which is also satisfactory, since $\{p^r,m\}=p^j$ when $j$ is even.
% Note that $1$ has four, not two, square roots modulo $8$.
\end{proof}

Suppose that $q=p^r$ is a prime power and that $m=0$. Then it follows  
 that $j=r$ in the proof of Lemma \ref{lem:eta}. 
But then 
$\eta_{p^r}(0)=p^{r-\lceil r/2 \rceil}$, whence
\begin{equation}\label{eq:eta}
\eta_{p^r}(0)=p^{\lfloor r/2 \rfloor}.
\end{equation}

\subsection*{The dual form redux}

Recall that the dual form for our problem is given by  the  sextic form
$D(\m,\n)$ in \eqref{eq:Dmn}.
It will be convenient to henceforth set
\begin{equation}\label{eq:Gmn}
G(\m,\n) \defeq 6D(\m,\n).
\end{equation}
Define 
\begin{equation}\label{eq:LLi}
L_i(\m,\n) = 2m_in_i^2 - \sum_{1\le j\le 3} m_jn_j^2,
\end{equation}
for $1\leq i\leq 3$.
Then 
it follows that
\begin{equation}\label{eq:derivatives}
\frac{\partial G}{\partial m_i} =12n_i^2 L_i(\m,\n), \quad 
\frac{\partial G}{\partial n_i} =24m_in_i L_i(\m,\n),
\end{equation}
for $1\leq i\leq 3$.

\subsection*{First steps}

For $q\in \NN$ and $(\m,\n)\in \ZZ^6$, we shall now conduct a careful analysis of the exponential sum $S_q(\m,\n)$
that was defined in 
\eqref{eq:char-sum-in}.
The trivial bound is  $|S_q(\m,\n)|\leq q^7$. In what follows,
we shall produce a range of better estimates  for $S_q(\m,\n)$, which get
sharper if $q$ is devoid of certain prime factors related to $\m$ and $\n$. 
Our first observation concerns the basic multiplicativity relation
\begin{equation}\label{eq:multiply}
S_{q_1q_2}(\m,\n)=S_{q_1}(\m,\n)S_{q_2}(\m,\n),
\end{equation}
for any coprime  $q_1,q_2\in \NN$. The proof of this identity is standard and will not be repeated here.
We proceed to record the following expression for  
$S_q(\m,\n)$. 

\begin{lemma}\label{lem:return}
Let $q\in \NN$  and let $\m,\n\in \ZZ^3$. Then 
$$
S_q(\m,\n)=q^3
\sideset{}{^{\star}}\sum_{a \bmod{q}}
\sum_{\substack{\y \bmod{q}\\ ay_i^2+m_i\equiv 0\bmod{q}\\1\leq i\leq 3 }} e_q(\n.\y).
$$
\end{lemma}

\begin{proof}
On recalling the shape \eqref{eq:FF} of $F(\x,\y)$, 
we may  observe that 
$$
S_q(\m,\n)=\sideset{}{^{\star}}\sum_{a \bmod{q}}
\prod_{1\leq i\leq 3} T_q(a;m_i,n_i),
$$
where
$$
T_q(a;m,n)=
\sum_{x, y\bmod{q}}e_q(axy^2+mx+ny),
$$
for any $m,n\in \ZZ$. 
The statement of the lemma is  immediate on  executing the sum over $x$ using orthogonality of characters.
\end{proof}

\begin{lemma}
\label{n-vanishing}
Suppose $S_q(\m,\n)\ne 0$.
Then $\{q,m_i\}^{1/2}\mid n_i$ for all $1\le i\le 3$.
\end{lemma}

\begin{proof}
By \eqref{eq:multiply}, we may assume that $q = p^r$.
Let $p^{j_i} = \gcd(p^r,m_i)$.
The congruence $ay_i^2+m_i\equiv 0\bmod{q}$ implies
$$\gcd(q,y_i^2) = \gcd(q,m_i) = p^{j_i}.$$
In particular, $p^{\lceil{j_i/2}\rceil}\mid y_i$.
Moreover, $(y_i+p^{k_i}t_i)^2\equiv y_i^2\bmod{q}$ for all $t_i\in \mathbb{Z}$,
provided that $q\mid \gcd(2p^{\lceil{j_i/2}\rceil}p^{k_i},p^{2k_i})$.
Replacing $y_i$ with $y_i+p^{k_i}t_i$
in the statement of Lemma~\ref{lem:return},
and averaging over $t_i\bmod{q}$,
we find that $S_q(\m,\n)=0$
unless $n_ip^{k_i}\equiv 0\bmod{q}$ for all $i$.
Taking $$k_i = \max(r-\lceil{j_i/2}\rceil,\lceil{r/2}\rceil),$$
we conclude from the assumption $S_q(\m,\n) \ne 0$ that
\begin{equation*}
v_p(n_i)\ge r-k_i
= \min(\lceil{j_i/2}\rceil,r-\lceil{r/2}\rceil)
\ge \floor{j_i/2},
\end{equation*}
since $r-\lceil{r/2}\rceil=\lfloor{r/2}\rfloor$.
By definition, $\floor{j_i/2} = v_p(\{q,m_i\}^{1/2})$.
\end{proof}

\begin{corollary}\label{cor:basic}
Let $q\in \NN$.
Then 
$$
S_q(\m,\n)\ll 8^{\omega(q)}q^{4}
\prod_{1\le i\le 3} \{q,m_i\}^{1/2} \1_{\{q,m_i\}^{1/2}\mid n_i}.
$$
\end{corollary}

\begin{proof}
By Lemma~\ref{n-vanishing}, we may assume that $\{q,m_i\}^{1/2}\mid n_i$ for  $1\leq i\leq 3$.
On the other hand,
it follows from 
Lemma \ref{lem:return} and 
the triangle inequality that 
$$
|S_q(\m,\n)|\leq q^3\sideset{}{^{\star}}\sum_{a \bmod{q}} \eta_q(-\bar{a}m_1)
\eta_q(-\bar{a}m_2)\eta_q(-\bar{a}m_3),
$$
where $\eta_q(m)$ denotes the number of $y\in \ZZ/q\ZZ$ such that 
$y^2\equiv m\bmod{q}$, for any $m\in \ZZ$. The claimed bound follows on applying Lemma \ref{lem:eta} and summing trivially over $a$.
\end{proof}

\begin{corollary}\label{cor:basic-00}
Let $p^r$ be a prime power. 
Then 
$$
S_{p^r}(\0,\0)= 
p^{4r+3\lfloor r/2 \rfloor}\left(1-\frac{1}{p}\right).
$$
\end{corollary}

\begin{proof}
Taking $\m=\n=\0$ in Lemma \ref{lem:return}, we see that 
$$
S_{p^r}(\0,\0)=p^{3r}\phi(p^r) \eta_{p^r}(0)^3.
$$
The desired result now follows from \eqref{eq:eta}.
\end{proof}

In view of the 
multiplicativity relation 
\eqref{eq:multiply}, it suffices to analyse $S_q(\m,\n)$ when $q=p^r$ is a prime power.
We proceed to record the following result. 

\begin{lemma}\label{lem:primepower}
Let $p^r$ be a prime power. Then 
$$
S_{p^r}(\m,\n)=
\frac{p^{4r}}{\phi(p^r)}\left(N_0(p^r)-p^{-1}N_1(p^r)\right),
$$
where
\begin{equation}\label{eq:Njp}
N_j(p^r)=\#\left\{
(a,\y)\in (\ZZ/p^r\ZZ)^4: 
\begin{array}{l}
p\nmid a, ~\n.\y\equiv 0 \bmod{p^{r-j}}\\
ay_i^2+m_i\equiv 0\bmod{p^r} \text{ for $1\leq i\leq 3$}
\end{array}
\right\},
\end{equation}
for any $j\in \{0,1\}.$
\end{lemma}

\begin{proof}
It follows from 
Lemma \ref{lem:return} that 
$$
S_{p^r}(\m,\n)=
p^{3r}
\sideset{}{^{\star}}\sum_{a \bmod{p^r}}
\sum_{\substack{\y \bmod{p^r}\\ ay_i^2+m_i\equiv 0\bmod{p^r}\\
1\leq i\leq 3 }} e_{p^r}(\n.\y).
$$
We introduce a dummy sum over  $b\in (\ZZ/p^r\ZZ)^*$ and make the  changes of variable 
$\y\mapsto b\y$ and $a\mapsto a\bar b^2$. This leads to the expression
\begin{align*}
S_{p^r}(\m,\n)
&=
\frac{p^{3r}}{\phi(p^r)}
\sideset{}{^{\star}}\sum_{a,b \bmod{p^r}}
\sum_{\substack{\y \bmod{p^r}\\ ay_i^2+m_i\equiv 0\bmod{p^r}\\
1\leq i\leq 3 }} e_{p^r}(b\n.\y)\\
&=
\frac{p^{3r}}{\phi(p^r)}
\sideset{}{^{\star}}\sum_{a \bmod{p^r}}
\sum_{\substack{\y \bmod{p^r}\\ ay_i^2+m_i\equiv 0\bmod{p^r}\\
1\leq i\leq 3 }} 
\left(\sum_{b\bmod{p^r}}
e_{p^r}(b\n.\y)-
\sum_{b\bmod{p^{r-1}}}
e_{p^{r-1}}(b\n.\y)\right),
\end{align*}
from which the statement of the lemma easily follows. 
\end{proof}

\begin{lemma}\label{lem:Njp}
Let $p^r$ be a prime power and let $j\in \{0,1\}$. Then 
\begin{enumerate}
\item[(i)]
 $N_j(p^r) =0$ if $r> j+v_p(D(\m,\n))$;
\item[(ii)]
$N_j(p^r)
 \leq \phi(p^r)$ if $j=1$ and  $r=j+v_p(D(\m,\n))$.
\end{enumerate}
\end{lemma}

\begin{proof}
Recalling the definition \eqref{eq:Njp} of $N_j(p^r)$, we  abuse  notation and 
fix a lift $y_i\in \ZZ_p$ of each residue $y_i\bmod{p^r}$.
We may (uniquely) extend the valuation $v_p$ on $\QQ_p$ to $\ol{\QQ}_p$, keeping $v_p(p)=1$.
Let $\mathcal{O}$ be the integral closure of $\ZZ_p$ in $\ol{\QQ}_p$ and let $z_i\in \mathcal{O}$ be a solution to $az_i^2+m_i = 0$ such that $v_p(y_i-z_i)\ge v_p(y_i+z_i)$.
Then $v_p(y_i^2-z_i^2) \ge r$, since $v_p(a)=0$.
If $\eps_1,\eps_2,\eps_3\in \{\pm 1\}$, then
$$
\n.(\eps_1z_1,\eps_2z_2,\eps_3z_3)
\equiv 0\bmod{p^{\min(r-j,v_p(y_1-\eps_1z_1),v_p(y_2-\eps_2z_2),v_p(y_3-\eps_3z_3))}},
$$
since $\n.\y\equiv 0 \bmod{p^{r-j}}$.
Permuting the indices $i$ if necessary, we may assume without loss of generality that  $v_p(y_1-z_1)\ge v_p(y_2-z_2)\ge v_p(y_3-z_3)$.
Then, taking $\eps_1=1$ and multiplying over all $2^2 = 4$ choices of signs $(\eps_2,\eps_3)$, using the factorization in \eqref{eq:factorize-dual-form}, we find that
\begin{equation*}
v_p(D(\m,\n)) \ge \sum_{(\eps_2,\eps_3)} \min(r-j,v_p(y_1-z_1),v_p(y_2-\eps_2z_2),v_p(y_3-\eps_3z_3)).
\end{equation*}
However, $v_p(y_1-z_1)\ge v_p(y_2-z_2)\ge v_p(y_3-z_3)\ge v_p(y_3+z_3)$,
so $$\min(r-j,v_p(y_1-z_1),v_p(y_2-z_2),v_p(y_3-\eps_3z_3))
= \min(r-j,v_p(y_3-\eps_3z_3)).
$$
Thus, keeping only the terms with $\eps_2=1$
(and noting that each term is non-negative), we get the lower bound
\begin{equation*}
v_p(D(\m,\n)) \ge \sum_{\eps_3} \min(r-j,v_p(y_3-\eps_3z_3)).
\end{equation*}
However, for any reals $a,b,b'\ge 0$,
we have $\min(a,b)+\min(a,b') \ge \min(a,b+b')$,
since $a+a,a+b',b+a\ge a$ and $b+b'\ge b+b'$.
Therefore, we obtain
\begin{equation*}
v_p(D(\m,\n)) \ge \min\left(r-j,\sum_{\eps_3} v_p(y_3-\eps_3z_3)\right)
\ge \min(r-j,r) = r-j.
\end{equation*}
It follows that $N_j(p^r)=0$ unless $r\le j+v_p(D(\m,\n))$, as required for (i).

For (ii) we suppose that $j=1$ and $r=j+v_p(D(\m,\n))$.
Then $v_p(y_3-z_3)>r-j$,
or else we would have $r-j\ge v_p(y_3-z_3)\ge v_p(y_3+z_3)$ and
$$v_p(D(\m,\n)) \ge \sum_{\eps_3} \min(r-j,v_p(y_3-\eps_3z_3))
= \sum_{\eps_3} v_p(y_3-\eps_3z_3) \ge r > r-j.$$
Since $j=1$, this means $v_p(y_3-z_3)\ge r$.
But $v_p(y_1-z_1)\ge v_p(y_2-z_2)\ge v_p(y_3-z_3)$ by assumption.
It follows that $\y\equiv (z_1,z_2,z_3)\bmod{p^r}$,
so that 
$$
N_1(p^r)\le \sum_{a\in (\ZZ/p^r\ZZ)^\times} 1 = \phi(p^r),
$$
as claimed.
\end{proof}

\subsection*{Square-free moduli}

We have multiplicativity of $S_q(\m,\n)$, by 
\eqref{eq:multiply}. Thus for  square-free $q$ we can focus on the evaluation of $S_p(\m,\n)$ with $p$ a prime. 
We proceed to provide a necessary condition for the non-vanishing of $S_p(\m,\n)$.

\begin{lemma}\label{lem:Sp1}
Let $p\nmid 2m_1m_2m_3$ be a prime. Then 
$S_p(\m,\n)=0$ unless 
\begin{equation}\label{eq:m123}
\left(\frac{m_1m_2}{p}\right)=\left(\frac{m_2m_3}{p}\right)=\left(\frac{m_3m_1}{p}\right)=1.
\end{equation}
\end{lemma}

\begin{proof}
Taking $q=p$ in Lemma \ref{lem:return}, we deduce that 
\begin{align*}
S_p(\m,\n)
&=p^3\sum_{a\in \FF_p^*}
\prod_{1\leq i\leq 3}
\sum_{\substack{y\in \FF_p\\ ay^2+m_i\equiv 0\bmod{p}}} e_p(n_iy)\\
&=p^3
\sum_{\substack{\y\in (\FF_p^*)^3\\ 
\overline{m_1}y_1^2\equiv \overline{m_2}y_2^2
\equiv \overline{m_3}y_3^2 \bmod{p}}} 
e_p(\n.\y).
\end{align*}
It is now clear that 
$S_p(\m,\n)=0$ unless \eqref{eq:m123} holds.
\end{proof}

The following result 
gives a complete evaluation of $S_p(\m,\n)$ and depends on the 
sextic form $D(\m,\n)$ that was defined in \eqref{eq:Dmn}. 

\begin{lemma}\label{lem:Sp2}
Let $p>2$ be a prime. 
Suppose  that $p\nmid m_1m_2m_3$ and that \eqref{eq:m123} holds.
Then 
$$
S_p(\m,\n)=
\begin{cases}
-4p^3 & \text{ if $p\nmid D(\m,\n)$,}\\
p^4-4p^3 & 
\text{ if $p\nmid n_1n_2n_3$ and $p\mid D(\m,\n)$,}\\
2p^4-4p^3 & 
\text{ if $p\mid n_1n_2n_3$, $p\nmid \n$, and  $p\mid D(\m,\n)$,}\\
4p^4-4p^3 & 
\text{ if $p\mid  \n$.}
\end{cases}
$$
Suppose  that $p\mid m_1m_2m_3$.
Let $\{i,j,k\}=\{1,2,3\}$ be a permutation such that $p\mid m_j\Leftrightarrow p\mid m_k$.
(Such a permutation exists by the pigeonhole principle.)
Then 
$$
S_p(\m,\n)=
\begin{cases}
p^4-p^3 & \text{ if $p\mid \m$,}\\
p^4\1_{p\mid n_i} -p^3
& \text{ if $p\nmid m_i$, $p\mid (m_j,m_k)$,}\\
0
& \text{ if $p\mid m_i$,  $(\frac{m_jm_k}{p})=-1$,}\\
2p^4 -2p^3
& \text{ if $p\mid (m_i,n_j,n_k)$, $(\frac{m_jm_k}{p})=1$,}\\
p^4 -2p^3
& \text{ if $p\mid (m_i,D(\m,\n))$, $p\nmid (n_j,n_k)$, $(\frac{m_jm_k}{p})=1$,}\\
-2p^3
& \text{ otherwise.}
\end{cases}
$$
\end{lemma}

\begin{proof}
Taking $r=1$ in Lemma \ref{lem:primepower}, we obtain
\begin{equation}\label{eq:bottle}
S_{p}(\m,\n)=
\frac{p^{4}}{p-1}\left(N_0(p)-p^{-1}N_1(p)\right),
\end{equation}
where
$$
N_j(p)=\#\left\{
(a,\y)\in \FF_p^*\times \FF_p^3:
\begin{array}{l}
\n.\y\equiv 0 \bmod{p^{1-j}}\\
ay_i^2+m_i\equiv 0\bmod{p} \text{ for $1\leq i\leq 3$}
\end{array}
\right\},
$$
for any $j\in \{0,1\}.$ It follows from Lemma \ref{lem:Njp} that 
$N_j(p)=0$ unless $j+v_p(D(\m,\n))\geq 1$.  

\medskip
\noindent
{\em The case $p\nmid m_1m_2m_3$.} It  follows from Lemma \ref{lem:Sp1}
that \eqref{eq:m123} holds.  Executing the sum over $a$, we deduce 
that 
\begin{align*}
N_0(p)&=
\#\left\{\y\in (\FF_p^*)^3:  \overline{m_1}y_1^2\equiv \overline{m_2}y_2^2
\equiv \overline{m_3}y_3^2 \bmod{p}, ~\n.\y\equiv 0\bmod{p}\right\},\\
N_1(p)&=
\#\left\{\y\in (\FF_p^*)^3:  \overline{m_1}y_1^2\equiv \overline{m_2}y_2^2
\equiv \overline{m_3}y_3^2 \bmod{p}\right\},
\end{align*}
in  \eqref{eq:bottle}.
It is clear that
$N_1(p)=4(p-1)$,
whence
$$
S_p(\m,\n)
=\frac{p^4N_0(p)}{p-1} -4p^3.
$$

Our calculation of $N_0(p)$ depends on the value of $\n \bmod{p}.$
If $p\nmid D(\m,\n)$ then $N_0(p)=0$ and we are done. Hence we can assume that $p\mid D(\m,\n)$.
Suppose first that $p\mid \n$. Then $N_0(p)=N_1(p)=4(p-1)$ and we are done. 
Henceforth we assume that  $p\nmid \n$, with $p\nmid n_3$, say. 
Then we  can eliminate $y_3$ to deduce that 
\begin{align*}
N_0(p)
&=
\#\left\{(y_1,y_2)\in (\FF_p^*)^2:  \overline{m_1}y_1^2\equiv \overline{m_2}y_2^2
\equiv \overline{m_3n_3^2}(n_1y_1+n_2y_2)^2 \bmod{p}\right\}\\
&=
(p-1)L(p),
\end{align*}
where 
$$
L(p)=
\#\left\{\eta\in \FF_p: 
\eta^2\equiv 	m_1 \overline{m_2} \bmod{p}, ~
n_3^2\eta^2\equiv 
m_1 \overline{m_3}(n_1\eta+n_2)^2 \bmod{p}\right\}.
$$
Thus we have 
$$
S_p(\m,\n)
=p^4L(p) -4p^3,
$$
and our task falls to analysing $L(p)$.

The quantity $L(p)$ is equal to the number of roots modulo $p$ of the pair of equations
$Ax^2+Bx+C=0$
and 
$Dx^2+E=0$, with 
$$
A=m_1n_1^2-m_3n_3^2, \quad B=2n_1n_2m_1,\quad C=m_1 n_2^2, \quad D= m_2, \quad
E=-m_1.
$$
If $p\nmid n_1n_2n_3$ then these quadratic polynomials are not proportional and so it follows that
$$
L(p)
=
\begin{cases}
1 & \text{ if $p\mid \mathrm{Res}(Ax^2+Bx+C,Dx^2+E)$,}\\
0 & \text{ if $p\nmid \mathrm{Res}(Ax^2+Bx+C,Dx^2+E)$.}
\end{cases}
$$
But 
\begin{align*}
\mathrm{Res}(Ax^2+Bx+C,Dx^2+E)
&=
C^2D^2 + B^2DE - 2ACDE + A^2E^2\\
&=m_1^2D(\m,\n),
\end{align*}
in the notation of \eqref{eq:Dmn}. This therefore completes the proof of the lemma when 
$p\nmid n_1n_2n_3$. 

Suppose finally that $p\nmid n_3$ and $p\mid n_1n_2$. Then $B\equiv 0\bmod{p}$ and the polynomials  $Ax^2+C$
and 
$Dx^2+E$ are proportional modulo $p$ if and only if 
$m_1n_1^2-m_3n_3^2\equiv m_2n_2^2\bmod{p}$, which is if and only if
$p\mid D(\m,\n)$ when $p\mid n_1n_2$.
But then  $L(p)=2$, since 
$(\frac{m_1m_2}{p})=1$. 

\medskip
\noindent
{\em The case $p\mid m_1m_2m_3$.}
We return to \eqref{eq:bottle}. If $p\mid \m$ then $N_0(p)=N_1(p)=p-1$ and 
so $S_p(\m,\n)=p^4-p^3$.  
Without loss of generality we work with the permutation $(i,j,k)=(3,1,2)$ in the statement of the lemma.
Suppose  that $p\mid (m_1,m_2)$ and $p\nmid m_3$.
Then 
$$
N_j(p)=\#\left\{
(a,y_3)\in \FF_p^*\times \FF_p:
\begin{array}{l}
n_3y_3\equiv 0 \bmod{p^{1-j}}\\
ay_3^2+m_3\equiv 0\bmod{p}
\end{array}
\right\},
$$
for  $j\in \{0,1\}.$ But then  $N_0(p)=(p-1)\1_{p\mid n_3}$ and $N_1(p)=p-1$, whence
$
S_{p}(\m,\n)=p^4\1_{p\mid n_3}-p^3.
$
Finally, we suppose that $p\nmid m_1m_2$ and $p\mid m_3$.
Then 
$$
N_j(p)=\#\left\{
(a,y_1,y_2)\in \FF_p^*\times \FF_p^2:
\begin{array}{l}
n_1y_1+n_2y_2\equiv 0 \bmod{p^{1-j}}\\
ay_i^2+m_i\equiv 0\bmod{p} \text{ for $i=1,2$}
\end{array}
\right\},
$$
for $j\in \{0,1\}.$ In particular $N_j(p)=0$ unless $(\frac{m_1m_2}{p})=1$, in which case
$$
N_j(p)=\#\left\{
(y_1,y_2)\in \FF_p^*\times \FF_p^*:
\begin{array}{l}
n_1y_1+n_2y_2\equiv 0 \bmod{p^{1-j}}\\
\overline{m_1}y_1^2\equiv 
\overline{m_2}y_2^2
\bmod{p}
\end{array}
\right\}.
$$
Clearly $N_1(p)=2(p-1)$. If $p\mid (n_1,n_2)$ then also $N_0(p)=2(p-1)$. If 
$p\nmid (n_1,n_2)$ then 
$$
N_0(p)=\begin{cases}
p-1 & \text{ if $p\mid D(\m,\n)$,}\\
0 & \text{ otherwise,}
\end{cases}
$$
since 
$m_1n_1^2\equiv m_2n_2^2\bmod{p}$ if and only if 
$p\mid D(\m,\n)$,  when $p\mid m_3$. The statement of the lemma follows.
\end{proof}

When $p=2$ we trivially have $|S_2(\m,\n)|\leq 2^7$.
Noting that $p\mid D(\m,\n)$ if $p\mid \n$, the following result is an immediate consequence of combining  Lemma \ref{lem:Sp2} with 
\eqref{eq:multiply}.

\begin{corollary}\label{cor.3.8}
Assume that $q\in \NN$ is square-free. Then 
$$
S_q(\m,\n)\ll 4^{\omega(q)}q^3\gcd(q,D(\m,\n)).
$$
\end{corollary}

\subsection*{Square-full moduli}

For our  next result we recall  the notation
$\kappa(q)=\prod_{p\mid q}p,
$ 
for the square-free kernel of $q$. 

\begin{lemma}\label{lem:square-full}
Let $q$ be square-full. Then $S_q(\m,\n)=0$ unless
$\kappa(q)\mid D(\m,\n)$ and
$q\mid \kappa(q) D(\m,\n)$.
\end{lemma}

\begin{proof}
It suffices to prove the result when $q=p^r$, with $r\geq 2$.
Lemma \ref{lem:primepower} implies that 
$$
S_{p^r}(\m,\n)=
\frac{p^{4r}}{\phi(p^r)}\left(N_0(p^r)-p^{-1}N_1(p^r)\right).
$$
Suppose that $S_{p^r}(\m,\n)\neq 0$. Then 
$N_0(p^r)\neq 0$
or $N_1(p^r)\neq 0$. This is only possible if $r\leq 1+v_p(D(\m,\n))$, by Lemma \ref{lem:Njp}, which implies that
$p\mid D(\m,\n)$, since 
$r\geq 2$.
\end{proof}

\begin{lemma}
\label{grad1}
Let $p$ be a prime.
Assume $p\mid D(\m,\n)$
and $p\nmid \nabla G(\m,\n)$.
Then $p\nmid n_1n_2n_3$.
Moreover, if $S_{p^r}(\m,\n)\ne 0$ for some $r\ge 2$,
then $p\nmid m_1m_2m_3$.
\end{lemma}

\begin{proof}
Suppose $p\mid n_i$ for some $i$.
Then $p\mid \frac{\partial G}{\partial m_i},\frac{\partial G}{\partial n_i}$
by \eqref{eq:derivatives}.
Let $\{i,j,k\}=\{1,2,3\}$.
Then $p\mid D(\m,\n)$ implies $p\mid (m_jn_j^2-m_kn_k^2)^2$, by \eqref{eq:Dmn}.
So $p\mid m_jn_j^2-m_kn_k^2-m_in_i^2 = L_j(\m,\n)$, in the notation of \eqref{eq:LLi}.
Thus $p\mid \frac{\partial G}{\partial m_j},\frac{\partial G}{\partial n_j}$
by \eqref{eq:derivatives}.
Similarly, $p\mid L_k,\frac{\partial G}{\partial m_k},\frac{\partial G}{\partial n_k}$.
It follows that  $p\mid \nabla{G}(\m,\n)$, which is a contradiction.

Thus we have shown that  $p\nmid n_1n_2n_3$.
Assume now that $S_{p^r}(\m,\n)\ne 0$ for some $r\ge 2$, and 
suppose $p\mid m_i$ for some $i$.
Since $r\ge 2$ and $S_{p^r}(\m,\n)\ne 0$,
Lemma~\ref{lem:return} forces $v_p(m_i)\ge 2$.
But then $p^2\mid \{p^r,m_i\}$,
so Lemma~\ref{n-vanishing} forces $p\mid n_i$.
This is impossible.
Thus in fact $p\nmid m_1m_2m_3$.
\end{proof}

\subsection*{A further bound}

Recall the definition \eqref{eq:Gmn} of $G(\m,\n)$. 
Our final bound for $S_q(\m,\n)$ is valid for any $q\in \NN$ 
devoid of certain prime factors.

\begin{lemma}\label{lem:beach}
Let $p^r$ be a prime power such that $p\nmid \nabla G(\m,\n)$. Then 
$$
|S_{p^r}(\m,\n)|\leq p^{4r}. 
$$
\end{lemma}

\begin{proof}
Since $p\nmid \nabla G(\m,\n)$, we can henceforth assume that $p\geq 5$.
Moreover, by the first part of Lemma~\ref{grad1}, we may assume that 
$p\nmid D(\m,\n)$ or $p\nmid n_1n_2n_3$.

Suppose first that $r=1$ and $p\nmid m_1m_2m_3$. 
 The first part of Lemma 
\ref{lem:Sp2} yields
\begin{align*}
|S_{p}(\m,\n)|&\leq 
\begin{cases}
4p^{3} &\text{ if $p\nmid D(\m,\n)$,}\\ 
p^4-4p^{3} &\text{ if $p\mid D(\m,\n)$ and $p\nmid n_1n_2n_3$,}\\
\end{cases}\\
&\leq p^4,
\end{align*}
since  $p\geq 5$.
We  now deal with the case  $r=1$ and $p\mid m_1m_2m_3$. 
By the second part of  Lemma \ref{lem:Sp2} we see that $|S_p(\m,\n)|\leq p^4$ unless 
$p\mid (m_i,n_j,n_k)$, for some permutation $\{i,j,k\}=\{1,2,3\}$.
% But clearly 
% $p\mid (L_1(\m,\n),L_2(\m,\n), L_3(\m,\n))$ in this case, which is impossible.
But clearly $p\mid D(\m,\n)$ and $p\mid n_1n_2n_3$ in this case, which is impossible.

	Suppose next that $r\geq 2$. 
It follows from Lemmas~\ref{lem:square-full}
and~\ref{grad1}
that we may proceed under the assumption that 
$p\mid  D(\m,\n)$ and  $p\nmid m_1m_2m_3n_1n_2n_3$.
We  return to 
Lemma \ref{lem:primepower}.
Since  $p\nmid m_1m_2m_3$, we can assume $p\nmid y_1y_2y_3$ for the $\y$ counted in $N_j(p^r)$.  Eliminating $a$, we obtain
$$
S_{p^r}(\m,\n)=
\frac{p^{4r}}{\phi(p^r)}
(M_0-p^{-1}M_1),
$$
where
\begin{align*}
M_j
&=\#\left\{
\y\in (\ZZ/p^r\ZZ)^3: 
\begin{array}{l}
\n.\y\equiv 0 \bmod{p^{r-j}},~p\nmid y_1y_2y_3\\
y_1^2\overline{m_1}\equiv
y_2^2\overline{m_2}\equiv
y_3^2\overline{m_3} \bmod{p^{r}}
\end{array}
\right\}\\
&=\phi(p^r)\#\left\{
(\eta_1,\eta_2)\in (\ZZ/p^r\ZZ)^2: 
\begin{array}{l}
n_1\eta_1+n_2\eta_2+n_3\equiv 0 \bmod{p^{r-j}},~p\nmid \eta_1\eta_2\\
\eta_1^2\equiv m_1\overline{m_3} \bmod{p^r}, ~
\eta_2^2\equiv m_2\overline{m_3} \bmod{p^r}
\end{array}
\right\},
\end{align*}
for  $j\in \{0,1\}.$
In particular it follows that $m_im_3$ must be a quadratic residue modulo $p^r$, for $1\leq i\leq 3$. 
We can write 
\begin{align*}
M_j
&=\phi(p^r)
\sum_{(\eta_1,\eta_2)\in U(\m)}
\delta(\eta_1,\eta_2),
\end{align*}
for  $j\in \{0,1\}$, 
where $U(\m)$ is the set of $(\eta_1,\eta_2)\in (\ZZ/p^r\ZZ)^2$ such that 
$\eta_i^2\equiv m_i\overline{m_3}\bmod{p^r}$ for $i=1,2$, and 
$$
\delta(\eta_1,\eta_2)=
\begin{cases}
1 & \text{ if 
$n_1\eta_1+n_2\eta_2+n_3\equiv 0 \bmod{p^{r-j}}$,}\\
0 & \text{ otherwise.}
\end{cases}
$$
We claim that $\#U(\m)\leq 1$ if $p\nmid \nabla G(\m,\n)$. 
From this it follows that 
$M_j \leq \phi(p^r)$, for $j\in \{0,1\}$, whence
$$
|S_{p^r}(\m,\n)|=
\frac{p^{4r}}{\phi(p^r)}
\max\{M_0-p^{-1}M_1, p^{-1}M_1-M_0\}\leq \frac{p^{4r}}{\phi(p^r)} \max\{M_0,M_1\}\leq p^{4r},
$$
as desired.

To prove the claim, we suppose that 
$\#U(\m)\geq 2$. Then, without loss of generality we may assume that 
$\delta(\eta_1,\eta_2)=\delta(-\eta_1,\eta_2)=1$.
Since $r-j\geq 1$,  we must then have 
$$
\pm n_1\eta_1+n_2\eta_2+n_3\equiv 0 \bmod{p},
$$ 
whence
$n_2\eta_2\equiv -n_3\bmod{p}$ and $p\mid n_1$. But this implies that
$m_2n_2^2\equiv m_3n_3^2\bmod{p}$ and $p\mid n_1$,
% whence in fact $p\mid (L_1(\m,\n),
% L_2(\m,\n),L_3(\m,\n))$, which is not allowed.
contradicting the fact that $p\nmid n_1n_2n_3$.
\end{proof}

\subsection*{A convolution of exponential sums}

Guided by the ``generic'' case $p\nmid G(\m,\n)$ of Lemma \ref{lem:Sp2}, we set
\begin{equation}
\label{eq:S1q}
S^{(1)}_q(\m,\n)
\defeq \prod_{\substack{p\parallel q \\ p\nmid G(\m,\n)}} (-p^3\mathfrak{Q}(\m,p)),
\end{equation}
where $\mathfrak{Q}(\m,p)\defeq 0$ for $p\in \{2,3\}$ (for later convenience) and
\begin{equation}
\label{eq:Qmp}
\mathfrak{Q}(\m,p) \defeq 1 + \left(\frac{m_1m_2}{p}\right)+\left(\frac{m_2m_3}{p}\right)+\left(\frac{m_3m_1}{p}\right)\in \{0,1,2,4\}
\end{equation}
for $p\ge 5$.
In particular, $S^{(1)}_1(\m,\n) = 1$
and 
\begin{equation}\label{eq:SUP}
q^{-3}|S^{(1)}_q(\m,\n)|
\leq 4^{\omega(q)}\leq 
\tau(q)^2.
\end{equation}
Define $S^{(2)}_q(\m,\n)$ so that
\begin{equation}\label{eq:q1q2}
S_q(\m,\n) = \sum_{q_1q_2=q} S^{(1)}_{q_1}(\m,\n) S^{(2)}_{q_2}(\m,\n).
\end{equation}
Explicitly at prime powers, we have 
$$
S_{p^r}(\m,\n) = S^{(2)}_{p^r}(\m,\n)
- \mathfrak{Q}(\m,p)p^3 S^{(2)}_{p^{r-1}}(\m,\n) \bm{1}_{p\nmid G(\m,\n)},
$$
so that
\begin{equation}\label{eq:S2-as-S}
S^{(2)}_{p^r}(\m,\n) = S_{p^r}(\m,\n)
+ \sum_{j=1}^r (\mathfrak{Q}(\m,p)p^3)^j S_{p^{r-j}}(\m,\n) \bm{1}_{p\nmid G(\m,\n)}.
\end{equation}
Our remaining results in this section comprise of various estimates and observations about the sums 
$S^{(2)}_{q}(\m,\n)$, for $\m,\n\in \ZZ^3$. 
We begin with a crude upper bound, in which we recall the notation $\{q,m\}$ from \eqref{eq:sq-gcd}.

\begin{lemma}\label{lem:L1}
Let $q\in \NN$. Then 
$$
S_q^{(2)}(\m,\n)\ll q^{4+\ve} \sqrt{\{q,m_1\}\{q,m_2\}\{q,m_3\}},
$$
for any $\ve>0$.
\end{lemma}

\begin{proof}
It suffices to establish this result when $q=p^r$ is a prime power. 
When $r=1$ it follows from \eqref{eq:S2-as-S} and 
Corollary \ref{cor.3.8} that 
$$
S^{(2)}_{p}(\m,\n) = S_{p}(\m,\n)
+ \mathfrak{Q}(\m,p)p^3 \bm{1}_{p\nmid G(\m,\n)} \ll p^4.
$$
Suppose next that $r\geq 2$. Then we deduce from 
\eqref{eq:S2-as-S} and 
Corollary \ref{cor:basic} that 
\begin{align*}
|S^{(2)}_{p^r}(\m,\n)|
&\leq
\sum_{j=0}^{r-1} (\mathfrak{Q}(\m,p)p^3)^j |S_{p^{r-j}}(\m,\n)|+ (\mathfrak{Q}(\m,p)p^3)^r\\
&\ll 
\sum_{j=0}^{r-1} (4p^3\bm{1}_{p\ge 5})^j 
p^{4(r-j)}\sqrt{\{p^{r-j},m_1\}\{p^{r-j},m_2\}\{p^{r-j},m_3\}}
+ (4p^3\bm{1}_{p\ge 5})^r\\
&\ll p^{4r}\sqrt{\{p^{r},m_1\}\{p^{r},m_2\}\{p^{r},m_3\}}
\sum_{j=0}^{r-1} \left(\frac{4}{p}\bm{1}_{p\ge 5}\right)^j
+ (4p^3\bm{1}_{p\ge 5})^r,
\end{align*}
since $\{p^{r-j},m\}\leq \{p^r,m\}$ for any $m\in \ZZ$.
The remaining geometric series is $O(1)$, since $\frac{4}{p}\bm{1}_{p\ge 5} < 1$,  and so the statement of the lemma follows.
\end{proof}

Our next result gives conditions under which  $S^{(2)}_{q}(\m,\n)$ vanishes. 

\begin{lemma}\label{lem:L2}
Suppose that $p\nmid G(\m,\n)$. Then $S^{(2)}_{p^r}(\m,\n)=0$ for any integer $r\geq 1$.
\end{lemma}

\begin{proof}
Returning to 
\eqref{eq:S2-as-S}, we obtain
$$
S^{(2)}_{p^r}(\m,\n) = 
(\mathfrak{Q}(\m,p)p^3)^r+(\mathfrak{Q}(\m,p)p^3)^{r-1} S_{p}(\m,\n),
$$
since $S_{p^j}(\m,\n)=0$ for $j\geq 2$ (by Lemma~\ref{lem:square-full}).
We now proceed by casework on which coordinates of $\m$ are divisible by $p$.
We note that $p\ge 5$, since $p\nmid G(\m,\n)$.

Suppose first that $p\nmid m_1m_2m_3$.
If \eqref{eq:m123} holds,
then $S_{p}(\m,\n)=-4p^3$ by the first part of Lemma~\ref{lem:Sp2},
and $\mathfrak{Q}(\m,p)=4$ by \eqref{eq:Qmp}. Thus 
 $S^{(2)}_{p^r}(\m,\n)=0$ in this case.
If \eqref{eq:m123} does not hold,
then $S_{p}(\m,\n)=0$ by Lemma~\ref{lem:Sp1},
and $\mathfrak{Q}(\m,p)=0$ by \eqref{eq:Qmp}. Thus $S^{(2)}_{p^r}(\m,\n)=0$ in this case.

Suppose next that $p\mid m_1m_2m_3$ and 
$p\nmid G(\m,\n)$. Thus there exists a permutation $\{i,j,k\}=\{1,2,3\}$ such that 
$p\nmid m_jm_k,\; p\mid m_i,\; p\nmid (n_j,n_k)$,
or $p\mid (m_j,m_k),\; p\nmid m_in_i$.
In the former case, the second part of Lemma~\ref{lem:Sp2} implies that
$S_p(\m,\n) = -(1+(\frac{m_jm_k}{p}))p^3$,
whereas $\mathfrak{Q}(\m,p)=1+(\frac{m_jm_k}{p})$ by \eqref{eq:Qmp},
so $S^{(2)}_{p^r}(\m,\n)=0$.
In the latter case, Lemma~\ref{lem:Sp2} implies that
$S_p(\m,\n) = -p^3$,
whereas $\mathfrak{Q}(\m,p)=1$ by \eqref{eq:Qmp},
so $S^{(2)}_{p^r}(\m,\n)=0$.
\end{proof}

Let $p^r$ be a prime power. On combining 
\eqref{eq:S2-as-S} with 
Lemma \ref{lem:primepower}, we obtain
\begin{equation}\label{eq:fount}
S^{(2)}_{p^r}(\m,\n)
= S_{p^r}(\m,\n)
= \frac{p^{4r}}{\phi(p^r)}\left(N_0(p^r)-p^{-1}N_1(p^r)\right)
\end{equation}
for any $p\mid G(\m,\n)$.

\begin{lemma}\label{lem:L3}
Let $q\in \NN$.
Suppose that for every $p\mid q$, we have $p\nmid \nabla G(\m,\n)$. Then 
% $$
% S^{(2)}_{q}(\m,\n)\ll \left(\frac{q}{\phi(q)}\right)^4 q^4.
% $$
$$
\abs{S^{(2)}_{q}(\m,\n)}\le q^4.
$$
\end{lemma}

\begin{proof}
% Let $p^r$ be a prime power such that $p\nmid \nabla G(\m,\n)$. In particular, we may assume that $p\geq 5$. 
% It now follows from 
% \eqref{eq:S2-as-S} and Lemma \ref{lem:beach} that 
% \begin{align*}
% |S^{(2)}_{p^r}(\m,\n)|\leq (4p^3)^r
% +
% \sum_{j=0}^{r-1} (4p^3)^j\cdot p^{4(r-j)}
% &=
% p^{4r}\sum_{j=0}^{r} \left(\frac{4}{p}\right)^j\\
% &=\left(1-\frac{1}{p}\right)^{-4}\left(1+O(p^{-2})\right).
% \end{align*}
% Taking a product over the prime divisors of $q$ we are easily led to the statement of the lemma. 

We may assume that $S^{(2)}_{q}(\m,\n)\ne 0$, else the result is trivial.
Then Lemma~\ref{lem:L2} implies $p\mid G(\m,\n)$ for all $p\mid q$.
Thus $S^{(2)}_{q}(\m,\n)=S_{q}(\m,\n)$,
by multiplying \eqref{eq:fount} over the prime divisors of $q$.
Yet $\abs{S_{q}(\m,\n)}\le q^4$, by Lemma \ref{lem:beach}.
\end{proof}

\begin{lemma}\label{lem:L4}
Let $p$ be a prime and suppose that $r\geq 2+v_p(G(\m,\n))$. Then 
$$S^{(2)}_{p^r}(\m,\n)=0.$$
\end{lemma}

\begin{proof}
We may  assume that $p\mid G(\m,\n)$, since Lemma~\ref{lem:L2} covers the complementary case.
In particular, $S^{(2)}_{p^r}(\m,\n) = S_{p^r}(\m,\n)$, by \eqref{eq:fount}.
The statement of the lemma therefore follows from
% combining  Lemmas \ref{lem:primepower} and  \ref{lem:Njp}.
the second part of Lemma~\ref{lem:square-full}.
\end{proof}

Our final  result produces a useful upper bound for 
$S^{(2)}_{q}(\m,\n)$.

\begin{lemma}\label{lem:L5}
% Let $p$ be a prime such that 
% $v_p(G(\m,\n))\leq 1$ and let $r\geq 2$. Then 
% $$
% S^{(2)}_{p^r}(\m,\n)\ll p^{\frac{7r}{2}}.
% $$
Let $p$ be a prime and let 
$r=1+v_p(G(\m,\n))$. Then 
$$
\abs{S^{(2)}_{p^r}(\m,\n)}\le p^{4r-1}.
$$
\end{lemma}

\begin{proof}
% We may also assume $v_p(G(\m,\n))=1$, since Lemma~\ref{lem:L2} covers the 
% case that $p\nmid G(\m,\n)$.
% In particular, 
% Lemma \ref{lem:L4}  yields
% $S^{(2)}_{p^r}(\m,\n)=0$ if $r\geq 2+v_p(G(\m,\n)) =3$. Hence we may proceed under the assumption that $r=2$.
% If $p\leq 3$ then the statement of the lemma is  obvious and so we may assume  that $p>3$. 

% We must have $p\nmid \m$ 
% and $p\nmid (m_i,m_j,n_k)$ for any permutation 
% $\{i,j,k\}=\{1,2,3\}$, 
% since  $v_p(G(\m,\n))=1$. It follows from 
% \eqref{eq:fount} that 
% $$
%  S^{(2)}_{p^2}(\m,\n)=
%  S_{p^2}(\m,\n)
% = \frac{p^{6}}{1-\frac{1}{p}}\left(N_0(p^2)-p^{-1}N_1(p^2)\right),
% $$
% since $p\mid G(\m,\n).$ 
% Moreover, an application of Lemma \ref{lem:Njp} reveals that $N_0(p^2)=0$ and $N_1(p^2)\leq p^2-p$.  Hence  $|S^{(2)}_{p^2}(\m,\n)|\leq  p^{7}$. This completes the proof. 

We may also assume $v_p(G(\m,\n))\ge 1$, since Lemma~\ref{lem:L2} covers the 
case that $p\nmid G(\m,\n)$.
It follows from 
\eqref{eq:fount} that 
$$
S^{(2)}_{p^r}(\m,\n)
= S_{p^r}(\m,\n)
= \frac{p^{4r}}{\phi(p^r)}\left(N_0(p^r)-p^{-1}N_1(p^r)\right),
$$
since $p\mid G(\m,\n).$ 
Moreover, an application of Lemma \ref{lem:Njp} reveals that $N_0(p^r)=0$ and $N_1(p^r)\leq \phi(p^r)$.  Hence  $|S^{(2)}_{p^r}(\m,\n)|\leq p^{4r-1}$. This completes the proof. 
%Supposing without loss of generality that $p\nmid m_3$ we can eliminate $a$ and $y_3$ from $N_j(p^2)$ to obtain $N_j(p^2)=\phi(p^2)L_j$, where
%$$
%L_j=\#\left\{(y_1,y_2)\in (\ZZ/p^2\ZZ)^2: 
%\begin{array}{l}
%n_1y_1+n_2y_2+n_3\equiv 0 \bmod{p^{2-j}}\\
%m_3y_i^2\equiv m_i\bmod{p^2} \text{ for $i=1,2$}
%\end{array}
%\right\}.
%$$
%If $j=0$ then we easily conclude that $L_0=0$. Indeed, if $L_0\geq 1$ then  $p^2\mid G(\m,\n)$, which is impossible. Hence 
%$S^{(2)}_{p^2}(\m,\n)=-p^7L_1(p^2)$ and it remains to prove that $L_1=O(1)$.
%
%If $p\mid (m_1,m_2)$ then $L_1=0$, since we cannot have 
%$p\mid n_3$. Thus we may assume that $p\nmid (m_1,m_2)$. Suppose 
%without loss of generality that 
%$p\nmid m_2$. 
%Then there are at most $2$ choices for $y_2$ such that 
%$m_3y_2^2\equiv m_2\bmod{p^2}$. If also $p\nmid m_1$ then there are at most $2$ choices for $y_1$ such that 
%$m_3y_1^2\equiv m_1\bmod{p^2}$
%and we conclude that $L_1\leq 4$.  If $p\| m_1$ then clearly $L_1=0$. It remains to deal with the possibility that $p^2\mid m_1$, in which case 
%$$
%L_1=p~
%\#\left\{y_2\in \ZZ/p^2\ZZ: 
%\begin{array}{l}
%n_2y_2+n_3\equiv 0 \bmod{p}\\
%m_3y_2^2\equiv m_2\bmod{p^2}
%\end{array}
%\right\}.
%$$
%If $L_1\geq 1$ then one deduces that $m_2n_2^2\equiv m_3n_3^2\bmod{p}$, whence
%$p^2\mid (m_2n_2^2- m_3n_3^2)^2$. But this implies that 
%\begin{align*}
%G(\m,\n)
%&\equiv m_2^2n_2^4+m_3^2n_3^4-2m_2m_3n_2^2n_3^2 \bmod{p^2}\\
%&\equiv  (m_2n_2^2- m_3n_3^2)^2 \bmod{p^2}\\
%&\equiv 0 \bmod{p^2},
%\end{align*}
%since also $p^2\mid m_1$. This contradics our assumption 
%$v_p(G(\m,\n))= 1$ and so we conclude that $L_1=0$.
\end{proof}

\section{The main term}

We return to 
the main term 
$$
M(B)=
\sum_{q=1}^{\infty}\frac{1}{q^6} S_q(\0,\0)I_q(\0,\0)
$$
that was defined in \eqref{eq:main}.
The properties of the $h$-function 
that were recorded at the start of Section
\ref{sec:initial} ensure that only $q\ll Q$ contribute to this sum.
Let 
$$
\Sigma(x)=
\sum_{q\leq x} q^{-6} S_q(\0,\0).
$$
The following result is concerned with an asymptotic formula for this quantity.

\begin{lemma}\label{lem:Sigma-x}
There exists $b\in \RR$ and $\delta>0$ such that 
$$
\Sigma(x)=\frac{1}{2\zeta(3)}\log x  +b +O(x^{-\delta}).
$$
\end{lemma}

\begin{proof}
We begin by analysing the Dirichlet series
$$
F(s)=\sum_{q=1}^\infty q^{-s}S_q(\0,\0)
$$
for $s\in \CC$.  
Since $|S_q(\0,\0)|\leq q^7$, the series $F(s)$ is absolutely 
convergent for $\Re(s) >8$.
In view of 
Corollary \ref{cor:basic-00}
and 
the multiplicativity property
\eqref{eq:multiply}, we have 
\begin{align*}
F(s)
&=\prod_p
\left(
 1+
\left(1-\frac{1}{p}\right)\sum_{r\geq 1}
p^{4r+3\lfloor r/2 \rfloor-rs}\right),
\end{align*}
for $\Re(s)>8$.  Clearly 
$$
\sum_{r\geq 1}
p^{4r+3\lfloor r/2 \rfloor-rs}=
\sum_{r'\geq 1}
p^{11r'-2r's}+
p^{4-s}\sum_{r'\geq 0}
p^{11r'-2r's}.
$$
Since $\Re(s)>8>6$, we may execute the geometric series to conclude that 
$$
\sum_{r\geq 1}
p^{4r+3\lfloor r/2 \rfloor-rs}=
\left(\frac{1}{p^{s-4}}+\frac{1}{p^{2s-11}}\right)\left(1-\frac{1}{p^{2s-11}}\right)^{-1}.
%\frac{1}{p^{s-4}}\left(1+\frac{1}{p^{s-7}}\right)\left(1-\frac{1}{p^{2s-11}}\right)^{-1}.
$$
But then 
$
F(s)
=
\zeta(2s-11)D(s),
$
for $\Re(s)>8$,
where
$$
D(s)=
\prod_p
\left(
1-\frac{1}{p^{2s-10}}
+
\frac{1}{p^{s-4}}-\frac{1}{p^{s-3}}\right).
$$
Clearly this expression gives a meromorphic continuation 
of $F(s)$ 
to the half-plane $\Re(s)>5.5$, with a simple pole at $s=6$.

By Perron's formula, for $x$ not an integer, we get
$$
\Sigma(x)=\frac{1}{2\pi i} \int_{2+\ve-iT}^{2+\ve+iT} \frac{\zeta(2s+1)D(s+6) x^s }{s}\mathrm d s +O\left(\frac{x^{2+\ve}}{T}\right),
$$
for any  $\ve>0$.  We can shift the contour to $\Re(s)=-1/4$, encountering a double pole at $s=0$. Taking $T$ sufficiently large, it is now straightforward to deduce that 
$$
\Sigma(x)=\frac{D(6)}{2}\log x+b +O(x^{-\delta}),
$$
for some $\delta>0$ and a suitable constant $b.$
This completes the proof of the lemma.
\end{proof}

The (weighted) real density for our problem is defined in \eqref{eq:sigma-inf}.
We may now record the following result, which completes our treatment of the main term.

\begin{proposition}
\label{PROP:main}
There exists a constant $b^*\in \RR$ and 
$\delta>0$ such that
$$
M(B)=\frac{3\sigma_ \infty}{4\zeta(3)}
\log B +b^* +O(B^{-\delta}).
$$
\end{proposition}

\begin{proof}
We shall mimic an argument of Heath-Brown \cite[\S 13]{HB}.
Let $\rho\leq 1$ be a parameter at our disposal. Then,  for each $q\leq \rho Q$, 
it follows from \cite[Lemma 13]{HB} that 
$
I_q(\0)=\sigma_\infty +O_N(\rho^N),
$
for any $N>0$. 
Hence Lemma~\ref{lem:Sigma-x} implies that 
\begin{align*}
\sum_{q\leq \rho Q} q^{-6} S_q(\0,\0)I_q(\0,\0)
&=
\sigma_\infty  \Sigma(\rho Q)
+O_N(\rho^{N}Q^2)\\
&=
\sigma_\infty  
\left\{\frac{1}{2\zeta(3)}\log (\rho Q)+b\right\} +O((\rho Q)^{-\delta})
+O_N(\rho^{N}Q^2).
\end{align*}
Next, it follows from partial summation  and Lemma \ref{lem:Sigma-x} that 
$$
\sum_{q>\rho Q} q^{-6} S_q(\0,\0)I_q(\0,\0)=\frac{1}{2\zeta(3)}\int_{\rho Q}^\infty q^{-1} I_q(\0,\0)\mathrm d q +O((\rho Q)^{-\delta}).
$$
Hence 
$$
M(B)=\frac{\sigma_ \infty}{2\zeta(3)}  \log Q +b\sigma_\infty +\frac{1}{2\zeta(3)}K(\rho)+ 
O((\rho Q)^{-\delta})+O_N(\rho^{N}Q^2),
$$
where
$$
K(\rho)=\sigma_\infty  \log \rho+\int_{\rho Q}^\infty q^{-1} I_q(\0,\0)\mathrm d q.
$$
Arguing as in Heath-Brown \cite[\S 13]{HB}, one finds that 
there exists a constant $K$ such that 
$K(\rho)=K+O(\rho^N)$, for any $N>0$. We finally conclude the proof of the lemma on taking $\rho=B^{-\ve}$ and recalling that $Q=B^{3/2}$.
\end{proof}

\section{The treatment of \texorpdfstring{$E_1(B)$}{E1(B)}}

When bounding the quantity $E_1(B)$ from \eqref{eq:E1B},
the simplest approach in our setting loses a key factor of $\log{B}$.
For $D(\m,\n)\neq 0$, 
Lemma \ref{lem:Sp2} implies that 
we have $S_p(\m,\n)=p^4+O(p^3)$ for $p\mid D(\m,\n)$, if $p\nmid n_1n_2n_3$.
Morally speaking, we shall need to tackle the sum
$$
\sum_{\substack{|\m|,|\n|\ll B^{1/2}\\D(\m, \n)\neq 0}}
\sum_{q|D(\m,\n)}\frac{1}{q^6} S_q(\m,\n)I_q(\m,\n).
$$
For $q$ of generic size $q\sim B^{3/2}$, we can ignore  the oscillation in $I_q(\m,\n)$, and then the sum is basically given by
$$
\sum_{\substack{|\m|,|\n|\ll B^{1/2}\\D(\m, \n)\neq 0}}
\sum_{\substack{q\sim B^{3/2}\\q\mid D(\m,\n)}}\frac{1}{q^2}
\ll \frac{1}{B^3}\sum_{\substack{|\m|,|\n|\ll B^{1/2}\\D(\m, \n)\neq 0}} \tau(D(\m,\n))
\ll \log B.
$$
Another difficulty we have to confront is that
a large power of $\log{Q}$
could in principle arise from sums of the shape
$$
\sum_{\substack{\m,\n\\ D(\m,\n)\ne 0}}
\sum_{\substack{q\sim Q \\ q\mid D(\m,\n)^k}} 1,
$$
where $k\ge 2$.
As a simpler toy problem, we might heuristically have
$$
\sum_{1\le n\le Q}
\sum_{\substack{q\sim Q \\ q\mid n^k}} 1
\gg_k Q (\log{Q})^{k-1}
$$
because $n$ should have roughly $1$ divisor on average in each interval $[e^j,e^{j+1})$,
with $0\le j\le \log{n}$.

In our analysis of $E_1(B)$, as 
defined in \eqref{eq:E1B}, 
we shall address the first issue using Hooley's $\Delta$-function in Section~\ref{s:hooley}.
To address all remaining issues, including the second mentioned above,
we use Lemma~\ref{int-estimate} and the following five facts,
which follow from
Lemmas~\ref{lem:L2},
\ref{lem:L4},
\ref{lem:L5},
\ref{lem:L1},
and~\ref{lem:L3},
respectively:
\begin{enumerate}
\item We have $S^{(2)}_q(\m,\n) = 0$ unless
$\kappa(q)\mid G(\m,\n)$. 

\item We have
$$
S^{(2)}_q(\m,\n) \ne 0 \Rightarrow q\mid G(\m,\n)\kappa(G(\m,\n)).
$$

\item
If the condition
$$
p\mid q\Rightarrow v_p(q) > v_p(G(\m,\n))
$$
holds,
then
$$
S^{(2)}_q(\m,\n) \ll q^4/\kappa(q).
$$

\item We always have
$$
S^{(2)}_q(\m,\n) \ll q^{4+\eps} \sqrt{\{q,m_1\}\{q,m_2\}\{q,m_3\}},
$$
where $\{q,m\}$ is defined in \eqref{eq:sq-gcd}.

\item Suppose that
$
\gcd(q,G(\m,\n),\nabla{G}(\m,\n)) = 1.
$
Then
$$
S^{(2)}_q(\m,\n) \ll q^4.
$$
\end{enumerate}

Let $Q=B^{3/2}$.  Using  (1)--(5),
we proceed to study the quantity $E_1(B)$ defined in 
\eqref{eq:E1B}. It follows 
from the decomposition in \eqref{eq:q1q2} that
\begin{equation}
\label{define-E1B}
E_1(B) 
= \sum_{m_1m_2m_3 D(\m,\n)\ne 0}
\sum_{q_1,q_0\ge 1}
\frac{S^{(1)}_{q_1}(\m,\n)}{q_1^3} \frac{S^{(2)}_{q_0}(\m,\n)}{q_0^6}
 \frac{I_{q_1q_0}(\m,\n)}{q_1^3}.
\end{equation}
Let $A_1>0$ be a large real constant to be specified later.
By partial summation over $q_1$,
using Lemma~\ref{int-estimate},
we find that the contribution 
$E_{1,2}=E_{1,2}(B,Q_1,Q_0)$
to $E_1(B)$ from $q_1\sim Q_1$ and $q_0\sim Q_0$ is
\begin{equation}
\label{define-and-partial-sum-E2}
E_{1,2}
\ll \sum_{\substack{D(\m,\n)\ne 0 \\ m_1m_2m_3\ne 0}}
\frac{J_1}{Q_1^3J_2}
\abs{\sum_{q_1\in \mathcal{I}_1}
\frac{S^{(1)}_{q_1}(\m,\n)}{q_1^3}}
\sum_{q_0\sim Q_0}
\frac{\abs{S^{(2)}_{q_0}(\m,\n)}}{q_0^6},
\end{equation}
for some interval $\mathcal{I}_1=\mathcal{I}_1(B,Q_1,Q_0)\belongs \{q_1\sim Q_1\}$ independent of $(\m,\n)$, where
\begin{equation}\label{eq:def-J1J2}
\begin{split}
J_1&\defeq \frac{(1+B\max\{|\m|,|\n|\}/Q')^{-2}}
{(1+\max\{|\m|,|\n|\}/B^{1/2})^{A_1}}
\le \frac{(B\max\{|\m|,|\n|\}/Q')^{-2}}
{(1+\max\{|\m|,|\n|\}/B^{1/2})^{A_1}}, \\
\quad J_2&\defeq (1+B\,|\hat{D}(\m,\n)|\max\{|\m|,|\n|\}/Q')^{A_1}
\end{split}
\end{equation}
are defined in terms of the convenient quantity $Q'\defeq Q_1Q_0$.
Since $I_q(\m,\n)$ is supported on a range of the form $q\ll Q$,
we may assume that $Q'\ll Q=B^{3/2}$.

Next, write $q_0=q_2q_3q_4$,
where $q_2\mid G(\m,\n)$ with
$
\gcd(q_2,\nabla{G}(\m,\n)) = 1,
$
where $q_3\mid G(\m,\n)$ with
$
\kappa(q_3)\mid \nabla{G}(\m,\n),
$
and where
$$
p\mid q_4\Rightarrow v_p(q_4)>v_p(G(\m,\n)).
$$
On observing that $\{q,m_i\}$ is a square-full integer $d_i\mid \gcd(q,m_i)$,
we find that
\begin{equation}
\begin{split}
\label{bound-E2}
E_{1,2}\ll
\sum_{\substack{d_1,d_2,d_3\ge 1 \\ \textnormal{square-full}}}
&\sum_{\substack{D(\m,\n)\ne 0 \\ d_i\mid m_i\ne 0}}
\frac{J_1}{Q_1^3J_2Q_0^6}
\abs{\sum_{q_1\in \mathcal{I}_1}
\frac{S^{(1)}_{q_1}(\m,\n)}{q_1^3}} \\
&\sum_{\substack{q_2q_3q_4\sim Q_0 \\
q_2,q_3\mid G(\m,\n) \\
d_i\mid q_3 \\
\kappa(q_3)\mid \nabla{G}(\m,\n) \\
p\mid q_4\Rightarrow v_p(q_4)=1+v_p(G(\m,\n))\ge 2}}
q_2^4
q_3^{4+\eps} (d_1d_2d_3)^{1/2}
\frac{q_4^4}{\kappa(q_4)}.
\end{split}
\end{equation}
Let $Q_2,Q_3,Q_4\gg 1$ such that  $Q_2Q_3Q_4\asymp Q_0$. 
Let 
$E_{1,3}=E_{1,3}(B,Q_4,\dots,Q_1)$
be 
the contribution
 to the right-hand side from
$q_2\sim Q_2$, $q_3\sim Q_3$ and $q_4\sim Q_4$. Then 
\begin{equation}
\begin{split}
\label{define-and-bound-E3}
E_{1,3}\ll 
\frac{Q_2^4
Q_3^{4+\eps}
Q_4^4}{Q_1^2Q_0^6}
&\sum_{\substack{d_1,d_2,d_3\ge 1 \\ \textnormal{square-full}}}
(d_1d_2d_3)^{1/2} \\
&\sum_{\substack{D(\m,\n)\ne 0 \\ d_i\mid m_i\ne 0}}
\frac{J_1}{J_2}
\abs{\sum_{q_1\in \mathcal{I}_1}
\frac{S^{(1)}_{q_1}(\m,\n)}{Q_1q_1^3}}
\hspace{-0.3cm}
\sum_{\substack{q_j\sim Q_j,\;(2\le j\le 4) \\
q_2,q_3\mid G(\m,\n) \\
d_i\mid q_3 \\
\kappa(q_3)\mid \nabla{G}(\m,\n) \\
p\mid q_4\Rightarrow v_p(q_4)=1+v_p(G(\m,\n))\ge 2}}
\hspace{-0.3cm}
\frac{1}{\kappa(q_4)}.
\end{split}
\end{equation}
In view of the convergence of the series
$\sum_{\textnormal{square-full }d\ge 1} d^{-\frac12-\eps}$,
it will be convenient to use the inequality $d_i\le q_3$ to write
\begin{equation}
\begin{split}
\label{before-main-holder}
E_{1,3}\ll 
\frac{Q_2^4
Q_3^{4+7\eps}
Q_4^4}{Q_1^2Q_0^6}
&\sum_{\substack{d_1,d_2,d_3\ge 1 \\ \textnormal{square-full}}}
(d_1d_2d_3)^{\frac12-2\eps} \\
&\sum_{\substack{D(\m,\n)\ne 0 \\ d_i\mid m_i\ne 0}}
\frac{J_1}{J_2}
\abs{\sum_{q_1\in \mathcal{I}_1}
\frac{S^{(1)}_{q_1}(\m,\n)}{Q_1q_1^3}}
\hspace{-0.3cm}
\sum_{\substack{q_j\sim Q_j,\;(2\le j\le 4) \\
q_2,q_3\mid G(\m,\n) \\
d_i\mid q_3 \\
\kappa(q_3)\mid \nabla{G}(\m,\n) \\
p\mid q_4\Rightarrow v_p(q_4)=1+v_p(G(\m,\n))\ge 2}}
\hspace{-0.3cm}
\frac{1}{\kappa(q_4)}.
\end{split}
\end{equation}
At this point, the extreme case $Q'\asymp Q_1$ is a real analysis and large sieve problem,
the case $Q'\asymp Q_2$ is a Hooley $\Delta$-function problem,
the case $Q'\asymp Q_3$ is an Ekedahl sieve problem,
and the case $Q'\asymp Q_4$ is a divisor function problem.
In general, we combine these ingredients in a suitable way,
based on the relative sizes of $Q_1,Q_2,Q_3,Q_4$.

We can handle the aspects $\mathcal{I}_1$, $Q_3$, and $Q_4$
by taking large moments
in H\"{o}lder's inequality over $(\m,\n)$,
which will have the advantage of  separating the variables $q_1,q_2,q_3,q_4$.
Recall the definition \eqref{eq:def-J1J2} of $J_1$ and $J_2$ and fix $\mathbf{d}=(d_1,d_2,d_3)\in \NN^3$.
For any $\delta\geq 0$, define 
$$
\Sigma_1^\delta (\mathbf{d})=
\sum_{\substack{D(\m,\n)\ne 0 \\ d_i\mid m_i\ne 0}}
\left(\sum_{\substack{q_2\sim Q_2 \\ q_2\mid G(\m,\n)}} 1\right)^{1+\delta}
J_1.
$$
Next, for any $A\geq 0$, we let
\begin{align}
\nonumber
\Sigma_2^A (\mathbf{d})&=
\sum_{\substack{D(\m,\n)\ne 0 \\ d_i\mid m_i\ne 0}}
\frac{J_1}{J_2^A},\\
\label{eq:def_SIG3}
\Sigma_3^A (\mathbf{d})&=
\sum_{\substack{D(\m,\n)\ne 0 \\ d_i\mid m_i\ne 0}}
\abs{\frac{1}{Q_1}\sum_{q_1\in \mathcal{I}_1} q_1^{-3}S^{(1)}_{q_1}(\m,\n)}^A
J_1,\\
\label{eq:def_SIG4}
\Sigma_4^A (\mathbf{d})&=
\sum_{\substack{D(\m,\n)\ne 0 \\ d_i\mid m_i\ne 0}}
\left(\sum_{\substack{q_3\sim Q_3 \\ q_3\mid G(\m,\n) \\ \kappa(q_3)\mid \nabla{G}(\m,\n)}} 1\right)^A
J_1,
\end{align}
and 
$$
\Sigma_5^A (\mathbf{d})=
\sum_{\substack{D(\m,\n)\ne 0 \\ d_i\mid m_i\ne 0}}
\left(\sum_{\substack{q_4\sim Q_4 \\ p\mid q_4\Rightarrow v_p(q_4)=1+v_p(G(\m,\n))\ge 2}} \frac{1}{\kappa(q_4)}\right)^A
J_1
$$
With this notation it now follows that the sum over $\m,\n$ in \eqref{before-main-holder} is
\begin{equation}
\label{main-holder-bound}
\le \Sigma_1^\delta(\mathbf{d})^{1/(1+\delta)} \prod_{2\leq i\leq 5} 
\Sigma_i^{4A}(\mathbf{d})^{1/(4A)},
\end{equation}
provided that $\delta,A>0$ are reals with $1 = \frac{1}{1+\delta} + \frac{1}{A}$.
This leads us
prove the following five lemmas,
where $\log_+{x}\defeq \max\{1,\log{x}\}$.

\begin{lemma}
\label{Q2-aspect}
Let  $\delta> 0$ and assume that $A_1>4$. Then
$$
\Sigma_1^\delta(\mathbf{d})
\ll_{\delta,A_1} \frac{B^3 (\log_+{B})^{3\delta}}
{(d_1d_2d_3)^{1-\eps} (B^{3/2}/Q')^2}.
$$
\end{lemma}

\begin{proof}
Putting $u=\log Q_2$, we see that 
$$
\sum_{\substack{q_2\sim Q_2 \\ q_2\mid G(\m,\n)}} 1\leq 
\sum_{\substack{e^u<q_2\leq e^{1+u} \\ q_2\mid G(\m,\n)}} 1
\leq \Delta(G(\m,\n)),
$$
where $\Delta(n)$ is Hooley's $\Delta$-function, as defined in 
\eqref{eq:Hooley}.
Breaking the sum over $\m,\n$ into dyadic intervals, we find that 
\begin{equation}
\label{dyadic-T-for-Q2}
\Sigma_1^\delta(\mathbf{d})
\ll \sum_{T} 
 \frac{(BT/Q')^{-2}}
{(1+T/B^{1/2})^{A_1}} S_T,
\end{equation}
where
 $S_T$ is defined in 
\eqref{eq:definition_of_ST}.
It now follows from Lemma~\ref{Tim-S_T-Hooley-delta-bound} that 
$$S_T
= \sum_{\substack{D(\m,\n)\ne 0 \\ d_i\mid m_i\ne 0\\ 
\max\{|\m|,|\n|\}\leq  T
}} 
\Delta(G(\m,\n))^{1+\delta}
\ll_\delta \frac{T^6(\log{T})^{3\delta}}{(d_1d_2d_3)^{1-\eps}}$$
if $d_1d_2d_3\le T^{1/3}$, say.
On the other hand, the same bound on $S_T$
holds trivially by the divisor bound if $d_1d_2d_3>T^{1/3}$.
Plugging this bound into \eqref{dyadic-T-for-Q2}, we get
$$
\Sigma_1^\delta(\mathbf{d})
\ll_\delta (d_1d_2d_3)^{\eps-1} \sum_T
\frac{T^6(\log{T})^{3\delta}}{(BT/Q')^2 (1+T/B^{1/2})^{A_1}}
\ll_{\delta,A_1} (d_1d_2d_3)^{\eps-1}
\frac{B^3(\log_+{B})^{3\delta}}{(B^{3/2}/Q')^2},$$
upon summing separately over the ranges $T\le B^{1/2}$
and $T\ge B^{1/2}$, provided $A_1>4$.
\end{proof}

\begin{lemma}
\label{J2-Q'-aspect}
Let  $A>0$ and assume that $A_1>6$. Then
$$
\Sigma_2^A(\mathbf{d})
\ll_{A,A_1} \frac{B^3}
{d_1d_2d_3 (B^{3/2}/Q')^{2+\min(A_1A,1/5)}}.
$$
\end{lemma}

\begin{proof}
Let $L\leq 1\leq T$ and suppose that 
 $\max\{\norm{\m},\norm{\n}\}\asymp T$ and  $0\ne |\hat D(\m,\n)|\asymp L$.
 Then 
 $J_2\asymp (1+BLT/Q')^{A_1}$ by \eqref{eq:def-J1J2},
and $L\asymp \abs{D(\m,\n)}/T^6\gg 1/T^6$ by the definition \eqref{normalize-D} of $\hat D(\m,\n)$.
Moreover, we have 
\begin{align*}
&\#\{\max\{\norm{\m},\norm{\n}\}\asymp T:
D(\m,\n)\ne 0,
\; d_i\mid m_i\ne 0,
\; \hat D(\m,\n)\ll L\} \\
&\quad \le \#\left\{|\m|,|\n|\ll T:
d_i\mid m_i\ne 0,
\; \min_{\eps_i=\pm 1} \abs{\sum_i \eps_i (m_i/T)^{1/2} (n_i/T)}\ll L^{1/4}\right\},
\end{align*}
by the factorization in \eqref{eq:factorize-dual-form}.
Once $\m,n_1,n_2$ are specified, the number of available choices for $n_3$ in the latter count is $\ll 1 + TL^{1/4}/\abs{m_3/T}^{1/2}$,
so the last display is
$$\ll \frac{T^4}{d_1d_2} \sum_{0<m'_3\ll T/d_3} \left(1 + \frac{TL^{1/4}}{\abs{d_3m'_3/T}^{1/2}}\right)
\ll \frac{T^5}{d_1d_2d_3} + \frac{T^6L^{1/4}}{d_1d_2d_3},$$
where we have written $\abs{m_3} = d_3m'_3$.
Summing dyadically over $T$ and $L$, we get
$$\Sigma_2^A(\mathbf{d})
\ll \sum_{T\ge 1\ge L\gg T^{-6}} \frac{T^6}{d_1d_2d_3}
\frac{T^{-1}+L^{1/4}}{(BT/Q')^2 (1+T/B^{1/2})^{A_1}(1+BLT/Q')^{A_1A}}.$$
Since $1+BLT/Q'\ge 1$ and $L^{1/4}/(BLT/Q')^{\min(A_1A,1/5)}$ is a strictly increasing function of $L$, it follows that 
\begin{align*}
\Sigma_2^A(\mathbf{d})\ll~&
 \sum_{T\ge 1} \frac{T^6}{d_1d_2d_3}
\frac{T^{-1}\log_+{T}}{(BT/Q')^2 (1+T/B^{1/2})^{A_1}}\\
&+ \sum_{T\ge 1} \frac{T^6}{d_1d_2d_3}
\frac{1}{(BT/Q')^{2+\min(A_1A,1/5)} (1+T/B^{1/2})^{A_1}},
\end{align*}
where we have separately bounded the contributions from the two terms in the expression $T^{-1}+L^{1/4}$.
Since $T^a/(BT/Q')^b$ is strictly increasing in $T\le B^{1/2}$
and $(T/B^{1/2})^{-A_1}T^a/(BT/Q')^b$ is strictly decreasing in $T\ge B^{1/2}$
for any fixed exponents $5\le a\le 6$ and $2\le b\le 2+\min(A_1A,1/5)$,
assuming $A_1>6$, we obtain
\begin{align*}
\Sigma_2^A(\mathbf{d})
&\ll \frac{(B^{1/2})^{5+\eps}}{d_1d_2d_3}
\frac{1}{(B^{3/2}/Q')^2}
+ \frac{(B^{1/2})^6}{d_1d_2d_3}
\frac{1}{(B^{3/2}/Q')^{2+\min(A_1A,1/5)}}\\
&\ll \frac{B^3}{d_1d_2d_3 (B^{3/2}/Q')^{2+\min(A_1A,1/5)}},
\end{align*}
where the final inequality holds because
$(B^{3/2}/Q')^{1/5}
\ll (B^{3/2})^{1/5}
\ll (B^{1/2})^{1-\eps}$.
\end{proof}

\begin{lemma}
\label{Q1-aspect}
Let  $A>0$ and assume that $A_1>6$. Then, 
for all $\eps>0$, we have
$$
\Sigma_3^A(\mathbf{d})
\ll_{A,A_1} \frac{B^3 (\log_+{Q_1})^{-1/\eps}}
{(d_1d_2d_3)^{1-\eps} (B^{3/2}/Q')^2}.
$$
\end{lemma}

\begin{proof}
% Recall the character notation introduced before Lemma~\ref{cheap-large-sieve}.
We will use Lemmas~\ref{twisted-PNT} and~\ref{cheap-large-sieve}.
Letting $r_5$ be the maximal square-full divisor of $q$
and letting $r_0 = \gcd(q/r_5, G(\m,\n))$,
 the definition \eqref{eq:S1q} of $S^{(1)}_q(\m,\n)$ implies that
\begin{equation*}
 \frac{S^{(1)}_q(\m,\n)}{q^3}
= \sum_{\substack{r_5r_0r_1r_2r_3r_4=q \\
\gcd(r_i,r_j)=1,\; (0\le i<j\le 5) \\
p\mid r_5\Rightarrow p^2\mid r_5 \\
r_0=\kappa(r_0)\mid G(\m,\n) \\
\gcd(r_1r_2r_3r_4,G(\m,\n)) = 1}} \mu(r_4)
\prod_{1\le i<j\le 3} \mu(r_k) (\frac{m_im_j}{r_k}),
\end{equation*}
where $k=6-i-j$ is such that $\{i,j,k\}=\{1,2,3\}$.
Let $H(\m,\n)\defeq m_1m_2m_3G(\m,\n)$ and $$
G_t(\m,\n)\defeq 
\begin{cases}
G(\m,\n) & \text{ if $1\leq t \leq 4$,}\\
1 & \text{ if $t\in \{0,5\}$.}
\end{cases}
$$
Suppose $r_t\sim R_t$ for $0\le t\le 5$,
such that 
\begin{equation}\label{eq:R0..5}
R_5R_0R_1R_2R_3R_4\asymp Q_1.
\end{equation}
Freezing all but one variable $r_t$,
it follows from
the triangle inequality
that
\begin{equation}
\label{Q1-expand-triangle-ineq}
\sum_{q_1\in \mathcal{I}_1} \frac{S^{(1)}_{q_1}(\m,\n)}{q_1^{3}}
\ll \min_{0\le t\le 5}
\sum_{q\asymp Q_1/R_t}
\tau(q)^5
\abs{S_t},
\end{equation}
where $S_t$ denotes the quantity
\begin{equation*}
\sum_{\substack{r_t\sim R_t\\
r_tq\in \mathcal{I}_1\\
\gcd(r_t,qG_t(\m,\n))=1}}
\hspace{-0.8cm}
\left(\1_{\kappa(r_t)^2\mid r_t} \1_{t=5}
+ \1_{r_t\mid G(\m,\n)} \1_{t=0}
+ \mu(r_t) \1_{t=4}
+ \mu(r_t) \left(\frac{m_im_j}{r_t}\right) \1_{t=k\in \{1,2,3\}}\right).
\end{equation*}
% Here, we set $\chi_{[4]} = \chi_1 = 1$
% and $\chi_{[k]} = \chi_{m_im_j}$, for $1\le i<j\le 3$.

It follows from \eqref{eq:SUP} that
\begin{equation}\label{eq:SUP'}
\sum_{q_1\in \mathcal{I}_1} q_1^{-3}S^{(1)}_{q_1}(\m,\n) \ll \sum_{q_1\sim Q_1} \tau(q_1)^2 \ll Q_1(1+\log{Q_1})^3.
\end{equation}
By \eqref{eq:SUP'} and \eqref{Q1-expand-triangle-ineq},
we have
\begin{equation}
\label{sunny}
\sum_{\substack{|\m|,|\n|\ll T \\ H(\m,\n)\ne 0}}
\abs{\sum_{q_1\in \mathcal{I}_1} q_1^{-3}S^{(1)}_{q_1}(\m,\n)}^A
\ll (Q_1(1+\log{Q_1})^3)^{A-1}
\sum_{R_i} U(\mathbf{R}),
\end{equation}
where
\begin{equation*}
U(\mathbf{R})
= \sum_{q\asymp Q_1/R_t}
\sum_{\substack{|\m|,|\n|\ll T \\ H(\m,\n)\ne 0}}
\tau(q)^5
\abs{S_t},
\end{equation*}
where $0\le t\le 5$ is chosen
in terms of $\mathbf{R}$
in such a way that $R_t = \max(R_0,\dots,R_5)$.
Then $R_t\gg Q_1^{1/6}$ by \eqref{eq:R0..5}.
The number of available choices for $\mathbf{R}$
is $\ll (1+\log{Q_1})^5$.

% Under RH over $r_4$,
% and GRH over $r_1,r_2,r_3$,
% it should be routine to bound $S_t$ for $1\le t\le 4$,
% at least when $Q_1$ exceeds an arbitrarily small power of $B$.
% Unconditionally, we can use a large sieve over $r_1,r_2,r_3$
% and the prime number theorem over $r_4$.

Clearly $S_5\ll R_5^{1/2}$,
and if $G(\m,\n)\ne 0$ then $S_0\ll |G(\m,\n)|^\eps$.
Thus,
\begin{equation*}
U(\mathbf{R}) \1_{t\in \{0,5\}}
\ll \frac{Q_1^{1+\eps}T^6}{R_5^{1/2}} \1_{t=5}
+ \frac{Q_1^{1+\eps}T^{6+\eps}}{R_0} \1_{t=0}
\ll \frac{Q_1^{1+\eps}T^{6+\eps}}{Q_1^{1/12}}.
\end{equation*}
By Lemma~\ref{twisted-PNT} with $P = |G(\m,\n)|$ and $z=q$, we have
\begin{equation*}
\frac{U(\mathbf{R}) \1_{t=4}}{T^6}
\ll \frac{(1+\log{Q_1})^{2^5-1}}{(1+\log{R_4})^B} Q_1
+ (TQ_1)^\eps \frac{Q_1}{R_4^{1/2}}
\ll \frac{Q_1}{(1+\log{Q_1})^{B-31}}
+ \frac{T^\eps Q_1^{1+\eps}}{Q_1^{1/12}},
\end{equation*}
for any $B\geq 1$.
Finally, if $t=k\in \{1,2,3\}$,
then by Lemma~\ref{cheap-large-sieve}
with $(h,m) = (m_i,m_j)$, we have 
\begin{equation*}
U(\mathbf{R})
\ll T^4 (Q_1/R_t)^{1+\eps}
(TR_t)^\eps (T^2 R_t^{1/2} + T R_t)
\ll T^{6+\eps} Q_1^{1+\eps} (Q_1^{-1/12} + T^{-1}).
\end{equation*}

Choosing $\eps$ to be small in terms of $A$,
and letting $B = 3/\eps$, it follows from  \eqref{sunny} that 
\begin{equation}\label{eq:**}
\frac{\sum_{|\m|,|\n|\ll T}
\abs{\frac{1}{Q_1}\sum_{q_1\in \mathcal{I}_1} q_1^{-3}S^{(1)}_{q_1}(\m,\n)}^A}
{T^6 }
\ll \frac{1}{(\log{Q_1})^{2/\eps}}
+ \frac{T^\eps}{Q_1^{1/12-\eps}}
+ \frac{Q_1^\eps}{T^{1-\eps}},
\end{equation}
where we have included the terms $H(\m,\n)=0$
in the summation over $|\m|,|\n|\ll T$,
using the bound \eqref{eq:SUP'}. Let $C(A,\ve)$ be the estimate
$$
\frac{\sum_{|\m|,|\n|\ll T}
\abs{\frac{1}{Q_1}\sum_{q_1\in \mathcal{I}_1} q_1^{-3}S^{(1)}_{q_1}(\m,\n)}^A}
{T^6 }
\ll \frac{1}{(\log{Q_1})^{1/\eps}}+ \frac{Q_1^{2\eps}}{T}.
$$
Clearly $C(A,\ve)$ follows from \eqref{eq:**} when  $Q_1^{A+1} \geq T$, and so we proceed under the assumption that 
$Q_1^{A+1} \le T$. 
We may assume without loss of generality that 
 $A$ be an even integer. 
Let 
 $W\ge 0$ is a smooth function supported on a ball in $\RR^6$,
and let $K\ge 0$ be an arbitrarily large constant.
Writing 
$$
f(\m,\n)=\abs{\frac{1}{Q_1}\sum_{q_1\in \mathcal{I}_1} q_1^{-3}S^{(1)}_{q_1}(\m,\n)}^A,
$$
for notational convenience, it  follows from Poisson summation that
\begin{align*}
\frac{\sum_{|\m|,|\n|\ll T} f(\m,\n)
}
{T^6 }
&\ll \frac{\sum_{\m,\n} W(\frac{\m,\n}{T}) f(\m,\n)}
{T^6} \\
&= \frac{\sum_{\m,\n} W(\frac{\m,\n}{Q_1^{A+1}}) f(\m,\n)
+ O_K(Q_1^{-K})}
{(Q_1^{A+1})^6 },
\end{align*}
on opening up the sum over $q_1$ in both of the expressions on the right hand side, 
switching the order of summation, 
breaking into residue classes modulo 
an integer of order $Q_1^A$,  and then applying Poisson summation. 
But then 
\begin{align*}
\frac{\sum_{|\m|,|\n|\ll T} f(\m,\n)
}
{T^6}
&\ll \frac{\sum_{|\m|,|\n|\ll Q_1^{A+1}}
\abs{\frac{1}{Q_1}\sum_{q_1\in \mathcal{I}_1} q_1^{-3}S^{(1)}_{q_1}(\m,\n)}^A
+ O_K(Q_1^{-K})}
{(Q_1^{A+1})^6} \\
&\ll \frac{1}{(\log{Q_1})^{1/\eps}},
\end{align*}
which is satisfactory for the claimed estimate $C(A,\ve)$.

%\footnote{We use Poisson summation to evaluate the average
%$\frac{1}{X^6} \sum_{\m,\n} W(\frac{\m,\n}{X})
%|\sum_{q_1\in \mathcal{I}_1} \cdots|^A$
%for $X\in \{T,Q_1^{A+1}\}$,
%getting a main term independent of $X$.
%The error term is $O(Q_1^{-K})$.}

Applying $C(A/\ve,\ve^2)$, 
it now follows from an application of  H\"{o}lder's inequality that 
\begin{equation*}
\begin{split}
\sum_{\substack{|\m|,|\n|\ll T \\ d_i\mid m_i\ne 0}} f(\m,\n)
&\ll \left(\sum_{\substack{|\m|,|\n|\ll T \\ d_i\mid m_i\ne 0}} 1\right)^{1-\eps}
\left(\sum_{|\m|,|\n|\ll T} f(\m,\n)^{1/\ve}
\right)^\eps \\
&\ll \left(\frac{T^6}{d_1d_2d_3}\right)^{1-\eps}
(T^6)^\eps \left(\frac{1}{(\log{Q_1})^{1/\eps^2}}
+ \frac{Q_1^{2\eps^2}}{T}\right)^\eps \\
&\ll \frac{T^6}{(d_1d_2d_3)^{1-\eps}}
\left(\frac{1}{(\log{Q_1})^{1/\eps}}
+ \frac{Q_1^{\eps^2}}{T^\eps}\right).
\end{split}
\end{equation*}

We now recall the definition \eqref{eq:def_SIG3} of $\Sigma_3^A(\mathbf{d})$ and the upper bound 
\eqref{eq:def-J1J2} for $J_1$.
On invoking  dyadic summation over $T$, we finally deduce that 
\begin{align*}
\Sigma_3^A(\mathbf{d})
&\ll \sum_T \frac{T^6 (BT/Q')^{-2}}
{(d_1d_2d_3)^{1-\eps} (1+T/B^{1/2})^{A_1}}
\left(\frac{1}{(\log{Q_1})^{1/\eps}}
+ \frac{Q_1^{\eps^2}}{T^\eps}\right) \\
&\ll \frac{(B^{1/2})^6 (B^{3/2}/Q')^{-2}}
{(d_1d_2d_3)^{1-\eps}}
\left(\frac{1}{(\log{Q_1})^{1/\eps}}
+ \frac{Q_1^{\eps^2}}{(B^{1/2})^\eps}\right),
\end{align*}
since  $A_1>6$. The statement of the lemma now follows, since   $Q_1\ll Q'\ll B^{3/2}$.
\end{proof}

\begin{lemma}
\label{Q3-aspect}
Let  $A\geq 0$ and assume that $A_1>4$. Then, 
for all $\delta\in (0,1/3)$, we have
$$
\Sigma_4^A(\mathbf{d})
\ll \frac{\gcd(d_1,d_2,d_3)^{2\delta} B^3 Q_3^{-\delta/2000}}
{(d_1d_2d_3)^{1-\eps} (B^{3/2}/Q')^2}.
$$
\end{lemma}

\begin{proof}
The proof is modelled after the general Ekedahl sieve,
but also makes use of the particular structure of the variety $G=\nabla{G}=0$.
Using the derivative formulas \eqref{eq:derivatives},
it is easy to check that if $p\mid \nabla{G}(\m,\n)$,
then $p\mid 6n_1n_2n_3\gcd(m_1,m_2,m_3)$.
Therefore,
$$\sum_{\substack{D(\m,\n)\ne 0 \\ |\m|,|\n|\ll T \\ d_i\mid m_i\ne 0}}
\sum_{\substack{q_3\sim Q_3 \\ q_3\mid G(\m,\n) \\ \kappa(q_3)\mid \nabla{G}(\m,\n)}} 1
\le \frac{T^{5+\eps}}{d_1d_2d_3}
+ \sum_{\substack{r,g\ge 1 \\ r=\kappa(r)}}
\sum_{\substack{D(\m,\n)\ne 0 \\ |\m|,|\n|\ll T \\ d_i\mid m_i\ne 0
\\ r\mid 6n_1n_2n_3\ne 0 \\ g = \gcd(\m) \\
\prod_{1\le i<j\le 3} (m_in_i^2-m_jn_j^2) \ne 0}}
\sum_{\substack{q_3\sim Q_3 \\ r\mid q_3\mid G(\m,\n) \\ \kappa(q_3)\mid rg}} 1,
$$
where the first term on the right hand side accounts for the possibility that
$n_1n_2n_3=0$
or $m_in_i^2=m_jn_j^2$ for some $1\le i<j\le 3$.
Now write
$$\sum_{\substack{r,g\ge 1 \\ r=\kappa(r)}}
\sum_{\substack{D(\m,\n)\ne 0 \\ |\m|,|\n|\ll T \\ d_i\mid m_i\ne 0
\\ r\mid 6n_1n_2n_3\ne 0 \\ g = \gcd(\m) \\
\prod_{1\le i<j\le 3} (m_in_i^2-m_jn_j^2) \ne 0}}
\sum_{\substack{q_3\sim Q_3 \\ r\mid q_3\mid G(\m,\n) \\ \kappa(q_3)\mid rg}} 1
\le \Sigma_1 + \Sigma_2 + \Sigma_3,$$
where $\Sigma_1$ denotes the contribution from $r\ge L$,
where $\Sigma_2$ denotes the contribution from $g\ge L$,
and where $\Sigma_3$ denotes the contribution from $r,g\le L$.

%We break up this problem (Problem 8.4) into two parts, one where $q_3$ is close to 100th-power free (say), and one where $q_3$ is divisible by a large 100-power full integer, the latter event being rare by an easy (unconditional) case of the square-free sieve conjecture.
%
%Use divisor bound and Ekedahl sieve

In $\Sigma_1$, observe that
$r\mid \gcd(r,6n_1)\gcd(r,6n_2)\gcd(r,6n_3)$
by reduction modulo $r$, so
$$\max\{\gcd(r,6n_1),\gcd(r,6n_2),\gcd(r,6n_3)\}
\ge r^{1/3}\ge L^{1/3}.$$
Since any prime $p\mid \gcd(G(\m,\n),6n_1)$
divides $6(m_2n_2^2-m_3n_3^2)$, for instance,
we may replace $r$ with $\max\{\gcd(r,6n_1),\gcd(r,6n_2),\gcd(r,6n_3)\}$
and permute indices $\{1,2,3\}$ to assume that
$$\Sigma_1
\ll \sum_{\substack{r\ge L^{1/3} \\ r=\kappa(r)}}
\sum_{\substack{D(\m,\n)\ne 0 \\ |\m|,|\n|\ll T \\ d_i\mid m_i\ne 0
\\ r\mid 6n_1\ne 0 \\
r\mid 6(m_2n_2^2-m_3n_3^2) \ne 0}} T^\eps
\ll \frac{T^{6+2\eps}}{d_1d_2d_3 L^{1/3}},
$$
where we have first summed over $g$ and $q_3$ using the divisor bound,
before summing over $n_1$,
then over $r$ using the divisor bound,
and finally over $\m,n_2,n_3$.

For $\Sigma_2$, we first eliminate $r$ and $q_3$ using the divisor bound, getting
$$\Sigma_2
\ll \sum_{g\ge L}
\sum_{\substack{D(\m,\n)\ne 0 \\ |\m|,|\n|\ll T \\ \lcm(g,d_i)\mid m_i\ne 0}} T^\eps
\ll \sum_{\substack{g\ge L}}
\frac{T^{6+\eps}}{\lcm(g,d_1)\lcm(g,d_2)\lcm(g,d_3)}
= \sum_{\substack{g\ge L}}
\frac{g_1g_2g_3 T^{6+\eps}}{g^3 d_1d_2d_3},
$$
where $g_i = \gcd(g,d_i)$.
But, by Rankin's trick, we have 
\begin{align*}
\sum_{g\ge L}
\frac{g_1g_2g_3}{g^3}
\le \frac{1}{L^\delta} \sum_{g\ge 1}
\frac{g_1g_2g_3}{g^{3-\delta}}
&\ll \frac{1}{L^\delta}
\prod_p \sum_{e\ge 0} \frac{\prod_i \gcd(p^e,d_i)}
{(p^e)^{3-\delta}}.
\end{align*}
Put $a_i=v_p(d_i)$ and relabel so that $a_1\leq a_2\leq a_3$. 
If $a_3=0$ then the local factor is $1+O(p^{-(3-\delta)})$. If $a_3\geq 1$ then 
the local factor is
$$
\sum_{0\leq e\leq a_1} p^{\delta e} +
\sum_{a_1< e\leq a_2} p^{a_1-e(1-\delta)} +
\sum_{a_2< e\leq a_3} p^{a_1+a_2-e(2-\delta)} 
+
\sum_{e>a_3} p^{a_1+a_2+a_3-e(3-\delta)} 
$$
Each of these sums is $O(p^{\delta a_1})$, whence
\begin{align*}
\sum_{g\ge L}
\frac{g_1g_2g_3}{g^3}
&\ll \frac{\gcd(d_1,d_2,d_3)^{2\delta}}{L^\delta}.
\end{align*}
The final product over $p$ is
$\ll \gcd(d_1,d_2,d_3)^{(\delta+\eps)/(1-2\delta)}$, and so it follows that
\begin{equation}
\label{g1g2g3-bound}
\Sigma_2
\ll \frac{T^{6+\eps}\gcd(d_1,d_2,d_3)^{2\delta}}
{(d_1d_2d_3)^{1-\eps}L^\delta}.
\end{equation}

Finally,
$$\Sigma_3
\le \sum_{\substack{r,g\le L \\ r=\kappa(r)}}
\sum_{\substack{|\m|,|\n|\ll T \\ d_i\mid m_i\ne 0
\\ g = \gcd(\m)}}
\sum_{\substack{q_3\sim Q_3 \\ q_3\mid G(\m,\n) \\ \kappa(q_3)\mid rg}} 1
= \sum_{\substack{r,g\le L \\ r=\kappa(r)}}
\sum_{\substack{q_3\sim Q_3 \\ \kappa(q_3)\mid rg}} 
\sum_{\substack{|\m|,|\n|\ll T \\ d_i\mid m_i\ne 0
\\ g = \gcd(\m)
\\ q_3\mid G(\m,\n)}} 1.$$
Assume $L = Q_3^{1/1000}$, say.
Since every prime factor of $q_3$ is $\le L$,
it follows that $q_3$ has an integer factor $f$ with
$Q_3^{1/4}\ll f\ll Q_3^{1/3}$, say.
Given $\m$, the number of available choices for $\n$
is $\ll (T+f)^3 (f/g^2)^{\eps-1/4}$, say,
by the univariate degree $4$ case of
\eqref{eq:regis}.
Thus
\begin{align*}
\Sigma_3
\ll \sum_{\substack{r,g\le L \\ r=\kappa(r)}}
\sum_{\substack{q_3\sim Q_3 \\ \kappa(q_3)\mid rg}} 
\sum_{\substack{\m\ll T \\ d_i\mid m_i\ne 0}}
(Q_3L)^\eps \frac{T^3+Q_3}{Q_3^{1/16}} L^{1/2}
&\ll L^2 (Q_3L)^{2\eps} \frac{T^3}{d_1d_2d_3}
\frac{T^3+Q_3}{Q_3^{1/16}} L^{1/2}\\
&\ll \frac{T^3(T^3+Q_3)}{d_1d_2d_3Q_3^{1/17}},
\end{align*}
where the sum over $q_3$ is bounded using \eqref{eq:kernel-divisor-bound}.

It will be convenient to 
put 
$$
f(\m,\n)=\sum_{\substack{q_3\sim Q_3 \\ q_3\mid G(\m,\n) \\ \kappa(q_3)\mid \nabla{G}(\m,\n)}} 1.
$$
Since $\tau(G(\m,\n)) \ll T^\eps$ whenever $|\m|,|\n|\ll T$ with $D(\m,\n)\ne 0$, it follows that
\begin{align*}
\sum_{\substack{D(\m,\n)\ne 0 \\ |\m|,|\n|\ll T \\ d_i\mid m_i\ne 0}}
f(\m,\n)^A
&\ll T^\eps \left(
\frac{T^{5+\eps}}{d_1d_2d_3}+
\Sigma_1+\Sigma_2+\Sigma_3\right)\\
&\ll 
\frac{T^{6+3\eps}\gcd(d_1,d_2,d_3)^{2\delta}}
{(d_1d_2d_3)^{1-\eps}L^\delta}
+ \frac{T^{3+\eps}Q_3^{16/17}}{d_1d_2d_3},
\end{align*}
since 
$$
\frac{T^{5+\eps}}{d_1d_2d_3}\leq 
\left(
\frac{T^{6+3\eps}\gcd(d_1,d_2,d_3)^{2\delta}}
{(d_1d_2d_3)^{1-\eps}L^\delta}\right)^{2/3}
\left( \frac{T^{3+\eps}Q_3^{16/17}}{d_1d_2d_3}\right)^{1/3}.
$$
However, if $K>0$ is a sufficiently large constant
and $T\ge d_1d_2d_3Q_3^A$, then
\begin{equation}
\label{combo-null}
\sum_{\substack{D(\m,\n)\neq 0\\
|\m|,|\n|\ll T \\ d_i\mid m_i}}f(\m,\n)^A
\ll \left(\frac{T}{d_1d_2d_3Q_3^A}\right)^6
\sum_{\substack{D(\m,\n)\ne 0 \\ |\m|,|\n|\le Kd_1d_2d_3Q_3^A \\ d_i\mid m_i\ne 0}} f(\m,\n)^A.
\end{equation}
Indeed, for any integer $\Pi\asymp d_1d_2d_3Q_3^A$
% Note: The case of interest is $\Pi = d_1d_2d_3 q_3^{(1)}\cdots q_3^{(A)}$.
and any residue class $(\a,\b)\bmod{\Pi}$, 
there exists a pair $(\m,\n)\equiv (\a,\b)\bmod{\Pi}$
with $|\m|,|\n|\le 10\Pi$ such that $m_1m_2m_3D(\m,\n)\ne 0$,
by the combinatorial nullstellensatz, as used  in \cite{BGW2024positive}*{proof of Lemma~7.7}.
The key point is that such a residue class $(\a,\b)\bmod{\Pi}$
is counted $\ll (\frac{T}{\Pi})^6 \asymp (\frac{T}{d_1d_2d_3Q_3^A})^6$ times on the left-hand side of \eqref{combo-null},
and is counted at least once on the right-hand side.

We may now deduce that 
\begin{equation*}
\frac{1}{T^6} \sum_{\substack{D(\m,\n)\ne 0 \\ |\m|,|\n|\ll T \\ d_i\mid m_i\ne 0}}
f(\m,\n)^A
\ll \frac{\gcd(d_1,d_2,d_3)^{2\delta}}
{(d_1d_2d_3)^{1-\eps}L^{\delta/2}}
+ \frac{(T^{-3}+(d_1d_2d_3Q_3^A)^{-3}) Q_3^{17/18}}{(d_1d_2d_3)^{1-\eps}}.
\end{equation*}
This is a trivial consequence of our earlier estimate if $T\le d_1d_2d_3Q_3^A$, and it follows from   \eqref{combo-null} if $T\ge d_1d_2d_3Q_3^A$. But now, since $(Q_3^A)^{-3} Q_3^{17/18} \le 1/L^{\delta/2}$ for $A\ge 1$,
we get
\begin{equation*}
\frac{1}{T^6} \sum_{\substack{D(\m,\n)\ne 0 \\ |\m|,|\n|\ll T \\ d_i\mid m_i\ne 0}}f(\m,\n)^A
\ll \frac{\gcd(d_1,d_2,d_3)^{2\delta}}
{(d_1d_2d_3)^{1-\eps}L^{\delta/2}}
+ \frac{T^{-3} Q_3^{17/18}}{(d_1d_2d_3)^{1-\eps}},
\end{equation*}
if $A\ge 1$. The same bound clearly holds for any  $A>0$.

We now recall the definition \eqref{eq:def_SIG4} of $\Sigma_4^A(\mathbf{d})$ and the upper bound 
\eqref{eq:def-J1J2} for $J_1$. 
Appealing to  dyadic summation over $T$, we obtain
\begin{align*}
\Sigma_4^A(\mathbf{d})
&\ll \sum_T \frac{\gcd(d_1,d_2,d_3)^{2\delta} T^6 (BT/Q')^{-2}}
{(d_1d_2d_3)^{1-\eps} (1+T/B^{1/2})^{A_1}}
\left(\frac{1}{L^{\delta/2}}
+ \frac{Q_3^{17/18}}{T^3}\right) \\
&\ll \frac{\gcd(d_1,d_2,d_3)^{2\delta} B^3 (B^{3/2}/Q')^{-2}}
{(d_1d_2d_3)^{1-\eps}}
\left(\frac{1}{L^{\delta/2}}
+ \frac{Q_3^{17/18}}{B^{3/2}}\right)
\end{align*}
since  $A_1>4$. The statement of the lemma follows, since  $Q_3\ll Q'\ll B^{3/2}$.
\end{proof}

\begin{lemma}
\label{Q4-aspect}
Let  $A\geq 0$ and assume that $A_1>4$. Then, 
for all $\delta\in (0,1/3)$, we have
\begin{equation*}
\Sigma_5^A(\mathbf{d})
\ll \frac{\gcd(d_1,d_2,d_3)^{2\delta} B^3 Q_4^{-\delta/2000}}
{(d_1d_2d_3)^{1-\eps} (B^{3/2}/Q')^2}.
\end{equation*}
\end{lemma}

\begin{proof}
The condition
$p\mid q_4\Rightarrow v_p(q_4)=1+v_p(G(\m,\n))\ge 2$
implies, in particular, that
$\kappa(q_4)\mid G(\m,\n)$.
Since $\#\{q_4\sim Q_4: \kappa(q_4)=r\}\ll (Q_4r)^\eps$ by \eqref{eq:kernel-divisor-bound},
it follows that
\begin{align*}
\Sigma_L\defeq \sum_{\substack{q_4\sim Q_4
\\ \kappa(q_4)\ge L
\\ p\mid q_4\Rightarrow v_p(q_4)=1+v_p(G(\m,\n))\ge 2}} \frac{1}{\kappa(q_4)}
&= \sum_{\substack{r\ge L \\ r=\kappa(r) \\ r\mid G(\m,\n)}}
\sum_{\substack{q_4\sim Q_4
\\ \kappa(q_4)=r
\\ p\mid q_4\Rightarrow v_p(q_4)=1+v_p(G(\m,\n))\ge 2}}
\frac{1}{r}\\
&\ll \sum_{\substack{r\ge L \\ r=\kappa(r) \\ r\mid G(\m,\n)}}
\frac{(Q_4r)^\eps}{r}.
\end{align*}
Assuming that  $|\m|,|\n|\ll T$ and $D(\m,\n)\ne 0$, 
 the divisor bound yields
\begin{equation}
\label{Q4-truncated-divisor-bound}
\Sigma_L \ll \frac{(Q_4T)^\eps}{L}.
\end{equation}
In particular, $\Sigma_1 \ll (Q_4T)^\eps$.

Arguing as in the proof of \eqref{g1g2g3-bound},
we have
\begin{equation}
\label{copy-g1g2g3-argument}
\sum_{\substack{D(\m,\n)\ne 0 \\ |\m|,|\n|\ll T \\ d_i\mid m_i\ne 0
\\ \gcd(\m)\ge L}}
\Sigma_1
\ll \sum_{g\ge L}
\sum_{\substack{D(\m,\n)\ne 0 \\ |\m|,|\n|\ll T \\ \lcm(g,d_i)\mid m_i\ne 0}}
(Q_4T)^\eps
\ll Q_4^\eps \frac{T^{6+\eps}\gcd(d_1,d_2,d_3)^{2\delta}}
{(d_1d_2d_3)^{1-\eps}L^\delta}.
\end{equation}
On the other hand,
$$\sum_{\substack{|\m|,|\n|\ll T \\ d_i\mid m_i\ne 0
\\ \gcd(\m)\le L}}
(\Sigma_1-\Sigma_L)
\le \sum_{\substack{r\le L \\ r=\kappa(r)}}
\sum_{\substack{q_4\sim Q_4
\\ \kappa(q_4)=r}}
\sum_{\substack{|\m|,|\n|\ll T \\ d_i\mid m_i\ne 0
%\\ r\mid G(\m,\n)
\\ p\mid q_4\Rightarrow v_p(q_4)=1+v_p(G(\m,\n))\ge 2}}
\frac{1}{r}.
$$
Assume that $L = Q_4^{1/1000}$.
Since every prime factor of $q_4$ is $\le L$,
it follows that the integer $q_4/r\gg Q_4/L$
has an integer factor $f$ with
$Q_4^{1/4}\ll f\ll Q_4^{1/3}$, say.
Since $q_4/r$ divides $G(\m,\n)$, it follows that
once $\m$ is specified, the number of available choices for $\n$
is $\ll (T+f)^3 (f/\gcd(\m)^2)^{\eps-1/4}$,
by the case $n=1$ and $d=4$ of \eqref{eq:regis}.
Hence
\begin{align*}
\sum_{\substack{|\m|,|\n|\ll T \\ d_i\mid m_i\ne 0
\\ \gcd(\m)\le L}}
(\Sigma_1-\Sigma_L)
&\ll \sum_{\substack{r\le L \\ r=\kappa(r)}}
\sum_{\substack{q_4\sim Q_4
\\ \kappa(q_4)=r}}
\frac{T^3}{d_1d_2d_3}
Q_4^\eps \frac{T^3+Q_4}{r Q_4^{1/16}} L^{1/2}\\
&\ll Q_4^{3\eps} \frac{T^3}{d_1d_2d_3}
\frac{T^3+Q_4}{Q_4^{1/16}} L^{1/2}\\
&\ll \frac{T^3}{d_1d_2d_3} \frac{T^3+Q_4}{Q_4^{1/17}}.
\end{align*}
It follows from this bound,
together with \eqref{Q4-truncated-divisor-bound} and \eqref{copy-g1g2g3-argument},
that
\begin{equation*}
\begin{split}
\sum_{\substack{D(\m,\n)\ne 0 \\ |\m|,|\n|\ll T \\ d_i\mid m_i\ne 0}}
\frac{\Sigma_1^{A-1}}{(Q_4T)^\eps} \Sigma_1
&\ll \frac{T^6}{d_1d_2d_3} \frac{(Q_4T)^\eps}{L}
+ \frac{T^3}{d_1d_2d_3} \frac{T^3+Q_4}{Q_4^{1/17}}
+ Q_4^\eps \frac{T^{6+\eps}\gcd(d_1,d_2,d_3)^{2\delta}}
{(d_1d_2d_3)^{1-\eps}L^\delta} \\
&\ll Q_4^\eps \frac{T^{6+\eps}\gcd(d_1,d_2,d_3)^{2\delta}}
{(d_1d_2d_3)^{1-\eps}L^\delta}
+ \frac{T^3 Q_4^{16/17}}{d_1d_2d_3}.
\end{split}
\end{equation*}

On the other hand,
the summation condition $p\mid q_4\Rightarrow v_p(q_4)=1+v_p(G(\m,\n))\ge 2$ in $\Sigma_1$ depends only on $(\m,\n)\bmod{q_4}$,
just as the conditions $q_3\mid G(\m,\n)$ and $\kappa(q_3)\mid \nabla{G}(\m,\n)$ in \eqref{combo-null} depend only on $(\m,\n)\bmod{q_3}$.
Therefore, on mimicking the proof of \eqref{combo-null},
we find that if $K>0$ is a sufficiently large constant
and $T\ge d_1d_2d_3Q_4^A$, then
\begin{equation*}
\sum_{\substack{|\m|,|\n|\ll T \\ d_i\mid m_i}}
\Sigma_1^A
\ll \left(\frac{T}{d_1d_2d_3Q_4^A}\right)^6
\sum_{\substack{D(\m,\n)\ne 0 \\ |\m|,|\n|\le Kd_1d_2d_3Q_4^A \\ d_i\mid m_i\ne 0}}
\Sigma_1^A.
\end{equation*}
Whether $T\le d_1d_2d_3Q_4^A$ or $T\ge d_1d_2d_3Q_4^A$,
and whether $A\ge 1$ or $A\le 1$,
it follows that
$$\frac{1}{T^6} \sum_{\substack{D(\m,\n)\ne 0 \\ |\m|,|\n|\ll T \\ d_i\mid m_i\ne 0}}
\Sigma_1^A
\ll \frac{\gcd(d_1,d_2,d_3)^{2\delta}}
{(d_1d_2d_3)^{1-\eps}L^{\delta/2}}
+ \frac{T^{-3}Q_4^{17/18}}{(d_1d_2d_3)^{1-\eps}}.$$
The rest of the proof is identical to that of Lemma~\ref{Q3-aspect}.
%By dyadic summation over $T$, we conclude that
%\begin{equation*}
%\begin{split}
%\sum_{\substack{D(\m,\n)\ne 0 \\ d_i\mid m_i\ne 0}}
%\Sigma_1^A
%J_1(Q',\m,\n)
%&\ll \sum_T \frac{\gcd(d_1,d_2,d_3)^{2\delta} T^6 (BT/Q')^{-2}}
%{(d_1d_2d_3)^{1-\eps} (1+T/B^{1/2})^{A_1}}
%(\frac{1}{L^{\delta/2}}
%+ \frac{Q_4^{17/18}}{T^3}) \\
%&\ll \frac{\gcd(d_1,d_2,d_3)^{2\delta} B^3 (B^{3/2}/Q')^{-2}}
%{(d_1d_2d_3)^{1-\eps}}
%(\frac{1}{L^{\delta/2}}
%+ \frac{Q_4^{17/18}}{B^{3/2}})
%\end{split}
%\end{equation*}
%provided that $A_1>4$.
%Since $Q_4\ll Q'\ll B^{3/2}$, the desired result follows.
\end{proof}

Ultimately, on plugging Lemmas~\ref{Q2-aspect},
\ref{J2-Q'-aspect},
\ref{Q1-aspect},
\ref{Q3-aspect},
and~\ref{Q4-aspect}
(with $4A$ in place of $A$ in the latter four)
into the upper bound \eqref{main-holder-bound} for \eqref{before-main-holder},
we find that
\begin{equation*}
E_{1,3}\ll
\frac{Q_2^4
Q_3^{4+7\eps}
Q_4^4}{Q_1^2Q_0^6}
\sum_{\substack{d_1,d_2,d_3\ge 1 \\ \textnormal{square-full}}}
(d_1d_2d_3)^{\frac12-2\eps}
\frac{B^3 H}{(d_1d_2d_3)^{1-\eps} (B^{3/2}/Q')^2},
\end{equation*}
where
\begin{equation*}
\begin{split}
H &= ((\log_+{B})^{3\delta})^{1/(1+\delta)}
((B^{3/2}/Q')^{-\min(4A_1A,1/5)})^{1/(4A)}
((\log_+{Q_1})^{-1/\eps})^{1/(4A)} \\
&\qquad
\times\left(\frac{\gcd(d_1,d_2,d_3)^{2\delta}}{Q_3^{\delta/2000}}\right)^{1/(4A)}
\left(\frac{\gcd(d_1,d_2,d_3)^{2\delta} }{Q_4^{\delta/2000}}\right)^{1/(4A)}.
\end{split}
\end{equation*}
Note that 
\begin{equation*}
\sum_{\substack{d_1,d_2,d_3\ge 1 \\ \textnormal{square-full}}}
\frac{\gcd(d_1,d_2,d_3)^{\delta/A}}{(d_1d_2d_3)^{\frac12+\eps}}
% = \prod_p \sum_{\substack{e_1,e_2,e_3,e\ge 0 \\ e_1,e_2,e_3\ne 1 \\ \min\{e_1,e_2,e_3\}=e}}
% \frac{p^{e\delta/A}}{p^{(e_1+e_2+e_3)(\frac12+\eps)}}
= \prod_p (1 + \frac{O(1)}{p^{1+2\eps}}
% + \frac{O(p^{\delta/A})}{p^{\frac32+3\eps}} but e\ge 1 implies e\ge 2.
+ \frac{O(p^{2\delta/A})}{p^{3+6\eps}})
\ll 1,
\end{equation*}
on assuming $\delta/A < 1$. Hence
\begin{equation*}
E_{1,3}\ll
\frac{Q_2^4
Q_3^{4+7\eps}
Q_4^4}{Q_1^2Q_0^6}
\frac{B^3 (\log_+{B})^{3\delta} (\log_+{Q_1})^{-2}}
{(B^{3/2}/Q')^{2+\eta} Q_3^\eta Q_4^\eta}
\le \frac{(Q_2Q_3Q_4)^4 (Q')^2 (\log_+{B})^{3\delta}}
{Q_1^2Q_0^6 (\log_+{Q_1})^2 (B^{3/2}/Q')^\eta Q_3^{\eta/2} Q_4^\eta},
\end{equation*}
on recalling that $Q'=Q_0Q_1\asymp Q_1Q_2Q_3Q_4$, 
where $\eta=\min(4A_1A,1/5,\delta/2000)/(4A)$
and $\eps\le \min(1/(8A),\eta/14)$.

However, by \eqref{bound-E2} and \eqref{define-and-bound-E3},
we have $E_{1,2}\ll \sum_{Q_2Q_3Q_4\asymp Q_0} E_{1,3}$.
Similarly, recalling the expressions for $E_{1}(B)$ and $E_{1,2}$
from \eqref{define-E1B} and  \eqref{define-and-partial-sum-E2}, respectively,
we have $E_1(B)=\sum_{Q_1,Q_0} E_{1,2}$.
Therefore,
$$
E_1(B)
\ll \sum_{Q_1Q_2Q_3Q_4\ll B^{3/2}} E_{1,3}
\ll \sum_{Q_1Q_2Q_3Q_4\ll B^{3/2}}
\frac{(\log_+{B})^{3\delta}}{(\log_+{Q_1})^2 Q_5^\eta Q_3^{\eta/2} Q_4^\eta},
$$
where $Q_5\defeq B^{3/2}/(Q_1Q_2Q_3Q_4)\asymp B^{3/2}/Q'\gg 1$.
Since $Q_2$ is determined by the quantities $Q_1,Q_3,Q_4,Q_5\gg 1$,
we conclude that
$$
E_1(B)
\ll \sum_{Q_1,Q_3,Q_4,Q_5\gg 1}
\frac{(\log_+{B})^{3\delta}}{(\log_+{Q_1})^2 Q_5^\eta Q_3^{\eta/2} Q_4^\eta}
\ll (\log_+{B})^{3\delta},
$$
since $\sum_{k\ge 1} \frac{1}{(2^k)^\eta}$ and $
\sum_{k\ge 1} \frac{1}{k^2} \ll 1$.
Taking $\delta\to 0$, we conclude the proof of the 
following result, which upper bounds the quantity $E_1(B)$ defined in \eqref{eq:E1B}.

\begin{proposition}
\label{p:E1B}
For all $B\ge 1$ and $\eps>0$, we have
$$
E_1(B)=
\sum_{\substack{\m,\n\in \ZZ^3\\
m_1m_2m_3 D(\m,\n)\ne 0}}
\sum_{q=1}^\infty
\frac{1}{q^6}S_q(\m,\n) I_q(\m,\n)
\ll (1+\log{B})^\eps.
$$
\end{proposition}

\section{Application of the Hooley \texorpdfstring{$\Delta$}{Delta}-function} \label{s:hooley}

The
{\em Hooley $\Delta$-function} is defined to be 
\begin{equation}\label{eq:Hooley}
\Delta(n)\defeq \max_{u\in \RR} \#\left\{d\in \NN: d\mid n \text{ and } e^u<d\leq e^{1+u}\right\},
\end{equation}
for any $n\in \mathbb{N}$. We extend this to all integers by writing $\Delta(-n)=\Delta(n)$ and $\Delta(0)=0$. We have 
$
\Delta(mn)\leq \Delta(m) \tau(n),
$
for any integers $m,n\in \NN$,  
by \cite[Lemma 61.1]{hall}.
The purpose of this section is to estimate the sum
\begin{equation}\label{eq:definition_of_ST}
S_T=
\sum_{\substack{D(\m,\n)\ne 0 \\ d_i\mid m_i\ne 0\\ 
|\m|,|\n|\leq  T
}} 
\Delta(G(\m,\n))^{1+\delta},
\end{equation}
for any real $\delta> 0$.
The sum $S_T$ vanishes unless $d_i \leq T$ for $1\leq i\leq 3$, but we shall proceed under the assumption that $d_1,d_2,d_3\leq \sqrt{T}$.

There is an extensive literature on 
upper bounds for 
averages of non-negative arithmetic functions over the values of polynomials, as 
found recently in the recent work of 
Chan--Koymans--Pagano--Sofos  \cite{chan}, and la Bret\`eche--Tenenbaum \cite{regis}.
We will require upper bounds that have sufficient uniformity in the polynomial involved and so it is the latter work that is more appropriate to our needs. 
To begin with, we 
note that 
$F(n)=
\Delta(n)^{1+\delta}$
 belongs to the class of non-negative arithmetic functions considered 
in \cite[\S~2]{regis}. 
We would like to apply \cite[Thm.~3.1]{regis} with  $k=1$ and $t=6$. 
The vectors $(\m,\n)$ with 
$n_1n_2n_3=0$ contribute $O(T^{5+\ve})$ to $S_T$, for any $\ve>0$.
Making the change of variables $m_i=d_im_i'$, and taking into account the possible signs of $m_i',n_i$, we deduce that 
\begin{equation}\label{eq:green_pot}
S_T=
2^3
\sum_{\ve_1,\ve_2,\ve_3\in \{\pm 1\}}
\sum_{\substack{
(\m',\n)\in \ZZ^6\\
0<m_i'\leq T/d_i\\
0<n_i\leq T
}} 
\Delta(G_{\dd}^{\bve}(\m',\n))^{1+\delta} +O(T^{5+\ve}),
\end{equation}
where
$$
G_{\dd}^{\bve}(\x,\y)=\sum_{i=1}^3 d_i^2x_i^2y_i^4-2\sum_{1\leq i<j\leq 3}\ve_i\ve_jd_id_jx_ix_jy_i^2y_j^2.
$$
For a fixed choice of $\bve=(\ve_1,\ve_2,\ve_3)$, we now  focus  on estimating the sum 
$$
S=\sum_{\substack{\mathbf{u}
\in \ZZ^6\\
0<u_j\leq X_j, ~(1\leq j\leq 6)
}} 
\Delta(G_{\dd}^{\bve}(\mathbf{u}))^{1+\delta},
$$
with
$$
X_j=\begin{cases}
T/d_j & \text{ if $1\leq j\leq 3$,}\\
T & \text{ if $4\leq j\leq 6$}.
\end{cases}
$$
Since $\Delta(mn)\ll m^\ve \Delta(n)$ for any coprime integers $m,n$, it follows that
$$
S\ll \gcd(d_1,d_2,d_3)^\ve \sum_{\substack{\mathbf{u}
\in \ZZ^6\\
0<u_j\leq X_j, ~(1\leq j\leq 6)
}} 
\Delta(G(\mathbf{u}))^{1+\delta},
$$
where $G=G^{\bve}_{\dd'}$ is the primitive irreducible form in $\ZZ[\x,\y]$
that corresponds to taking $\dd'=\dd/\gcd(d_1,d_2,d_3)$.
This is exactly the sum considered in \cite[Thm.~3.1]{regis} with $Y_j=X_j$.
Suppose that $T\gg 1$ and note 
 that the conditions in \cite[Eq.~(3.2)]{regis} all hold with $\alpha=1$ and $\beta=1/2$, since 
\begin{align*}
\left(\max_{1\leq j\leq 6} X_j+\|G\|\right)^{1/2} \leq \left(T+\max\{d_1,d_2,d_3\}^2\right)^{1/2}\ll 
\sqrt{T} 
&\ll 
\frac{T}{\max\{d_1,d_2,d_3\}}\\
&= \min_{1\leq j\leq 6} X_j, 
\end{align*}
under our  assumption that $d_1,d_2,d_3\leq \sqrt{T}$.

Let
$$
\rho_G(s)=\#\left\{\mathbf u\in (\ZZ/s\ZZ)^6: G(\mathbf u)\equiv 0 \bmod{s}\right\} ,
$$
for any $s\in \NN$.
Taking $g=6$ 
in \cite[Thm.~3.1]{regis} and applying \cite[Eq.~(3.5)]{regis}, 
it therefore follows that 
\begin{equation}\label{eq:nair}
S\ll \frac{\gcd(d_1,d_2,d_3)^\ve T^6}{d_1d_2d_3}
\sum_{\substack{s\in \NN\\
s\leq 6T}} \frac{\Delta(s)^{1+\delta}\rho_G(s)}{s^6}
\prod_{6<p\leq \sqrt{T}}\left(1-\frac{\rho_G(p)}{p^6}\right),
\end{equation}
if $T\gg 1$,
where we have taken $x=T/\max\{d_1,d_2,d_3\}\geq \sqrt{T}$ in 
the product over primes in \cite[Thm.~3.1]{regis}.
In fact this result holds for any $T\geq 1$, since it is trivially true when $T\ll 1$.
The following result collects together what we shall need about the function $\rho_G(s)$, which it suffices to understand at prime powers $s=p^j$.

\begin{lemma}\label{lem:prime_power}
Let $\bve\in \{\pm 1\}^3$, let $\mathbf{d}\in \NN^3$ such that $\gcd(d_1,d_2,d_3)=1$, and put
$G=G^{\bve}_{\dd}$. 
Let  $p^j$ be a 
 prime power. 
If  $j\geq 2$ then 
$$
p^{-6j}\rho_G(p^j)\ll \min\left\{ p^{-2},  (j+1)^5p^{-j/6}\right\},
$$
for an absolute implied constant.
Furthermore, we have 
$$
\rho_G(p)=
\begin{cases}
p^5+O(p^{9/2}) & \text{ if $p\nmid d_1d_2d_3$,}\\
p^5+O(p^{4}) & \text{ if $p\mid d_i$ and $p\nmid d_jd_k$,}\\
2p^5+O(p^{4}) & \text{ if $p\mid \gcd(d_i,d_j)$ and $p\nmid d_k$,}
\end{cases}
$$
for some permutation $\{i,j,k\}=\{1,2,3\}$. 
\end{lemma}

\begin{proof}
Since $G$ is irreducible and has content $c(G)=1$, the  bound 
$$
p^{-6j}\rho_G(p^j)\leq p^{-12}\rho_G(p^2)\ll p^{-2}
$$ 
is straightforward.  Moreover, 
the bound 
$p^{-6j}\rho_G(p^j)\leq 6^6 (j+1)^5p^{-j/6}$ is a direct  consequence of \eqref{eq:regis} with $d=6$.
Suppose now that $j=1$. 
The case $p=2$ is trivial and so we can assume that $p$ is odd. 
If  $p\nmid d_1d_2d_3$, then $G$ is irreducible over  $\overline \FF_p$ and it follows from the Lang--Weil estimate that 
$\rho_G(p)=p^5+O(p^{9/2})$.  If $p\mid d_1$ but $p\nmid d_2d_3$, say, then 
$$
G(\mathbf{u})\equiv (d_2u_2u_5^2 \pm d_3u_3u_6^2)^2 \bmod{p},
$$ 
for an appropriate sign $\pm$ depending on $\bve$. In this case clearly $\rho_G(p)=p^5+O(p^4).$
Finally, we get $\rho_G(p)=2p^5+O(p^4)$ if $p\mid (d_1,d_2)$ and $p\nmid d_3$, say.
\end{proof}

It follows from the second part of Lemma \ref{lem:prime_power} that 
$\rho_G(p)\geq p^5+O(p^{9/2})$, for any prime $p$. Hence
\begin{equation}\label{eq:yellow_pot}
\prod_{6<p\leq \sqrt{T}}\left(1-\frac{\rho_G(p)}{p^6}\right)\leq 
\prod_{6<p\leq \sqrt{T}}\left(1-\frac{1}{p}+O\left(\frac{1}{p^{3/2}}\right)\right)\ll \frac{1}{\log T}
\end{equation}
in \eqref{eq:nair}.
Any integer $s$ can be written $s=ab$, for coprime integers $a,b$ such that  $a$ is square-free and $b$ is  square-full. Since $\Delta(s)\leq \Delta(a) \tau(b)$, we deduce that 
$$
\sum_{\substack{s\in \NN\\
s\leq 6T}} \frac{\Delta(s)^{1+\delta}\rho_G(s)}{s^6}\leq
\sum_{\substack{a\in \NN\\
a\leq 6T}} \frac{\mu^2(a)\Delta(a)^{1+\delta}\rho_G(a)}{a^6}
\sum_{\substack{b\in \NN\\ \text{$b$ square-full}}} \frac{\rho_G(b)\tau(b)^{1+\delta}}{b^{6}}.
$$
Now 
\begin{align*}
\sum_{\substack{b\in \NN\\ \text{$b$ square-full}}} \frac{\rho_G(b)\tau(b)^{1+\delta}}{b^{6}}
&\leq
\prod_p  \left(1+\sum_{2\leq j\leq 12}\frac{O(1)}{p^2}+
\sum_{j>12}\frac{O((j+1)^{6+\delta})}{p^{j/6}} \right)\\
&\leq
\prod_p  \left(1+\frac{1}{p^2}\right)^{O(1)}\\
&\ll 1,
\end{align*}
by Lemma \ref{lem:prime_power}.
Let $D=\gcd(d_1,d_2)\gcd(d_1,d_3)\gcd(d_2,d_3)$.
Hence, on applying the second part of 
Lemma \ref{lem:prime_power}, we easily deduce that 
\begin{align*}
\sum_{\substack{s\in \NN\\
s\leq 6T}} \frac{\Delta(s)^{1+\delta}\rho_G(s)}{s^6}
&\ll 
\sum_{\substack{a\in \NN\\
a\leq 6T}} \frac{\mu^2(a)\Delta(a)^{1+\delta}}{a}
\prod_{\substack{p\mid a\\ 
p\nmid D
}} \left(1+\frac{1}{\sqrt{p}}\right)^{O(1)}
\prod_{\substack{p\mid \gcd(a,D)}} \left(2+\frac{1}{p}\right)^{O(1)}\\
&\ll 2^{O(\omega(D))}
\sum_{\substack{a\in \NN\\
a\leq 6T}} \frac{\mu^2(a)\Delta(a)^{1+\delta}}{a}
\prod_{\substack{p\mid a}} \left(1+\frac{1}{\sqrt{p}}\right)^{O(1)}.
\end{align*}
We have  $2^{O(\omega(D))}\ll D^\ve$. Furthermore, the function 
$$
g(a)=\mu^2(a)\tau(a)^\delta\prod_{\substack{p\mid a}} \left(1+\frac{1}{\sqrt{p}}\right)^{O(1)}
$$
clearly belongs to the class of functions considered in \cite{regis'}, with $A=O_\delta(1)$
and $y=2^\delta$. Hence it follows from \cite[Thm.~1.1]{regis'} that 
$$
\sum_{\substack{a\in \NN\\
a\leq 6T}} \mu^2(a)\Delta(a)^{1+\delta}
\prod_{\substack{p\mid a}} \left(1+\frac{1}{\sqrt{p}}\right)^{O(1)}\ll_\delta 
T(\log T)^{2^{1+\delta}-2}\exp\left(c\sqrt{\log\log T}\right),
$$
for any $c>\sqrt{2}\log 2$. Applying partial summation, it follows that 
\begin{equation}\label{eq:red_pot}
\sum_{\substack{s\in \NN\\
s\leq 6T}} \frac{\Delta(s)^{1+\delta}\rho_G(s)}{s^6}\ll_{\delta} 
D^\ve 
(\log T)^{2^{1+\delta}-1+\ve},
\end{equation}
for any $\ve>0$.

Note that $2^{1+\delta}-2+\eps=2(2^\delta-1)+\eps=2\delta\log 2+O(\delta^2)$,
if $\eps$ is taken to be sufficiently small in terms of $\delta$.
 Finally, we combine \eqref{eq:yellow_pot} and \eqref{eq:red_pot} in 
\eqref{eq:nair} and then \eqref{eq:green_pot}, in order to deduce the following result.

\begin{lemma}
\label{Tim-S_T-Hooley-delta-bound}
Let $\delta> 0$. 
Assume that $d_1,d_2,d_3\leq \sqrt{T}$. Then 
$$
S_T\ll_{\delta} \frac{D^\ve T^6 (\log T)^{2\delta\log 2+O(\delta^2)}}{d_1d_2d_3}
+T^{5+\ve},
$$
for any $\ve>0$, 
where $D=\gcd(d_1,d_2)\gcd(d_1,d_3)\gcd(d_2,d_3)$.
\end{lemma}

\section{The treatment of \texorpdfstring{$E_2(B)$}{E2(B)}}

Recall the definition of $E_2(B)$ from \eqref{eq:E2B}, which we shall estimate as follows.

\begin{proposition}
\label{PROP:E2B}
    We have
$
        E_2(B)\ll B^{-1/4+\varepsilon}.
$
\end{proposition}

There are two cases we need to consider:
(1) $m_1=0$, $m_2m_3\neq 0$,
and (2) $m_1=m_2=0$, $m_3\neq 0$.
It follows from Lemma \ref{int-estimate} that
\begin{align*}
    I_q(\mathbf{m}, \mathbf{n})\ll_A \frac{q^2}{B^2\max\{|\mathbf{m}|, |\mathbf{n}|\}^2}\left(1+\frac{\max\{|\mathbf{m}|, |\mathbf{n}|\}}{B^{1/2}}\right)^{-A},
\end{align*}
for any $A>0$. For  square-free $q$ we have,
by Corollary~\ref{cor.3.8},
\begin{align*}
S_q(\mathbf{m}, \mathbf{n})\ll q^{3+\varepsilon}
\gcd(q, D(\mathbf{m}, \mathbf{n})).
\end{align*}
For $q$ square-full, we will use
Lemma~\ref{lem:square-full} and Corollary~\ref{cor:basic} in case~(1),
and additionally Lemma~\ref{tricky22} in case~(2).

Let us consider the contribution $E_{2,1}$ to $E(B)$  from $m_1=0$, $m_2m_3\neq 0$.
In this case
$D(\mathbf{m}, \mathbf{n})=(m_2n_2^2-m_3n_3^2)^2$ is free of $n_1$, by 
 \eqref{eq:factorize-dual-form}.
Placing  $\max\{|\mathbf{m}|, |\mathbf{n}|\}$ into dyadic intervals of order $T$, we obtain
\begin{align*}
E_{2,1}
\ll~& \sum_{\substack{T\gg 1 \\ m_2,m_3,n_2,n_3\ll T
\\ m_2m_3\ne 0 \\ m_2n_2^2-m_3n_3^2\ne 0}}
\sum_{\substack{q_2\ll B^{3/2}\\q_2\;\text{square-full}\\ q_2\mid (m_2n_2^2-m_3n_3^2)^4}}\\
&\quad \times
\sum_{\substack{n_1\ll T \\ \{q_2,0\}^{1/2}\mid n_1
\\ q_1\ll B^{3/2}/q_2\\q_1\;\text{square-free}}}
\frac{B^\eps \gcd(q_1, D(\mathbf{m}, \mathbf{n}))
\sqrt{\{q_2,0\}\{q_2,m_2\}\{q_2,m_3\}}}{q_1 B^2 T^2 (1+T/B^{1/2})^A}.
\end{align*}
Summing over $q_1$ using \eqref{gcd-1-average},
then over $n_1$ using the condition $\{q_2,0\}^{1/2}\mid n_1$, gives
\begin{equation*}
E_{2,1}
\ll \sum_{\substack{T\gg 1 \\ m_2,m_3,n_2,n_3\ll T
\\ m_2m_3\ne 0 \\ m_2n_2^2-m_3n_3^2\ne 0}}
\sum_{\substack{q_2\ll B^{3/2}\\q_2\;\text{square-full}\\ q_2\mid (m_2n_2^2-m_3n_3^2)^4}}
(\{q_2,0\}^{1/2} + T)
\frac{(BT)^\eps
\sqrt{\{q_2,m_2\}\{q_2,m_3\}}}{B^2 T^2 (1+T/B^{1/2})^A}.
\end{equation*}
Next we observe that
$\{q_2,m_j\}\mid \{(m_2n_2^2-m_3n_3^2)^4,m_j\}$
for $j\in \{2,3\}$.
Hence
\begin{equation*}
\sqrt{\{q_2,m_2\}\{q_2,m_3\}}
\ll \{q_2,m_2\}+\{q_2,m_3\}
\ll \{m_3^4n_3^8,m_2\}+\{m_2^4n_2^8,m_3\}.
\end{equation*}
The number of choices for $q_2$ is $O(T^\eps)$, by the divisor bound. 
Noting that $\{q_2,0\}\leq q_1\ll B^{3/2}$, we see that 
\begin{equation*}
E_{2,1}
\ll \sum_{\substack{T\gg 1 \\ m_3,n_2\ll T
\\ m_3\ne 0}}
(B^{3/4} + T)
\frac{(BT)^\eps}{B^2 T^2 (1+T/B^{1/2})^A}
\sum_{n_3\ll T}
\sum_{0\ne m_2\ll T}
\{m_3^4n_3^8,m_2\}.
\end{equation*}
The sum over $m_2$
is $\ll T^{1+\eps} + T^2 \1_{n_3=0}$, by \eqref{gcd-1-average},
so the sum over $n_3$ is $\ll T^{2+\eps}$.
Thus
\begin{equation}
\label{E21-finale}
E_{2,1}
\ll \sum_{T\gg 1}
T^2 (B^{3/4} + T)
\frac{(BT)^\eps}{B^2 (1+T/B^{1/2})^A}
\ll B^{\eps-1/4}.
\end{equation}

Finally we consider the contribution $E_{2,2}$
to $E_2(B)$ from $m_1=m_2=0$, $m_3\neq 0$.
In this case we have  $D(\mathbf{m}, \mathbf{n})=m_3^2n_3^4\neq 0$.

\begin{lemma}
\label{tricky22}
Suppose $m_1=m_2=0$, and $q$ is square-full.
Then 
\begin{align*}
S_{q}(\mathbf{m},\mathbf{n})\ll q^{4+\eps}
\gcd(q, m_3)\gcd(q, n_3).
\end{align*}
\end{lemma}

\begin{proof}
It is enough to establish the bound for $q=p^r$ with $r\geq 2$.
We use Lemma~\ref{lem:primepower}.
To estimate $N_j(p^r)$, we observe that $p^h\mid y_1, y_2$ where $h=\ceil{r/2}$.
Hence the number of $(y_1, y_2)$ pairs is bounded by $p^{2\floor{r/2}}\le p^{r}$.
The congruence $\mathbf{n}.\mathbf{y}\equiv 0\bmod{p^{r-j}}$ implies that
once $(y_1,y_2)$ is chosen,
the number of $y_3$ is bounded by $p^j\gcd(p^{r-j}, n_3)
\le p^j\gcd(p^r, n_3)$.
Finally, the congruence $ay_3^2+m_3\equiv 0\bmod{p^r}$ implies that
once $\y$ is chosen,
the number of $a$ is bounded by $\gcd(p^r, y_3^2) = \gcd(p^r, m_3)$. 
The lemma follows.
\end{proof}    

Employing the bound in Lemma~\ref{tricky22} for square-full moduli, it follows that
\begin{equation*}
E_{2,2}
\ll \sum_{\substack{T\gg 1 \\ m_3,n_3\ll T
\\ m_3n_3\ne 0}}
\sum_{\substack{q_2\ll B^{3/2}\\q_2\;\text{square-full}\\ q_2\mid (m_3^2n_3^4)^2}}
\sum_{\substack{n_1,n_2\ll T \\ 
%\{q_2,0\}^{1/2}\mid n_1,n_2\\ 
q_1\ll B^{3/2}/q_2\\q_1\;\text{square-free}}}
\frac{B^\eps \gcd(q_1, D(\mathbf{m}, \mathbf{n}))
\gcd(q_2, m_3)\gcd(q_2, n_3)}{q_1 B^2 T^2 (1+T/B^{1/2})^A}.
\end{equation*}
We can actually ignore the condition $q_2\mid (m_3^2n_3^4)^2$.
%and $\{q_2,0\}^{1/2}\mid n_1,n_2$.
% The other conditions, such as $m_3n_3\ne 0$, remain essential.
Carrying out the sums over $q_1$, $n_3$ and $m_3$
using \eqref{gcd-1-average},
we get
\begin{equation*}
E_{2,2}
\ll \sum_{T\gg 1}
\sum_{\substack{q_2\ll B^{3/2}\\q_2\;\text{square-full}}}
\sum_{n_1,n_2\ll T}
\frac{(BT)^\eps}{B^2 (1+T/B^{1/2})^A}
\ll \sum_{T\gg 1}
\frac{B^{3/4} T^2 (BT)^\eps}{B^2 (1+T/B^{1/2})^A}.
\end{equation*}
As in \eqref{E21-finale},
this is $O(B^{\eps-1/4})$, which thereby completes the proof of Proposition~\ref{PROP:E2B}.

\section{The contribution from the dual variety}\label{s:dual}

The goal of this section is to prove the following result,
concerning $E_3(B)$ from \eqref{eq:E3B}.
\begin{proposition}
\label{final-E3-estimate}
We have
\begin{equation*}
E_3(B)
= \frac{\sigma_\infty}{4\zeta(3)} \log{B}
+ O(1).
\end{equation*}
\end{proposition}

Before proceeding, we offer some intuition for why
the contribution from vectors vanishing on the dual form make  up $\frac{1}{4}$ of the main term. 
If $\gcd(y_1,y_2,y_3)=h$, then 
for $1\ll H\ll B$ each dyadic range $h\sim H$
should contribute
$$\asymp \sum_{h\sim H} \frac{B^3 (B/h)^3}{B (B/h)^2}
= \sum_{h\sim H} \frac{B^3}{h} \asymp B^3$$
solutions $(\x,\y)$ to the counting function $N(B)$.
The solutions $(\x,\y)$ in the range $B^{3/4}\ll h\ll B$ are covered by $3$-dimensional lattices of relatively low height.
Specifically, if we write $\y=h\mathbf{t}$, then $(\x,\y)\in \Lambda(\mathbf{t})$, where $\Lambda(\mathbf{t})$ is defined below in \eqref{Lambda-t} and  $\mathbf{t}=\y/h\ll B/h\ll B^{1/4}$.
These lattices will arise naturally through our analysis of $E_3(B)$.
On the other hand, $M(B)$ captures the contribution from the range $h\ll B^{3/4}$, but the reason for this is more mysterious.
For a partial explanation, consider the following interrelated facts:
\begin{enumerate}
\item
prime-square moduli $p^2$ are responsible for the polar factor $\zeta(2s-11)$ of $F(s)$ in the proof of Lemma \ref{lem:Sigma-x}, in the sense that the restriction
$$\prod_p \left(1 + \left(1-\frac{1}{p}\right)\sum_{1\le r\ne 2}
p^{4r+3\lfloor r/2 \rfloor-rs}\right)$$
of $F(s)$ to products of prime powers $p^r$ with $r\ne 2$ would be absolutely convergent at $s=6$;
\item 
the locus $F(\x,\y) \equiv 0\bmod{h^2}$ includes the locus $\y \equiv \0 \bmod{h}$; and
\item
$\y = \0$ is the singular locus of $F(\x,\y)=0$.
\end{enumerate}
The moduli $q$ in the sum $M(B)$ are restricted to the range $q\ll Q$, so terms like $S_{h^2}(\0,\0)$, which carry information about the congruences $F(\x,\y) \equiv 0\bmod{h^2}$, might only be capable of detecting divisors $h\mid \y$ with $h^2\ll Q=B^{3/2}$; i.e.~$h\ll B^{3/4}$.
In the end, we will have
\begin{equation*}
M(B) = \int_{1}^{B^{3/4}} \frac{\sigma_\infty}{\zeta(3)} \frac{dh}{h} + O(1),
\qquad E_3(B) = \int_{B^{3/4}}^{B} \frac{\sigma_\infty}{\zeta(3)} \frac{dh}{h} + O(1).
\end{equation*}

\subsection*{Building blocks}

Define $\ZZp^3$ to be the set of non-zero vector $\mathbf{t}\in \mathbb{Z}^3$
such that $\gcd(t_1,t_2,t_3)=1$.
Given $\mathbf{t}\in \ZZp^3$, 
let $\mathbf{t}^2=(t_1^2,t_2^2,t_3^2)$ and let
\begin{equation}
\label{Lambda-t}
\Lambda=\Lambda(\mathbf{t})
= \{(\x,\y)\in (\mathbb{Z}^3)^2:
\x\cdot \mathbf{t}^2=0,
\; \y\in \mathbf{t}\mathbb{Z}\}.
\end{equation}
% $\mathbf{t}\times \y=0$, where
% \begin{equation*}
% \mathbf{t}\times \y=(t_2y_3-t_3y_2,t_3y_1-t_1y_3,t_1y_2-t_2y_1)
% \end{equation*}
% denotes the cross product.
Then $\Lambda\subset \mathbb{Z}^6$ is a primitive rank $3$ lattice
that vanishes on  the hypersurface $F(\x,\y)=0$.
Its orthogonal complement is
\begin{equation}
\label{Lambda-perp-t}
\Lambda^\perp=\Lambda^\perp(\mathbf{t})
= \{(\m,\n)\in (\mathbb{Z}^3)^2:
\m\in \mathbf{t}^2\mathbb{Z},
\; \n\cdot \mathbf{t}=0\}.
\end{equation}
It is readily checked that $\Lambda^\perp$ vanishes on  the dual hypersurface.

Let $\m\n\defeq (m_1n_1,m_2n_2,m_3n_3)$.
The equation $\m\n=\0$ cuts out a union of coordinate lattices
contained in the dual variety $D(\m,\n)=0$.
The remaining integral points on the dual variety are characterized by the following result, 
in which we identify $\PP^2(\ZZ)$ with $\ZZp^3$.

\begin{lemma}
\label{non-coordinate-lattices}
Let $\m,\n\in \ZZ^3$ such that 
 $\m\n\ne \0$ and  $D(\m,\n)=0$.
\begin{enumerate}
\item If $n_1n_2n_3\ne 0$,
% then $\dim_\QQ \sum_i m_i^{1/2}\QQ = 1$
then there exists a unique point $[\mathbf{t}]\in \PP^2(\ZZ)$
% i.e.~primitive vector $\mathbf{t}\in \ZZ^3$ up to scaling by $\pm 1$,
such that $(\m,\n)\in \Lambda^\perp(\mathbf{t})$.

\item Let $\{i,j,k\} = \{1,2,3\}$.
If $m_in_i=0$, then $m_jn_j^2 = m_kn_k^2 \ne 0$,
and there exists a unique point $[\mathbf{t}]\in \PP^2(\ZZ)$
with $t_i=0$
such that $(m_j,m_k)\in (t_j^2,t_k^2)\ZZ$
% and $n_jt_j+n_kt_k=0$.
and $(n_j,n_k)\in (t_k,-t_j)\ZZ$.

\item Suppose $m_jm_kn_jn_k\ne 0$ for some $j<k$.
Then $m_jm_k$ is a non-zero square.
\end{enumerate}
\end{lemma}

\begin{proof}
We have $\sum_i \eps_i m_i^{1/2}n_i = 0$ for some choice of signs $\eps_i=\pm 1$.
% The linear relation $\sum_i \eps_i m_i^{1/2}n_i = 0$
% is either indecomposable,
% in which case $\prod_i (m_in_i)\ne 0$ and $m_i^{1/2}\in m_j^{1/2}\QQ^\times$,
% or 
% If $\dim_\QQ \sum_i m_i^{1/2}\QQ = 0$,
% then $\m=\0$ and $(\m,\n)\in \Lambda^\perp(\0) = \{(\0,\n):\n\in \ZZ^3\}$.
Since $\m\ne\0$, we have $\dim_\QQ \sum_i m_i^{1/2}\QQ \ge 1$.
Since $\n\ne\0$, we have $\dim_\QQ \sum_i m_i^{1/2}\QQ \le 2$.

If $\dim_\QQ \sum_i m_i^{1/2}\QQ = 1$,
then there exists a point $[\mathbf{t}]\in \PP^2(\ZZ)$
such that $(\m,\n)\in \Lambda^\perp(\mathbf{t})$.
Moreover,
% $[\mathbf{t}]$ is unique if $\prod_i n_i\ne 0$,
% and $\mathbf{t}^2\in \ZZ^3$ is unique if $\prod_i n_i=0$.
$\mathbf{t}^2\in \ZZ^3$ is uniquely determined by $\m$.

If $\dim_\QQ \sum_i m_i^{1/2}\QQ = 2$,
then the multi-set $\{m_1^{1/2}\QQ,m_2^{1/2}\QQ,m_3^{1/2}\QQ\}$
contains at least two distinct non-zero spaces,
so by the pigeonhole principle
some non-zero space $m_i^{1/2}\QQ$ appears with multiplicity one.
Then $m_i\ne 0$, since $m_i^{1/2}\QQ\ne 0$;
and $n_i=0$, since
the set of all subspaces of $\overline{\QQ}$ of the form $r^{1/2}\QQ$,
where $r$ is a non-zero square-free integer,
is linearly independent over $\QQ$.

(1):
If $n_1n_2n_3\ne 0$, then by the discussion above,
$\dim_\QQ \sum_i m_i^{1/2}\QQ = 1$.
Existence of $\mathbf{t}$ follows,
as does uniqueness up to coordinate-wise scaling of $\mathbf{t}$ by $\{\pm 1\}^3$.
To check uniqueness, let $S = \{i: t_i\ne 0\}$
and note that if $[\mathbf{t}]\ne [\eps_it_i]$ with $\eps_i=\pm1$
and $\sum_i t_in_i=\sum_i \eps_it_in_i=0$,
then (because $\card{S}\le 3$) there exists $i\in S$
with $t_in_i=0$,
% so either $t_i=0$ or $n_i=0$,
which is impossible.

(2):
If $m_in_i=0$, then $D(\m,\n)=0$ implies $m_jn_j^2 = m_kn_k^2$,
and $\m\n\ne \0$ implies that the latter products are non-zero.
Existence and uniqueness of $\mathbf{t}$ are clear.

(3):
If $m_in_i=0$, then $m_jm_k = (t_j^2h)(t_k^2h) = (t_jt_kh)^2$
for some $t_jt_kh\ne 0$ by (2).
If $m_in_i\ne 0$, then $\m=\mathbf{t}^2h$ for some $t_it_jt_kh\ne 0$ by (1),
and the claim follows in the same way.
\end{proof}

% Recall the definition of $\chi_m(r)$ from just before
% Lemma~\ref{cheap-large-sieve}.
% In particular, $\chi_0(r)\defeq \1_{r=1}$ and $\chi_1(r)\defeq 1$ for all $r\ge 1$.

\begin{lemma}
% Let $S$ be a large finite set of polynomials in $\ZZ[\x,\y]$.
% Let $S = \{\Delta\in 2\ZZ[\x,\y]: \Delta\mid 2x_1x_2x_3y_1y_2y_3\}$.
Let $S = 2\ZZ[\x,\y]$.
Given $(\m,\n)\in \ZZ^6$ with $D(\m,\n)=0$,
% Then there exists a $\Delta\in S$ with $\Delta(\m,\n)\ne 0$ such that
let
\begin{equation}
\label{assign-monomial-Delta}
\Delta(\x,\y)
% \defeq 2\gcd(\1_{\#\{i:m_i=0\}\ge 2},
% x_jx_k(y_jy_k)^{\1_{n_jn_k\ne 0}} \1_{m_i=0,\; m_jm_k\ne 0},
% \prod_{i=1}^{3} x_i^{\1_{m_i\ne 0}} y_i^{\1_{n_i\ne 0}})
\defeq \begin{cases}
2 &\text{if }\#\{i:m_i=0\}\ge 2, \\
2\prod_{i:m_i\ne 0} x_i y_i^{\1_{n_i\ne 0}} &\text{else},
\end{cases}
\end{equation}
be a certain element of $S$.
Then $\Delta(\m,\n)\ne 0$, and for all primes $p$, we have 
\begin{equation}
\label{prime-target}
\begin{split}
&\left|\frac{S_p(\m,\n)}{p^4}
- 1
- \sum_{i<j} \left(\frac{m_im_k}{p}\right) \1_{m_jn_in_j\ne 0} \1_{n_k=0}
- \sum_{j<k} \left(\frac{m_jm_k}{p}\right) \1_{n_j=n_k=0}\right|\\
&
\hspace{7cm}
 \ll \frac1p + \1_{p\mid \Delta(\m,\n)}.
\end{split}
\end{equation}
\end{lemma}

\begin{proof}
Before proceeding, observe that
% non-zero Dirichlet characters only appear in \eqref{prime-target} if $\#\{i: m_i\ne 0\}\ge 2$,
the sum $\sum_{i<j} (\frac{m_im_k}{p}) \1_{m_jn_in_j\ne 0} \1_{n_k=0}$
is non-zero only if $m_1m_2m_3\ne 0$ and $\#\{i: n_i=0\} = 1$,
and that the sum $\sum_{j<k} (\frac{m_jm_k}{p}) \1_{n_j=n_k=0}$
is non-zero only if $\#\{i: m_i\ne 0\}\ge 2$ and $\#\{i: n_i=0\}\ge 2$.
Moreover, if $n_k=0$ then the first sum reduces to $(\frac{m_im_k}{p}) \1_{m_jn_in_j\ne 0}$,
whereas if $m_i=0$ then the second sum simplifies to $(\frac{m_jm_k}{p}) \1_{n_j=n_k=0}$.

Suppose first that $m_1m_2m_3\ne 0$.
If $n_1n_2n_3\ne 0$, then $m_im_j$ is a non-zero square for all $i<j$
by Lemma~\ref{non-coordinate-lattices},
so \eqref{prime-target} holds with $\Delta = 2m_1m_2m_3n_1n_2n_3$;
generically $$S_p(\m,\n) = p^4-4p^3$$
by Lemmas~\ref{lem:Sp1} and \ref{lem:Sp2}.
If $n_in_j\ne 0$ and $n_k=0$, then $m_im_j$ is a non-zero square,
so $(\frac{m_im_k}{p})=(\frac{m_jm_k}{p})$ for all $p\nmid m_im_jm_k$,
and \eqref{prime-target} holds with $\Delta = 2m_1m_2m_3n_in_j$;
generically 
$$S_p(\m,\n) = \left(1+\left(\frac{m_im_k}{p}\right)\right)(p^4-2p^3).
$$
If $n_i=n_j=0$, then $m_kn_k=0$, so $n_k=0$, 
so \eqref{prime-target} holds with $\Delta = 2m_1m_2m_3$;
generically 
$$S_p(\m,\n) = \left(1+\left(\frac{m_1m_2}{p}\right)+\left(\frac{m_2m_3}{p}\right)+\left(\frac{m_3m_1}{p}\right)\right)(p^4-p^3).
$$
% sum generically factors as (1+\chi_{m_1m_2}(p))(1+\chi_{m_2m_3}(p)) since \chi_{m_3m_1}(p) = \chi_{m_1m_2}(p) \chi_{m_2m_3}(p) generically.

Suppose now that $m_1m_2m_3=0$.
If $\m=\0$, then \eqref{prime-target} holds with $\Delta=2$;
generically $$S_p(\m,\n) = p^4-p^3.$$
If $m_i\ne 0$ and $m_j=m_k=0$, then $n_i=0$, so \eqref{prime-target} holds with $\Delta=2$;
generically $$S_p(\m,\n) = p^4-p^3.$$
Finally, suppose $m_i=0$ and $m_jm_k\ne 0$.
Then $m_jn_j^2 = m_kn_k^2$, and $n_j=0 \Leftrightarrow n_k=0$.
If $n_j=n_k=0$, then \eqref{prime-target} holds with $\Delta=2m_jm_k$;
generically 
$$
S_p(\m,\n) = \left(1+\left(\frac{m_jm_k}{p}\right)\right)(p^4-p^3).
$$
If $n_jn_k\ne 0$, then $m_jm_k$ is a non-zero square
by Lemma \ref{non-coordinate-lattices},
so \eqref{prime-target} holds with $\Delta=2m_jm_kn_jn_k$;
generically $$S_p(\m,\n) = p^4-2p^3.$$

In each case, $\Delta$ matches the description given in \eqref{assign-monomial-Delta}.
\end{proof}

For $(\m,\n)\in \ZZ^6$ with $D(\m,\n)=0$,
let $\Delta$ be as in \eqref{assign-monomial-Delta}
and write
\begin{equation}
\label{dual-q0q1q2}
\frac{S_q(\m,\n)}{q^4}
= \sum_{q_0q'=q} \frac{\phi(q_0)}{q_0} S'_{q'}(\m,\n)
= \sum_{q_0q_1q_2=q}
\frac{\phi(q_0)}{q_0} \Xi_{q_1}(\m,\n) S''_{q_2}(\m,\n)
\end{equation}
where
\begin{equation}
\label{degen-char}
\Xi_q(\m,\n)
\defeq \1_{\gcd(q,\Delta(\m,\n))=1} \sum_{q_{ij}r_{12}r_{13}r_{23}=q}
\left(\frac{m_im_k}{q_{ij}}\right)
\prod_{j<k} \left(\frac{m_jm_k}{r_{jk}}\right),
\end{equation}
where the conventions for $q_{ij}$ and $r_{jk}$ are as follows:
\begin{enumerate}
\item If there exist indices $1\le i<j\le 3$ with $m_im_jm_kn_in_j\ne 0$ and $n_k=0$
(where we note that such a pair of indices $i<j$ is necessarily unique),
then $q_{ij}$ ranges over all positive integers.
Otherwise, $q_{ij}\defeq 1$, so that it may be ignored.

\item For each pair $1\le j<k\le 3$ with $m_jm_k\ne 0$ and $n_j=n_k=0$,
we let $r_{jk}$ range over all positive integers.
Otherwise, $r_{jk}\defeq 1$, so that it may be ignored.
\end{enumerate}
Observe that if $q_{ij}\ne 1$, then $\Delta=2m_1m_2m_3n_in_j$ and $r_{12}r_{13}r_{23}=1$.
Similarly, if $r_{12}r_{13}r_{23}\ne 1$, then $\Delta=2\prod_{k:m_k\ne 0}m_k\in \{2m_1m_2m_3,2m_1m_2,2m_1m_3,2m_2m_3\}$ and $q_{ij}=1$.

\begin{lemma}
\label{q-invariance}
The quantity $S'_q(\m,\n)$ depends only on $q$ and $(\m,\n)\bmod{q}$.
The same holds for the quantities
$\Xi_q(\m,\n)$
and $S''_q(\m,\n)$
provided that $\Delta\in S$ is fixed,
which holds for example if we fix which coordinates of $(\m,\n)\in \ZZ^6$
are zero and which are non-zero.
\end{lemma}

\begin{proof}
The quantity $S_q(\m,\n)$ depends only on $q$ and $(\m,\n)\bmod{q}$.
Since $\frac{\phi(q)}{q}$ depends only on $q$,
the claim for $S'_q(\m,\n)$ follows via \eqref{dual-q0q1q2}.
Now fix $\Delta\in S$.
Then the quantity $\1_{\gcd(q,\Delta(\m,\n))=1}$ depends only on $q$ and $(\m,\n)\bmod{q}$.
Moreover, if the residue class $(\m,\n)\bmod{q}$ is fixed,
% so that $\gcd(q,\Delta(\m,\n))=1$,
then the values of the symbols
$(\frac{m_im_k}{q_{ij}})$ and $(\frac{m_jm_k}{r_{jk}})$
in the definition \eqref{degen-char} of $\Xi_q(\m,\n)$
are uniquely determined by the moduli $q_{ij}$ and $r_{jk}$, respectively.
% https://en.wikipedia.org/wiki/Jacobi_symbol#Properties
The claim for $\Xi_q(\m,\n)$,
and then for $S''_q(\m,\n)$ via \eqref{dual-q0q1q2},
follows.
\end{proof}

\begin{lemma}
\label{S''-bounds}
If $p$ is prime then $S''_p(\m,\n)\ll p^{-1} + \1_{p\mid \Delta(\m,\n)}$.
Moreover, in general
\begin{equation*}
S''_q(\m,\n)
\ll q^\eps \sum_{\substack{\textnormal{square-full }q'\mid q
\\ \kappa(q')\mid \nabla{G}(\m,\n)}} \frac{|S_{q'}(\m,\n)|}{(q')^4}.
\end{equation*}
\end{lemma}

\begin{proof}
The first sentence holds by \eqref{prime-target}.
By \eqref{dual-q0q1q2} and \eqref{degen-char},
writing $\Delta$ for $\Delta(\m,\n)$,
\begin{equation*}
\frac{\sum_{l\ge 0} S''_{p^l}(\m,\n) p^{-ls}}{\sum_{l\ge 0} p^{-4l}S_{p^l}(\m,\n) p^{-ls}}
= \frac{1-p^{-s}}{1-p^{-s-1}}
\left(1-\frac{(\frac{m_im_k}{p})}{p^{s}} \1_{p\nmid \Delta}\right)
\prod_{j<k} \left(1-\frac{(\frac{m_jm_k}{p})}{p^{s}} \1_{p\nmid \Delta}\right),
\end{equation*}
whose $p^{-ls}$ coefficient is
$\ll 1+p^{-1}+p^{-2}+\dots = \frac{1}{1-p^{-1}} \le 2$.
% (In fact the coefficient decays as $l\to\infty$, but for now we will not use this.)
It follows that
\begin{equation*}
|S''_q(\m,\n)|
\le \sum_{q'\mid q} C^{\omega(q/q')} \frac{|S_{q'}(\m,\n)|}{(q')^4}
\le C^{\omega(q)} \sum_{q'\mid q} \frac{|S_{q'}(\m,\n)|}{(q')^4},
\end{equation*}
for a suitable constant $C>0$.
Now factor out the square-free part of $q'$,
and the coprime-to-$\nabla{G}(\m,\n)$ part of $q'$,
leaving a square-full modulus $q''\mid q'\mid q$ with $\kappa(q'')\mid \nabla{G}(\m,\n)$.
Since we always have 
 $|S_p(\m,\n)|\le Cp^4$, and 
 Lemma \ref{lem:beach} implies that 
 $|S_{p^l}(\m,\n)|\le p^{4l}$ if $p\nmid \nabla{G}(\m,\n)$, 
it follows that 
$$
\sum_{q'\mid q} \frac{|S_{q'}(\m,\n)|}{(q')^4}\leq 
(2C)^{\omega(q/q'')}
\sum_{\substack{\textnormal{square-full }q''\mid q
\\ \kappa(q'')\mid \nabla{G}(\m,\n)}} 
\frac{|S_{q''}(\m,\n)|}{(q'')^4},
$$
since 
 the number of choices for $q'$, given $q''$, is at most $2^{\omega(q/q'')}$.
The divisor bound finishes the proof.
\end{proof}

\begin{remark}
\label{uniform-pointwise-bound}
By Corollary~\ref{cor:basic}, $q^{-4}S_q(\m,\n)\ll q^{3/2+\eps}$.
By Lemma~\ref{S''-bounds}, $S''_q(\m,\n)\ll q^{3/2+\eps}$.
Since $\Xi_q(\m,\n)\ll q^\eps$ by \eqref{degen-char},
it follows from \eqref{dual-q0q1q2} that $S'_q(\m,\n)\ll q^{3/2+\eps}$ as well.
\end{remark}

\subsection*{Dyadic decomposition}

Returning to \eqref{eq:E3B}, it follows from  \eqref{dual-q0q1q2} that
\begin{equation}
\label{global-E3-decomp}
E_3(B)
= \sum_{\substack{(\m,\n)\ne \0 \\ D(\m,\n)=0}}
\sum_{q_0q_1q_2=q\ge 1}
\frac{I_q(\m,\n)}{q^2}
\frac{\phi(q_0)}{q_0} \Xi_{q_1}(\m,\n) S''_{q_2}(\m,\n).
\end{equation}
The strategy is now similar to \cite{wang2023ratios}, 
except that the ``main'' quantity $\frac{\phi(q_0)}{q_0}$ is independent of $(\m,\n)$,
as in the setting of \cite{wang2024dual}.
% After partial summation over $q_{ij}$ and $r_{jk}$,
% we take absolute values
We seek upper bounds in most dyadic ranges of moduli.
% this may require either GLH,
% or H\"{o}lder's inequality (raising the character sums to a large but fixed power)
% plus Heath-Brown's large sieve.
On the other hand,
we asymptotically evaluate the sum
% (without absolute values)
when $B^{3/2}/q_0$ is tiny
in terms of $B$ and $(\m,\n)$.
This last part is similar to \cite{wang2024dual},
except that infinitely many lattices
(whose successive minima are at most some power of $B$)
are involved rather than finitely many.

For each $(\m,\n)\in \ZZ^6$ with $D(\m,\n)=0$,
define $[\mathbf{t}]\in \PP^2(\ZZ)$
via Lemma~\ref{non-coordinate-lattices} if $\m\n\ne\0$,
and let $[\mathbf{t}]\defeq [1:1:1]\in \PP^2(\ZZ)$ if $\m\n=\0$,
for notational convenience.
If $E_3(B;M,T,\mathbf{Q})$
denotes the contribution to $E_3(B)$ from
$|(\m,\n)|\sim M$,
$|\mathbf{t}|\sim T$,
% $\lambda_{\rank\Gamma}(\Gamma(\m,\n))\sim \Theta$,
$q_0\sim Q_0$, $q_1\sim Q_1$, $q_2\sim Q_2$,
% $q_2\defeq q_{ij}r_{12}r_{13}r_{23}\sim Q_2$.
% $q_{ij}\sim Q_{ij}$, $r_{jk}\sim R_{jk}$.
% Upper bound on such a piece:
% Do a min of partial summations over $q_{ij}$, $r_{jk}$.
% Upper bound on such a piece:
then by Lemma \ref{int-estimate} and 
partial summation over $q_1$, 
there exists an interval $I\belongs \{q_1\sim Q_1\}$ such that
\begin{equation*}
E_3(B;M,T,\mathbf{Q})
\ll 
\hspace{-0.2cm}
\sum_{\substack{D(\m,\n)=0 \\
|(\m,\n)|\sim M \\ |\mathbf{t}|\sim T}}
\sum_{q_0\sim Q_0} \sum_{q_2\sim Q_2}
\frac{(Q_0Q_1Q_2+BM)^{-2}}{(1+M/B^{1/2})^A}
% \frac{\phi(q_0)}{q_0}
\left|S''_{q_2}(\m,\n)
\sum_{q_1\in I} \Xi_{q_1}(\m,\n)\right|,
\end{equation*}
for any $A>0$.

\begin{lemma}
\label{degen-counts}
For all $T,M\ge 1$ we have
\begin{equation*}
\#\{D(\m,\n)=0: |\m|,|\n|\ll M,\; |\mathbf{t}|\sim T\}
\ll M^3 \1_{M\gg T^2}.
\end{equation*}
Moreover, the point count restricted to $\m\n=\0$ is $\ll M^3 \1_{T\asymp 1}$,
and the point count restricted to $m_3n_3=0$, $\m\n\ne\0$
is $\ll \frac{M^3}{T} \1_{M\gg T^2}$.
\end{lemma}

\begin{proof}
% m_1m_2m_3n_1n_2n_3\ne 0 has \ll T^3 (M/T^2) (M^2/T) = M^3 points.
% m_3n_3=0, m_1n_1^2=m_2n_2^2\ne 0 has \ll M T^2 (M/T^2) (M/T) = M^3/T points.
% \m\n=\0 has \ll M^3 points.
If $\m\n=\0$, then $|\mathbf{t}|\asymp 1$,
so the number of available pairs $(\m,\n)$ is $\ll M^3 \1_{T\asymp 1}$.
If $m_3n_3=0$, $\m\n\ne\0$, then by Lemma~\ref{non-coordinate-lattices}
the number of available pairs $(\m,\n)$ is
\begin{equation*}
\ll \sum_{m_3n_3=0} \sum_{t_1,t_2\ll T} \frac{M}{T^2} \1_{M/T^2\gg 1} \frac{M}{T}
\ll M T^2 \frac{M^2}{T^3} \1_{M\gg T^2}
= \frac{M^3}{T} \1_{M\gg T^2}.
\end{equation*}
If $m_1m_2m_3n_1n_2n_3\ne 0$, then by Lemma~\ref{non-coordinate-lattices}
the number of available pairs $(\m,\n)$ is
\begin{equation*}
\ll \sum_{t_1,t_2,t_3\ll T} \frac{M}{T^2} \1_{M/T^2\gg 1} \frac{M^2}{T}
\ll M^3 \1_{M\gg T^2},
\end{equation*}
where $O(\frac{M^2}{T})$
bounds the number of solutions $\n\ll M$ to the equation $\n\cdot \mathbf{t}=0$,
viewed as a congruence modulo $\max(|t_1|,|t_2|,|t_3|) \asymp T$.
\end{proof}

\subsection*{Character sums}

For $q_1$, Lemma~\ref{cheap-large-sieve} will be useful.

\begin{lemma}
\label{first-Q1-bound}
For all $Q_1,T,M,A\ge 1$ and $I\belongs \{q\sim Q_1\}$ we have
\begin{equation*}
\Sigma(I,T,M)
\defeq \sum_{\substack{D(\m,\n)=0 \\
1\le |(\m,\n)|\le M \\ |\mathbf{t}|\sim T}}
\left|\sum_{q\in I} \Xi_q(\m,\n)\right|^A
\ll_{A} (MQ_1)^\eps M^3 Q_1^A
\left(\frac{1}{Q_1^{1/6}} + \frac{1}{M}\right).
\end{equation*}
\end{lemma}

\begin{proof}
We let $\Sigma = \Sigma(I,T,M)$ and write
\begin{equation}
\label{n-strat}
\Sigma
\ll \Sigma_{n_1n_2n_3\ne 0}
+ \Sigma_{n_1n_2\ne 0,\; n_3=0}
+ \Sigma_{n_1\ne 0,\; n_2=n_3=0}
+ \Sigma_{\n=\0}.
\end{equation}
Within \eqref{n-strat} we further write
\begin{equation*}
\begin{split}
\Sigma_{n_1n_2\ne 0,\; n_3=0}
&\ll \Sigma_{m_1m_2m_3n_1n_2\ne 0,\; n_3=0}
+ \Sigma_{n_1n_2\ne 0,\; m_1m_2m_3=n_3=0}, \\
\Sigma_{n_1\ne 0,\; n_2=n_3=0}
&\ll \Sigma_{m_2m_3n_1\ne 0,\; n_2=n_3=0}
+ \Sigma_{n_1\ne 0,\; m_2m_3=n_2=n_3=0}, \\
\Sigma_{\n=\0}
&\ll \Sigma_{m_1m_2m_3\ne 0,\; \n=\0}
+ \Sigma_{m_1m_2m_3=0,\; \n=\0}.
\end{split}
\end{equation*}

On the pieces $\Sigma_{n_1n_2n_3\ne 0}$,
$\Sigma_{n_1n_2\ne 0,\; m_1m_2m_3=n_3=0}$
and $\Sigma_{n_1\ne 0,\; m_2m_3=n_2=n_3=0}$,
we have $\Xi_q=\1_{q=1}$ by \eqref{degen-char},
so that by Lemma~\ref{degen-counts} their total is
\begin{equation*}
\ll \#\{D(\m,\n)=0: |\m|,|\n|\ll M,\; |\mathbf{t}|\sim T\}
\, \1_{Q_1\asymp 1}
\ll M^3\, \1_{Q_1\asymp 1}
\ll M^3 Q_1^{A-1}.
\end{equation*}

On the piece $\Sigma_{m_1m_2m_3=0,\; \n=\0}$
we plug in the trivial bound $\Xi_q\ll q^\eps$ to get
\begin{equation*}
\Sigma_{m_1m_2m_3=0,\; \n=\0}
\ll_{A} M^2 Q_1^{A+\eps}.
\end{equation*}

On $\Sigma_{m_1m_2m_3n_1n_2\ne 0,\; n_3=0}$,
Lemma~\ref{non-coordinate-lattices} gives
$(m_1,m_2) = (t_1^2h,t_2^2h)$,
$(n_1,n_2) = (t_2b,-t_1b)$,
$\Delta(\m,\n) = 2m_1m_2m_3n_1n_2 = -2t_1^3t_2^3h^2b^2m_3$,
and $m_1m_3 = t_1^2hm_3$. But then  \eqref{degen-char} gives
\begin{equation*}
\Xi_q(\m,\n)
= \left(\frac{hm_3}{q}\right)
\1_{\gcd(q,2t_1t_2hbm_3)=1}
= \left(\frac{hm_3}{q}\right)
\1_{\gcd(q,2t_1t_2b)=1}.
\end{equation*}
By Lemma~\ref{cheap-large-sieve}
with $H = M/T^2$
and $t = 2t_1t_2b$,
after using the trivial bound $\Xi_q\ll 1$ to reduce from an $A$th moment to a first moment,
we get
\begin{equation*}
\begin{split}
\Sigma_{m_1m_2m_3n_1n_2\ne 0,\; n_3=0}
&\ll_{A} (MQ_1)^\eps Q_1^{A-1}
\sum_{t_1,t_2\ll T} \sum_{b\ll M/T} (HMQ_1^{1/2} + (HM)^{1/2}Q_1) \\
&\ll (MQ_1)^\eps Q_1^{A-1}
T^2 (M/T) ((M/T)^2Q_1^{1/2} + (M/T)Q_1) \\
&\ll (MQ_1)^\eps M^3 Q_1^A
\left(\frac{1}{T Q_1^{1/2}} + \frac{1}{M}\right).
\end{split}
\end{equation*}

On $\Sigma_{m_2m_3n_1\ne 0,\; n_2=n_3=0}$, we have $m_1=0$ and $\Delta(\m,\n)=2m_2m_3$, so that
\begin{equation*}
\Xi_q(\m,\n)
= (\frac{m_2m_3}{q})
\1_{\gcd(q,2m_2m_3)=1}
= (\frac{m_2m_3}{q})
\1_{\gcd(q,2)=1}.
\end{equation*}
By Lemma~\ref{cheap-large-sieve}
with $H = M$
and $t = 2$,
we get
\begin{align*}
\Sigma_{m_2m_3n_1\ne 0,\; n_2=n_3=0}
&\ll (MQ_1)^\eps Q_1^{A-1}
\sum_{n_1\ll M} (M^2Q_1^{1/2} + MQ_1)\\
&\ll (MQ_1)^\eps M^3 Q_1^A
\left(\frac{1}{Q_1^{1/2}} + \frac{1}{M}\right).
\end{align*}

For $\Sigma_{m_1m_2m_3\ne 0,\; \n=\0}$, we have $\Delta(\m,\n)=2m_1m_2m_3$, so
\begin{align*}
\Xi_q(\m,\n)
&= \1_{\gcd(q,2m_1m_2m_3)=1} \sum_{r_{12}r_{13}r_{23}=q}
\prod_{j<k} \left(\frac{m_jm_k}{r_{jk}}\right)\\
&= \sum_{r_{12}r_{13}r_{23}=q}
\prod_{j<k} \left(\frac{m_jm_k}{r_{jk}}\right) \1_{\gcd(r_{jk},2m_i)=1},
\end{align*}
where $i=6-j-k$.
By the triangle inequality, there exists a choice of
dyadic ranges $r_{jk}\sim R_{jk}$, where $R_{12}R_{13}R_{23}\asymp Q_1$, for which
\begin{equation*}
\Sigma_{m_1m_2m_3\ne 0,\; \n=\0}
\ll_{A} Q_1^{A-1+\eps} \sum_{m_1m_2m_3\ne 0}
\left|\sum_{\substack{r_{12}r_{13}r_{23}\in I \\ r_{jk}\sim R_{jk}}}
\prod_{j<k} \left(\frac{m_jm_k}{r_{jk}}\right) \1_{\gcd(r_{jk},2m_i)=1}\right|.
\end{equation*}
Suppose $R_{12}\ge R_{13}\ge R_{23}$.
Then $R_{12}\ge (R_{12}R_{13}R_{23})^{1/3} \asymp Q_1^{1/3}$.
We find by Lemma~\ref{cheap-large-sieve},
with $H = M$,
$q = r_{12}$
and $t = 2m_3$,
that
\begin{align*}
\Sigma_{m_1m_2m_3\ne 0,\; \n=\0}
&\ll_{A} (MQ_1)^\eps Q_1^{A-1} \sum_{m_3\ll M}
\sum_{r_{13}\sim R_{13}} \sum_{r_{23}\sim R_{23}}
(M^2R_{12}^{1/2} + MR_{12}) \\
&\ll (MQ_1)^\eps Q_1^{A-1} M Q_1
(M^2R_{12}^{-1/2} + M)\\
&\ll (MQ_1)^\eps M^3 Q_1^A
\left(\frac{1}{Q_1^{1/6}} + \frac{1}{M}\right).
\end{align*}
In the light of  \eqref{n-strat}, we are done.
\end{proof}

\subsection*{Gains from sparsity}

For $q_2$, the following variant of Lemma~\ref{degen-counts}
will be useful.
\begin{lemma}
\label{HB-lattice-bound}
Let $\mathbf{t},\dd\in \ZZ^3$
with $|\mathbf{t}|\ge 1$
and with $d_i\ge 1$.
% Then
% \begin{equation*}
% \#\{\n\ll M: d_i\mid n_i,\; \sum n_it_i = 0\}
% \ll 1 + M^{1+\eps} + \frac{\gcd(d_1t_1,d_2t_2,d_3t_3)}{d_1d_2d_3} \frac{M^2}{T}.
% % \ll \frac{(M+T)^2}{T}
% \end{equation*}
% % For example, if $M\asymp T^2$ then we get a bound of $\ll T^3$.
Then
\begin{equation*}
\#\{|\n|\ll M: d_i\mid n_i,\; \n\cdot \mathbf{t} = 0,\; n_1n_2n_3\ne 0\}
\ll \frac{M}{|\dd|}
+ \frac{\gcd(d_1t_1,d_2t_2,d_3t_3)}{d_1d_2d_3} \frac{M^2}{|\mathbf{t}|}.
\end{equation*}
If $t_3=0$, then
\begin{equation*}
\#\{|\n|\ll M: d_i\mid n_i,\; \n\cdot \mathbf{t} = 0,\; n_1n_2\ne 0\}
\ll \frac{\gcd(d_1t_1,d_2t_2)}{d_1d_2d_3} \frac{Md_3+M^2}{|\mathbf{t}|}.
\end{equation*}
\end{lemma}

\begin{proof}
Assume $|t_1|\ge |t_2|\ge |t_3|$
and let $g = \gcd(d_1t_1,d_2t_2,d_3t_3)$.
If $\gcd(n_i/d_i)=h\ge 1$,
then $(n_1/d_1,n_2/d_2,n_3/d_3)=h\mathbf{u}$ where $\mathbf{u}\in \ZZp^3$.
It therefore follows from Heath-Brown \cite{HBLattice}*{Lemma~3},
with $X_i = M/(d_ih)>0$ and $v_i = d_it_i/g\in \ZZ$, that
\begin{equation*}
\#\{|\n|\ll M: \gcd(n_i/d_i)=h,\; \n\cdot \mathbf{t} = 0\}
\ll 1 + \frac{M^2}{d_2d_3h^2|v_1|}.
\end{equation*}
If $n_1n_2n_3\ne 0$, then $h\ll M/|\dd|$.
Summing over $h$, the first bound follows,
since $|v_1|\asymp d_1|\mathbf{t}|/g$.

If $t_3=0$, then $(n_1/d_1,n_2/d_2) = (bd_2t_2/g,-bd_1t_1/g)$
for some non-zero integer $b\ll Mg/|d_2d_1t_1|$.
Since $\#\{n_3\ll M: d_3\mid n_3\} \ll 1 + M/d_3$,
the second bound follows.
\end{proof}

For any  $H\geq 1$ and square-full integer $r\geq 1$, we claim that 
\begin{equation}\label{eq:***}
\#\left\{
h\leq H:  \text{$h$ is square-full  and  $r\mid h$}\right\}
\ll r^\ve \left(H/r\right)^{1/2}.
\end{equation}
To see this we note that any square-full integer $h\leq H$ divisible by a given square-full integer $r\ge 1$
is of the form $rh_1h_2'$, where $h_2'$ is square-full, $h_1$ is square-free, 
and $\gcd(h_1,h_2')=1$.
Then $h_1^2\mid rh_1h'_2=h_2$, say, so that $h_1\mid r$.
Thus the number of possible $h_2$, given $r$, is
$\ll \sum_{h_1\mid r} (\frac{H}{rh_1})^{1/2}
\ll (H/r)^{1/2}r^\eps$, as claimed.

\begin{lemma}
\label{first-Q2-bound}
For any $Q_2,T,M\ge 1$ we have
\begin{equation*}
\Sigma(Q_2,T,M)
\defeq \sum_{\substack{D(\m,\n)=0 \\
1\le |(\m,\n)|\le M \\ |\mathbf{t}|\sim T}}
\sum_{q\sim Q_2} |S''_q(\m,\n)|
\ll (MQ_2)^\eps Q_2^{1/2}
(MQ_2 + M^3Q_2^{1/3}).
\end{equation*}
\end{lemma}

\begin{proof}
We write  $q = q_1q_\Delta q_2$, where $q_1q_\Delta$ is square-free
and $\gcd(q_1q_\Delta,q_2)=1$, 
where 
$q_1\nmid \Delta(\m,\n)$, $q_\Delta\mid \Delta(\m,\n)$,
and $q_2$ is square-full. It now follows from  Lemma~\ref{S''-bounds} that 
\begin{equation*}
\begin{split}
\Sigma(Q_2,T,M)
&\ll \sum_{q_1q_\Delta q_2\sim Q_2}
\sum_{\substack{D(\m,\n)=0 \\
1\le |(\m,\n)|\le M \\ |\mathbf{t}|\sim T}}
q_1^{\eps-1}
q_\Delta^\eps
|S''_{q_2}(\m,\n)| \\
&\ll \sum_{q_\Delta q_2\ll Q_2}
\sum_{\substack{D(\m,\n)=0 \\
1\le |(\m,\n)|\le M \\ |\mathbf{t}|\sim T}}
\left(\frac{Q_2}{q_\Delta q_2}\right)^\eps
q_\Delta^\eps
|S''_{q_2}(\m,\n)| \\
&\ll \sum_{q_2\ll Q_2}
\sum_{\substack{D(\m,\n)=0 \\
1\le |(\m,\n)|\le M \\ |\mathbf{t}|\sim T}}
\left(\frac{Q_2}{q_2}\right)^\eps
M^\eps
|S''_{q_2}(\m,\n)|,
\end{split}
\end{equation*}
by the divisor bound applied to the integer $\Delta(\m,\n)\ne 0$.

We now bound $S''_{q_2}(\m,\n)$ using Lemma~\ref{S''-bounds}.
It follows from \eqref{eq:***} that the number of possible $q_2$, given $q$, is
$O((Q_2/q)^{1/2}q^\eps)$.
It follows that
\begin{equation}
\begin{split}
\label{yunnan-1}
\Sigma(Q_2,T,M)
&\ll (MQ_2)^\eps
\sum_{\kappa(q)^2\mid q\ll Q_2}
\sum_{\substack{D(\m,\n)=0 \\
1\le |(\m,\n)|\le M \\ |\mathbf{t}|\sim T}}
\left(\frac{Q_2}{q}\right)^{1/2}
\frac{|S_q(\m,\n)|}{q^4} \\
&\ll (MQ_2)^{2\eps} Q_2^{1/2} \Sigma',
\end{split}
\end{equation}
by Corollary~\ref{cor:basic},
where we let
\begin{equation*}
\Sigma'\defeq
\sum_{\kappa(q)^2\mid q\ll Q_2}
\sum_{\substack{D(\m,\n)=0 \\
1\le |(\m,\n)|\le M \\ |\mathbf{t}|\sim T}}
\frac{1}{q^{1/2}}
\prod_{1\le i\le 3} \{q,m_i\}^{1/2} \1_{d_i\mid n_i},
\end{equation*}
where $d_i = \{q,m_i\}^{1/2}$.
Although we could include the condition $\kappa(q)\mid \nabla{G}(\m,\n)$ here,
it turns out to be more convenient not to.
% Note: Thus we potentially lose a factor of $\kappa(q)$ when summing over $h$ and $\n$.

Let $d_0$ be the largest integer such that $d_0^2\mid q$.
Fix $i\in \{1,2,3\}$.
For each square-full $q$,
\begin{equation}
\label{0-swap-trick}
\sum_{\substack{n_i\ll M \\ m_i=0}} \{q,m_i\}^{1/2} \1_{d_i\mid n_i}
= \sum_{n_i\ll M} d_0 \1_{d_0\mid n_i}
\asymp d_0 + M.
\end{equation}
Moreover, the contribution to the left-hand side of \eqref{0-swap-trick}
from $n_i=0$ is exactly $= d_0$,
and the contribution from $n_i\ne 0$ is $\ll M$.
On the other hand,
\begin{equation*}
\sum_{\substack{0\ne m_i\ll M \\ n_i=0}} \{q,m_i\}^{1/2} \1_{d_i\mid n_i}
= \sum_{0\ne m_i\ll M} \{q,m_i\}^{1/2}
\ll \sum_{\textnormal{square-full }d\mid q} d^{1/2} \frac{M}{d}
\ll q^\eps M.
\end{equation*}
If $\Sigma'_S$ denotes the contribution to $\Sigma'$
from all $(\m,\n)$ satisfying a given condition $S$,
then it follows,
by bounding any sum over $(m_i,n_i)\in \ZZ\times 0$
in terms of the corresponding sum over $(m_i,n_i)\in 0\times \ZZ$,
that
\begin{equation*}
\Sigma'_{\m\n=\0}
\ll Q_2^\eps \Sigma'_{\m=\0},
\end{equation*}
and that for any fixed $i\in \{1,2,3\}$,
\begin{equation*}
\Sigma'_{m_in_i=0,\; m_jn_j^2=m_kn_k^2\ne 0}
\ll Q_2^\eps \Sigma'_{m_i=0,\; m_jn_j^2=m_kn_k^2\ne 0}.
% \le Q_2^\eps \Sigma'_{t_i=0,\; \m\ne\0,\; (\m,\n)\in \Lambda^\perp(\mathbf{t})}.
\end{equation*}
% Conceptually: If $m_i=0$, then $n_i=0 \Leftrightarrow \nabla{D}(\m,\n)=0$; the plane containing $m_i=0\ne n_i$ can be extended to $m_i=0=n_i$?
Therefore, we deduce that
\begin{equation}
\begin{split}
\label{yunnan-2}
\Sigma'
&\le \Sigma'_{m_1m_2m_3n_1n_2n_3\ne 0}
+ \Sigma'_{m_3n_3=0,\; m_1n_1^2=m_2n_2^2\ne 0}
+ \Sigma'_{\m\n=\0} \\
&\ll \Sigma'_{m_1m_2m_3n_1n_2n_3\ne 0}
+ Q_2^\eps \Sigma'_{m_3=0,\; m_1n_1^2=m_2n_2^2\ne 0}
+ Q_2^\eps \Sigma'_{\m=\0}.
\end{split}
\end{equation}

By our calculations for $m_i=0$ above in \eqref{0-swap-trick},
we have, since $(\m,\n)\ne \0$,
\begin{equation*}
\begin{split}
\Sigma'_{\m=\0}
&\ll \sum_{\kappa(q)^2\mid q\ll Q_2} \frac{1}{q^{1/2}}
(d_0^2 M + d_0 M^2 + M^3) \\
&\ll M Q_2
+ M^2 Q_2^{1/2}
+ M^3 Q_2^\eps,
\end{split}
\end{equation*}
since $d_0\le q^{1/2}$.

Turning to $\Sigma'_{m_1m_2m_3n_1n_2n_3\ne 0}$,
we write $(\m,\n)\in \Lambda^\perp(\mathbf{t})$ where $t_1t_2t_3\ne 0$,
by Lemma~\ref{non-coordinate-lattices}.
Evaluating the sum over $n_1n_2n_3\ne 0$
using Lemma~\ref{HB-lattice-bound}, we get
\begin{equation*}
\Sigma'_{m_1m_2m_3n_1n_2n_3\ne 0}\ll
\hspace{-0.3cm}
\sum_{\kappa(q)^2\mid q\ll Q_2}
\sum_{\substack{|\mathbf{t}|\sim T \\ t_1t_2t_3\ne 0}}
\sum_{\substack{\m=\mathbf{t}^2h \\
|\m|\le M \\ h\ne 0}}
\left(\frac{M}{|\dd|} + \frac{\gcd(d_1t_1,d_2t_2,d_3t_3)}{d_1d_2d_3} \frac{M^2}{T}\right)
\frac{d_1d_2d_3}{q^{1/2}},
\end{equation*}
where  $d_i=\{q,m_i\}^{1/2}.$ We have 
$d_3\leq q^{1/2}$ and  $(d_1d_2)^2=\{q,m_2\}\{q,m_2\}\mid \{q,m_1m_2\}^{2}$, whence $d_1d_2\mid \{q,m_1m_2\}$. Taking $|\dd|\geq (d_1d_2)^{1/2}$, 
the total contribution from the $\frac{M}{|\dd|}$ term is found to be 
\begin{align*}
\sum_{\kappa(q)^2\mid q\ll Q_2}
\sum_{\substack{|\mathbf{t}|\sim T \\ t_1t_2t_3\ne 0}}
\sum_{\substack{\m=\mathbf{t}^2h \\
|\m|\le M \\ h\ne 0}}
\frac{d_1d_2d_3 M}{|\dd| q^{1/2}}
&\ll 
M\sum_{\kappa(q)^2\mid q\ll Q_2}
\sum_{\substack{|\mathbf{t}|\sim T \\ t_1t_2t_3\ne 0}}
\sum_{\substack{\m=\mathbf{t}^2h \\
|\m|\le M \\ h\ne 0}}
\{q,m_1m_2\}^{1/2}\\
&\ll 
M
\sum_{\substack{|\mathbf{t}|\sim T \\ t_1t_2t_3\ne 0}}
\sum_{\substack{\m=\mathbf{t}^2h \\
|\m|\le M \\ h\ne 0}} 
\sum_{\substack{r\mid m_1m_2\\
\text{$r$ square-full}}} r^{1/2+\ve} \left(\frac{Q_2}{r}\right)^{1/2}\\
&\ll (Q_2M)^\ve M^2 Q_2^{1/2}T,
\end{align*}
by \eqref{eq:***}.
The remaining contribution is
\begin{equation*}
\ll \frac{M^2}{T} (TM)^{1+\eps} Q_2^{1/2-\delta}
= (TM)^\eps M^3 Q_2^{1/2-\delta},
\end{equation*}
by \eqref{generic-rankin}.

For $\Sigma'_{m_3=0,\; m_1n_1^2=m_2n_2^2\ne 0}$,
we write $(\m,\n)\in \Lambda^\perp(\mathbf{t})$ where $t_1t_2\ne 0=t_3$,
by Lemma~\ref{non-coordinate-lattices} and \eqref{Lambda-perp-t},
and evaluate the sum over $\n$ using Lemma~\ref{HB-lattice-bound},
getting
\begin{equation*}
\Sigma'_{m_3=0,\; m_1n_1^2=m_2n_2^2\ne 0}\ll
\sum_{\kappa(q)^2\mid q\ll Q_2}
\sum_{\substack{|\mathbf{t}|\sim T \\ t_1t_2\ne 0=t_3}}
\sum_{\substack{\m=\mathbf{t}^2h \\
|\m|\le M \\ h\ne 0}}
\frac{\gcd(d_1t_1,d_2t_2)}{d_1d_2d_3} \frac{Md_3+M^2}{T}
\frac{d_1d_2d_3}{q^{1/2}}.
\end{equation*}
Taking 
 $d_3\le q^{1/2}$, 
the total contribution from the $Md_3$ term is
\begin{equation*}
\ll \frac{M}{T} (TM)^{1+\eps} Q_2^{1/2+\eps},
\end{equation*}
by \eqref{nasty-rankin} in
Lemma \ref{rankin-stuff}.
Using \eqref{generic-rankin}, 
the remaining contribution is
\begin{equation*}
\ll \frac{M^2}{T} (TM)^{1+\eps} Q_2^{1/2-\delta}
= (TM)^\eps M^3 Q_2^{1/2-\delta}.
\end{equation*}
Since $\delta = 1/6$ is admissible in \eqref{generic-rankin},
it follows from \eqref{yunnan-1} and \eqref{yunnan-2} that
\begin{equation*}
\Sigma(Q_2,T,M)
\ll (MQ_2)^\eps Q_2^{1/2}
(MQ_2 + M^3Q_2^{1/3}
+ M^2T Q_2^{1/2}).
\end{equation*}
We may assume $M\gg T^2$, or else $\Sigma=0$.
But then $M^2T Q_2^{1/2}
\ll (MQ_2)^{1/4} (M^3Q_2^{1/3})^{3/4}$ and the statement of the lemma follows.
\end{proof}

\begin{lemma}
\label{rankin-stuff}
Let $d_i=\{q,m_i\}^{1/2}$ for $1\leq i\leq 3$.  
Let $T,M,Q_2\ge 1$.
If $\eps>0$ and $0<\delta<1/2$ then 
\begin{align}
\Psi_1 &\defeq \sum_{\kappa(q)^2\mid q\ll Q_2}
\sum_{\substack{|\mathbf{t}|\sim T
\\ \#\{i:t_i=0\}\le 1}}
\sum_{\substack{h\ge 1 \\
M\gg m_i=t_i^2h}}
\frac{\gcd(d_1t_1,d_2t_2,d_3t_3)}{q^{1/2}}
\ll (TM)^{1+\eps} Q_2^{1/2-\delta},
\label{generic-rankin} \\
\Psi_2 &\defeq \sum_{\kappa(q)^2\mid q\ll Q_2}
\sum_{\substack{|\mathbf{t}|\sim T
\\ t_1t_2\ne 0=t_3}}
\sum_{\substack{h\ge 1 \\
M\gg m_i=t_i^2h}}
\gcd(d_1t_1,d_2t_2)
\ll (TM)^{1+\eps} Q_2^{1/2+\eps}.
\label{nasty-rankin}
\end{align}
\end{lemma}

% \begin{proof}
% Since $d_3\le q^{1/2}$ and $d_1d_2\le d_1^2+d_2^2$, we have
% \begin{equation*}
% \Psi\ll \sum_{\kappa(q)^2\mid q\ll Q_2}
% \sum_{\substack{|\mathbf{t}|\sim T
% \\ t_1t_2t_3\ne 0}}
% \sum_{\substack{h\ge 1 \\
% M\gg m_i=t_i^2h}}
% \gcd(q,m_1)
% \end{equation*}
% we find by \eqref{gcd-1-average} that [doesn't work; this is over sq-full]
% \end{proof}

\begin{proof}
Taking $\delta = 1/2 - \eps$ and $t_3=0$ in \eqref{generic-rankin},
and multiplying by $Q_2^{1/2}$,
we obtain \eqref{nasty-rankin}.
Therefore, we need only prove \eqref{generic-rankin}.

\medskip

\noindent
{\em Proof for $t_1t_2t_3\ne 0$.}
% \VW{Given $\mathbf{t}$, try to sum over $q$ and $h$ using Rankin's trick?
% Or maybe sum over all variables at once using multiplicativity,
% assuming say $t_1t_2t_3\ne 0$?
% Reduce to Euler product.}
We have $h\ll M/T^2$.
Let $H = M/T^2$ and
$$(\alpha,\beta,\gamma)\defeq \left(\frac12-\delta,1+\eps,1+\eps\right).$$
By Rankin's trick we have
\begin{equation*}
\Psi_{1,t_1t_2t_3\ne 0}\ll
\sum_{\kappa(q)^2\mid q} \left(\frac{Q_2}{q}\right)^\alpha
\sum_{t_1,t_2,t_3\ge 1} \left(\frac{T^3}{t_1t_2t_3}\right)^\beta
\sum_{h\ge 1} \left(\frac{H}{h}\right)^\gamma
\frac{\gcd(d_1t_1,d_2t_2,d_3t_3)}{q^{1/2}},
\end{equation*}
which is $Q_2^\alpha T^{3\beta} H^\gamma$ times
an Euler product whose definition is independent of $Q_2,T,H$.
It remains to show that this Euler product is convergent.

If $v_p(q) = e\in \{0,2,3,4,\dots\}$, $v_p(t_i)=f_i\ge 0$, and $v_p(h)=g\ge 0$,
then $$r_i\defeq v_p(\{q,t_i^2h\}) = 2\floor{\min(e,2f_i+g)/2}.$$
Since $d_i^2 = \{q,t_i^2h\}$, we get
\begin{equation*}
v\defeq v_p\left(\frac{\gcd(d_1t_1,d_2t_2,d_3t_3)}{q^{1/2}}\right)
= \min_{1\leq i\leq 3}\left(\frac12r_i+f_i\right) - \frac12e
\end{equation*}
and
\begin{equation*}
\ell\defeq v_p(q^\alpha (t_1t_2t_3)^\beta h^\gamma)
= \alpha e + \gamma g + \sum_i \beta f_i.
\end{equation*}

% https://sagecell.sagemath.org
% a = 0.334
% b = 1.001
% c = 1.001
% box = 5
% report = 2
% for e in range(box):
%     for f1 in range(box):
%         for f2 in range(box):
%             for f3 in range(box):
%                 for g in range(box):
%                     r1 = min(e,2*f1+g)
%                     r2 = min(e,2*f2+g)
%                     r3 = min(e,2*f3+g)
%                     if r1==1:
%                         r1=0
%                     if r2==1:
%                         r2=0
%                     if r3==1:
%                         r3=0
%                     s1 = floor(r1/2)
%                     s2 = floor(r2/2)
%                     s3 = floor(r3/2)
%                     v = min(s1+f1,s2+f2,s3+f3) - (e/2) + (r1+r2+r3)/2 - (s1+s2+s3)
%                     l = a*e+c*g+b*(f1+f2+f3)
%                     if e+f1+f2+f3+g > 0 and l-v < report:
%                         print([e,f1,f2,f3,g],l-v)
% [0, 0, 0, 0, 1] 1.00100000000000
% [0, 0, 0, 1, 0] 1.00100000000000
% [0, 0, 1, 0, 0] 1.00100000000000
% [0, 1, 0, 0, 0] 1.00100000000000
% [1, 0, 0, 0, 0] 0.834000000000000
% [1, 0, 0, 0, 1] 1.83500000000000
% [1, 0, 0, 1, 0] 1.83500000000000
% [1, 0, 1, 0, 0] 1.83500000000000
% [1, 1, 0, 0, 0] 1.83500000000000
% [2, 0, 0, 0, 0] 1.66800000000000

We begin with a clean general estimate.
Clearly
\begin{equation*}
v\le \frac13\sum_i(\frac12r_i+f_i) - \frac12e
\le \frac13\sum_i(2f_i+\frac12g) - \frac12e
= \sum_i \frac23f_i + \frac12g - \frac12e,
\end{equation*}
since $r_i\le 2f_i+g$.
Since $\beta,\gamma\ge 1$, it follows that
\begin{equation*}
\ell-v\ge (\alpha+\frac12)e + \frac12g + \sum_i \frac13f_i.
\end{equation*}
If $e\ge 2$ then $\ell-v\ge 1+\eps$,
since $\alpha>0$.
In general, it is also clear that if $L\ge 1$ then
$$\#\{(e,f,g): \ell-v\le L\} \ll L^5.$$
It remains to show that if $e=0$ and $(f,g)\ne \0$, then $\ell-v\ge 1+\eps$.
But if $e=0$,
then $r_i=0$,
so $v = \min_i f_i$.
Since $\beta\ge 1$, it follows that
$$\ell-v\ge \gamma g + \beta \max_i(f_i).$$
Thus $\ell-v\ge 1+\eps$, because $(f,g)\ne \0$.

\medskip

\noindent
{\em Proof for $t_1t_2t_3\ne 0$.}
In this case,
\begin{equation*}
\Psi_{1,t_3=0} \ll \sum_{\kappa(q)^2\mid q\ll Q_2}
\sum_{1\le t_1,t_2\ll T}
\sum_{h\ll M/T^2}
\frac{\gcd(d_1t_1,d_2t_2)}{q^{1/2}}.
\end{equation*}
Let $H = M/T^2$.
We will take
$$(\alpha,\beta,\gamma)\defeq \left(\frac12-\delta,\frac32+\eps,1+\eps\right).$$
(In fact the argument also goes through with  $\beta=1+\eps$, but there is no advantage in doing so.)
By Rankin's trick we have
\begin{equation*}
\Psi_{1, t_3=0}\ll
\sum_{\kappa(q)^2\mid q} \left(\frac{Q_2}{q}\right)^\alpha
\sum_{t_1,t_2\ge 1} \left(\frac{T^2}{t_1t_2}\right)^\beta
\sum_{h\ge 1} \left(\frac{H}{h}\right)^\gamma
\frac{\gcd(d_1t_1,d_2t_2)}{q^{1/2}}.
\end{equation*}

If $v_p(q) = e\in \{0,2,3,4,\dots\}$, $v_p(t_i)=f_i\ge 0$, and $v_p(h)=g\ge 0$,
then $$r_i\defeq v_p(\{q,t_i^2h\}) = 2\floor{\min(e,2f_i+g)/2}.$$
as before.
% Here $f_3=\infty$, and $r_3 = 2\floor{e/2}$.
Thus 
\begin{equation*}
v\defeq v_p\left(\frac{\gcd(d_1t_1,d_2t_2)}{q^{1/2}}\right)
= \min_{1\le i\le 2}\left(\frac12r_i+f_i\right) - \frac12e
\end{equation*}
and
\begin{equation*}
\ell\defeq v_p(q^\alpha (t_1t_2)^\beta h^\gamma)
= \alpha e + \gamma g + \sum_{1\le i\le 2} \beta f_i.
\end{equation*}

% a = 0.251
% b = 1.501
% c = 1.001
% box = 5
% report = 2
% for e in range(box):
%     for f1 in range(box):
%         for f2 in range(box):
%             for g in range(box):
%                 r1 = min(e,2*f1+g)
%                 r2 = min(e,2*f2+g)
%                 if r1==1:
%                     r1=0
%                 if r2==1:
%                     r2=0
%                 s1 = floor(r1/2)
%                 s2 = floor(r2/2)
%                 v = min(s1+f1,s2+f2) - floor(e/2) + (r1+r2)/2 - (s1+s2)
%                 l = a*e+c*g+b*(f1+f2)
%                 if e+f1+f2+g > 0 and l-v < report:
%                     print([e,f1,f2,g],l-v)
% [0, 0, 0, 1] 1.00100000000000
% [0, 0, 1, 0] 1.50100000000000
% [0, 1, 0, 0] 1.50100000000000
% [1, 0, 0, 0] 0.251000000000000
% [1, 0, 0, 1] 1.25200000000000
% [1, 0, 1, 0] 1.75200000000000
% [1, 1, 0, 0] 1.75200000000000
% [2, 0, 0, 0] 1.50200000000000
% [3, 0, 0, 0] 1.75300000000000

Observe that
\begin{equation*}
v \le \frac12 \sum_{1\le i\le 2}\left(\frac12r_i+f_i\right) - \frac12e
\le \frac12 \sum_{1\le i\le 2}\left(2f_i+\frac12g\right) - \frac12e
= \sum_{1\le i\le 2} f_i + \frac12 g - \frac12 e,
\end{equation*}
so
\begin{equation*}
\ell-v\ge \left(\alpha+\frac12\right)e
+ \left(\gamma-\frac12\right) g + \sum_{1\le i\le 2} (\beta-1)f_i.
\end{equation*}
If $e\ge 2$, then $\ell-v\ge 1+\eps$,
since $\alpha>0$.
Alternatively, if $e=0$, then $r_i=0$, so $v=\min_i{f_i}$ and
$\ell-v\ge \gamma g + \beta \max_i(f_i)$, since $\beta\ge 1$.
As in the case $t_1t_2t_3\ne 0$, this suffices.
\end{proof}

\subsection*{Uniform upper bounds}

We begin by recalling a classical result from the geometry of numbers.
\begin{lemma}
% [Mahler/Remak, 1938; Weyl, 1942; van der Waerden, 1956]
\label{shortest-basis}
Let $v_1,\dots,v_N$ be a \emph{shortest} 
basis of a lattice in a Euclidean space;
i.e.~a basis for which $\max(|v_1|,\dots,|v_N|)$ is minimal.
Then
$$\max(|v_1|,\dots,|v_N|) \ll_N \lambda_N,
$$
where $\lambda_N$ is the largest successive minimum of the lattice.
\end{lemma}

\begin{proof}
% For example, a greedy (Minkowski) basis suffices.
% https://arxiv.org/pdf/2106.03183
See Cassels \cite{cassels}*{p.~135, Lemma~8}.
\end{proof}

The following lemma is proved using the geometry of numbers and 
will help us remove certain akward factors of $M^\eps$.
\begin{lemma}
\label{stable-ranges}
Let $A>0$ be a large absolute constant.
Let $\a,\b\in (\ZZ/q_2\ZZ)^3$ and $c\in \ZZ/q_2\ZZ$.
Let $N(T,H)=N(T,H;\a,\b,c,q_2)$ be the number of tuples
$$(\mathbf{t},h,\n)\in \ZZ^3\times \ZZ\times \ZZ^3,
\; (\mathbf{t},h,\n)\equiv (\a,c,\b)\bmod{q_2}$$
with $\mathbf{t}$ primitive, $h\ne 0$, and $\n.\mathbf{t}=0$
such that $|\mathbf{t}|\sim T$, $h\ll H$, and $n_i\ll T^2H$.
% In practice, we have in mind $H = M/T^2$.
\begin{enumerate}
\item If $T>0$ and $H\ge Aq_2$,
then $$\frac{N(T,H)}{H (T^2H)^2} \asymp \frac{N(T,Aq_2)}{Aq_2 (T^2Aq_2)^2}.$$

\item If $0<H\le Aq_2$ and $T\ge Aq_2^2$,
then $$\frac{N(T,H)}{T^2 (T^2H)^2} \asymp \frac{N(Aq_2^2,H)}{(Aq_2^2)^2 ((Aq_2^2)^2H)^2}.$$
\end{enumerate}
\end{lemma}

\begin{proof}
(1):
We proceed to estimate $N(T,H)$. 
For fixed $\mathbf{t}$, 
the  successive minima of
the lattice $\Lambda=\{\n\in q_2\ZZ^3: \n. \mathbf{t}=0\}$
are $\ll q_2T\le T^2H/A$. Moreover, if the set $\{\n\in \ZZ^3: \n\equiv \b\bmod{q_2}: 
\n.\mathbf{t}=0\}$ is nonempty
then it contains a point inside a fundamental domain of $(\Lambda\otimes\RR)/\Lambda$
generated by a shortest basis of $\Lambda$,
which can be bounded using Lemma~\ref{shortest-basis}.
It follows that 
the number of available choices for $\n$ is $\asymp C(q_2,\mathbf{t},\b) (T^2H)^2$
by the geometry of numbers.
Assume that 
$H\ge Aq_2$. Then, on summing over all available choices for $\mathbf{t}$,
and 
$\asymp H/q_2$
available choices for $h$, 
it follows that 
\begin{equation*}
N(T,H) \asymp \sum_{\mathbf{t}} (H/q_2) C(q_2,\mathbf{t},\b) (T^2H)^2.
\end{equation*}
Comparing this quantity to its value at $H=Aq_2$ gives (1).

(2):
Given $\mathbf{t}$,
the number of available choices for $\n$ is again $\asymp C(q_2,\mathbf{t},\b) (T^2H)^2$
by the geometry of numbers, since $q_2T\le T^2/A\le T^2H/A$.
Moreover, $C(q_2,\mathbf{t},\b)\asymp C'(q_2,\a,\b)/T$,
since the set 
$\{\n\in \ZZ^3: \n\equiv \b\bmod{q_2}: 
\n.\mathbf{t}=0\}$
is nonempty
if and only if there exists $\n'\in \ZZ^3$ with $\sum_i (q_2n'_i+b_i)t_i = 0$,
which is equivalent to asking that $q_2\gcd(t_1,t_2,t_3)\mid \mathbf{b}.\mathbf{t}$.
But $\mathbf{t}$ is primitive and $\mathbf{b}.\mathbf{t}\equiv \mathbf{b}.\mathbf{a} \bmod{q_2}$, whence
\begin{equation*}
N(T,H) \asymp \sum_{h,\mathbf{t}} \frac{C'(q_2,\a,\b)}{T} (T^2H)^2
= \frac{C'(q_2,\a,\b)}{T} (T^2H)^2 (\#h)(\#\mathbf{t}).
\end{equation*}
The number of primitive vectors $\mathbf{t}\equiv \a\bmod{q_2}$
of height $\sim T$
is $\asymp C''(q_2,\a) T^3$,
since $T\ge Aq_2^2$.
If $\gcd(\a,q_2)\ne 1$ then no such $\mathbf{t}$
exist, whereas if $\gcd(\a,q_2)=1$ then M\"{o}bius inversion
yields a uniform estimate
$$
 M(T,q_2) + O\left(\left(\frac{T}{q_2}\right)^2 + T\right)
$$
for the number of $\mathbf{t}$, 
with main term
$$M(T,q_2) \asymp \left(\frac{T}{q_2}\right)^3 \prod_{p\nmid q_2} (1-p^{-3})
\asymp \left(\frac{T}{q_2}\right)^3.$$
Comparing $N(T,H)$ to its value at $T=Aq_2^2$ gives (2).
\end{proof}

The next stage of our argument is devoted to  a study of  the quantity
\begin{equation}
\label{clean-Q1Q2-hybrid}
\Sigma(I,Q_2,T,M)
\defeq \sum_{\substack{D(\m,\n)=0 \\
1\le |(\m,\n)|\le M \\ |\mathbf{t}|\sim T}}
\left|\sum_{q_1\in I} \Xi_{q_1}(\m,\n)\right|^2
\sum_{q_2\sim Q_2} |S''_{q_2}(\m,\n)|,
\end{equation}
for  $Q_2,T,M\ge 1$ and an interval $I\belongs \{q_1\sim Q_1\}$, where $Q_1\ge 1$.

\begin{lemma}
\label{q_0-stability}
There exists an absolute constant $A>16$ such that 
\begin{equation*}
\frac{\Sigma(I,Q_2,T,M)}{M^3}
\ll \frac{\Sigma(I,Q_2,T',M' )}
{(M'/A^2)^3}
+ 1,
\end{equation*}
where if $R=(AQ_1Q_2)^A$ then
\begin{equation}\label{eq:TM'}
T' = \min(T,R) , \quad 
M' = A^2\min(T,R)^2\min(M/T^2,R).
\end{equation}
\end{lemma}

Before establishing this result it will be convenient to establish a simpler variant in which
$I$ is taken to be $\{1\}$,
$S''_q(\m,\n)$ is replaced by $S'_q(\m,\n)$,
the vectors $(\m,\n)$ are constrained,
and the old definition of $\mathbf{t}$ in terms of $(\m,\n)$ is relinquished.
Let
\begin{equation*}
\Sigma'(Q_2,T,M)
\defeq \sum_{\substack{[\mathbf{t}]\in \PP^2(\ZZ) \\ |\mathbf{t}|\sim T}}
\sum_{\substack{(\m,\n)\in \Lambda^\perp(\mathbf{t}) \\
\m\ne\0 \\ 1\le |(\m,\n)|\le M}}
\sum_{q_2\sim Q_2} |S'_{q_2}(\m,\n)|.
\end{equation*}
The triples $(\mathbf{t},\m,\n)$ occurring in $\Sigma'$
can thus be parameterised in terms of the triples $(\mathbf{t},h,\n)$
appearing in Lemma~\ref{stable-ranges}.

\begin{lemma}
\label{reduce-T,H}
Let $A\geq 1$ be a large absolute constant and 
let $Q_2,T,M\geq 1$.
Then
\begin{equation*}
\frac{\Sigma'(Q_2,T,A^\tau M)}{M^3}
\bowtie_\tau \frac{\Sigma'(Q_2,\min(T,AQ_2^2),\min(T,AQ_2^2)^2\min(M/T^2,AQ_2))}
{(\min(T,AQ_2^2)^2\min(M/T^2,AQ_2))^3}
\end{equation*}
for all $\tau=\pm 1$,
where $\bowtie_1$ denotes $\gg$
and where $\bowtie_{-1}$ denotes $\ll$.
\end{lemma}

\begin{proof}
Since $|\mathbf{t}|\sim T$ and $\m=\mathbf{t}^2h$,
there exist implications of the form
\begin{equation*}
|h|\le A^{-1/4} M/T^2
\Rightarrow |\m|\le M
\Rightarrow |h|\le A^{1/4} M/T^2.
\end{equation*}
Let $H=M/T^2$.
We proceed by applying Lemma~\ref{stable-ranges}
in each residue class $(\mathbf{t},h,\n)\bmod{q_2}$,
which fixes the value of $S'_{q_2}(\m,\n)$ by Lemma~\ref{q-invariance},
and then summing over all residue classes.
By Lemma~\ref{stable-ranges}(1) if $H\ge AQ_2$, and trivially otherwise, we get
\begin{equation*}
\frac{\Sigma'(Q_2,T,A^\tau M)}{M^3/T^2}
= \frac{\Sigma'(Q_2,T,A^\tau T^2H)}{H(T^2H)^2}
\bowtie_\tau \frac{\Sigma'(Q_2,T,A^{\tau/2} T^2\min(H,AQ_2))}{\min(H,AQ_2)(T^2\min(H,AQ_2))^2}.
\end{equation*}
By Lemma~\ref{stable-ranges}(2) with $H=\min(H,AQ_2)$
if $T\ge AQ_2^2$, and trivially otherwise, we have
\begin{equation*}
\frac{\Sigma'(Q_2,T,A^{\tau/2} T^2\min(H,AQ_2))}
{T^2(T^2\min(H,AQ_2))^2}
\bowtie_\tau \frac{\Sigma'(Q_2,\min(T,AQ_2^2),\min(T,AQ_2^2)^2\min(H,AQ_2))}
{\min(T,AQ_2^2)^2(\min(T,AQ_2^2)^2\min(H,AQ_2))^2}.
\end{equation*}
Plugging the last display into the one before it, we get the desired result.
\end{proof}

\begin{proof}
[Proof of Lemma~\ref{q_0-stability}]
We may assume $M\gg T^2$,
or else 
Lemma~\ref{degen-counts} implies that
$\Sigma(I,Q_2,T,M)=0$.
For $T',M'$ defined in \eqref{eq:TM'}, 
with  $R=(AQ_1Q_2)^A$,
we have 
 $T'  \le T$
and $M' \le A^2M$.
If $\max(T,M/T^2)\le R$,
then $T'=T$ and $M'=A^2M$,
so the result is trivial.
Thus we may suppose that $\max(T,M/T^2)\ge R$.
In particular, since $M/T^2\gg 1$ we have
$$M = T^2(M/T^2)\gg R,
\qquad M' \gg A^2 (T')^2,
\qquad M' \gg A^2 R.$$
We begin by writing
\begin{equation}
\label{hybrid-strat}
\Sigma
\asymp \Sigma_{\m\ne\0,\; n_1n_2n_3\ne 0}
+ \Sigma_{n_3=0,\; m_1m_2m_3n_1n_2\ne 0}
+ \Sigma_{m_3=n_3=0,\; m_1m_2n_1n_2\ne 0}
+ \Sigma_{\m\n=\0}.
\end{equation}

By Lemma~\ref{reduce-T,H}, since $A^{-1}M'>0$ we have
\begin{equation*}
\frac{\Sigma'(Q_2,T',M')}{(M')^3}
\gg \frac{\Sigma'(Q_2,\min(T',AQ_2^2),\min(T',AQ_2^2)^2\min(A^{-1}M'/(T')^2,AQ_2))}
{(\min(T',AQ_2^2)^2\min(A^{-1}M'/(T')^2,AQ_2))^3}.
\end{equation*}
Similarly,  since $AM>0$ we also have
\begin{equation*}
\frac{\Sigma'(Q_2,T,M)}{M^3}
\ll \frac{\Sigma'(Q_2,\min(T,AQ_2^2),\min(T,AQ_2^2)^2\min(AM/T^2,AQ_2))}
{(\min(T,AQ_2^2)^2\min(AM/T^2,AQ_2))^3}.
\end{equation*}
We have $\min(T',AQ_2^2) = \min(T,R,AQ_2^2) = \min(T,AQ_2^2)$
and 
\begin{align*}
\min(A^{-1}M'/(T')^2,AQ_2) = \min(A\min(M/T^2,R),AQ_2)
&= A\min(M/T^2,R,Q_2)\\
&= \min(AM/T^2,AQ_2).
\end{align*}
Thus
\begin{equation}
\label{generic-stability}
\frac{\Sigma'(Q_2,T',M')}{(M')^3}
\gg \frac{\Sigma'(Q_2,T,M)}{M^3}.
\end{equation}
% We have
% \begin{equation*}
% \begin{split}
% \Sigma_{\m\ne\0,\; n_1n_2n_3\ne 0}(I,Q_2,T,M)
% &= \Sigma_{\m\ne\0,\; n_1n_2n_3\ne 0}(\{1\},Q_2,T,M) \1_{1\in I} \\
% &= \Sigma_{(\m,\n)\in \Lambda^\perp(\mathbf{t}),\; \m\ne\0}
% - \Sigma_{(\m,\n)\in \Lambda^\perp(\mathbf{t}),\; n_1n_2n_3=0,\; \m\ne\0}
% \end{split}
% \end{equation*}
But if $n_1n_2n_3\ne 0$,
then $\Xi_q(\m,\n)=\1_{q=1}$ by \eqref{degen-char},
whence $S''_q(\m,\n) = S'_q(\m,\n)$.
Thus
\begin{equation*}
\begin{split}
\Sigma_{\m\ne\0,\; n_1n_2n_3\ne 0}(I,Q_2,T,M)
&= \Sigma_{\m\ne\0,\; n_1n_2n_3\ne 0}(\{1\},Q_2,T,M) \1_{1\in I} \\
&\le \Sigma'(Q_2,T,M) \1_{1\in I}
\end{split}
\end{equation*}
by Lemma~\ref{non-coordinate-lattices}.
Similarly,
\begin{equation*}
\Sigma'(Q_2,T',M') \1_{1\in I}
= \Sigma_{\m\ne\0,\; n_1n_2n_3\ne 0}(I,Q_2,T',M')
+ \Sigma'_{n_1n_2n_3=0}(Q_2,T',M') \1_{1\in I}.
\end{equation*}
Let us first estimate the second term on the right.
By symmetry,
$$\Sigma'_{n_1n_2n_3=0}
\ll \Sigma'_{n_3=0}
\le \Sigma'_{n_3=m_1=m_2=0}
+ \Sigma'_{n_3=0,\; (m_1,m_2)\ne\0}.$$
If $n_3=m_1=m_2=0$,
then \eqref{degen-char} implies $\Xi_q(\m,\n)=\1_{q=1}$
and $S'_q(\m,\n) = S''_q(\m,\n)$.
Any pair $(\m,\n)$ appears in $\Sigma'$ for at most four choices of $[\mathbf{t}]$,
since $[\mathbf{t}^2] = [\m]$.
Therefore,
\begin{align*}
\Sigma'_{n_3=m_1=m_2=0}(Q_2,T',M') \1_{1\in I}
&\ll \Sigma_{n_3=m_1=m_2=0}(\{1\},Q_2,T',M') \1_{1\in I} \\
&= \Sigma_{n_3=m_1=m_2=0}(I,Q_2,T',M')\\
&\le \Sigma_{\m\n=\0}(I,Q_2,T',M').
\end{align*}
We deduce that
\begin{align*}
\Sigma'(Q_2,T',M') \1_{1\in I}
\ll~& (\Sigma_{\m\ne\0,\; n_1n_2n_3\ne 0}+\Sigma_{\m\n=\0})(I,Q_2,T',M')\\
&\qquad + \Sigma'_{n_3=0,\; (m_1,m_2)\ne\0}(Q_2,T',M') \\
\ll~& \Sigma(I,Q_2,T',M')
+ \Sigma'_{n_3=0,\; (m_1,m_2)\ne\0}(Q_2,T',M'),
\end{align*}
where in the last step we apply \eqref{hybrid-strat}.
However, 
on applying the second part of Lemma~\ref{HB-lattice-bound} with $d_1=d_2=d_3=1$, 
the number of triples $(\mathbf{t},\m,\n)$ counted in the second sum on the right is
% say $\m=\mathbf{t}^2h$, $n_1t_1+n_2t_2=0$, $n_3=0$
\begin{equation}
\label{gcd-sum}
\ll \sum_{\substack{t_1,t_2,t_3\ll T' \\ (t_1,t_2)\ne\0}}
\frac{M'}{(T')^2} \frac{M'\gcd(t_1,t_2)}{|(t_1,t_2)|}
\ll \sum_{1\ll U\ll T'} U^2 \frac{T'}{U} T'\frac{M'}{(T')^2} \frac{M'}{U}
\ll (M')^2 (T')^\eps,
% \ll (M')^{2+\eps},
\end{equation}
where $U$ is a dyadic parameter,
since
there are $O(U^2)$ choices of  primitive $(u_1,u_2)\in \ZZ^2$ 
such that $|(u_1,u_2)|\sim U$, there are $O(T')$ choices for $t_3$, and the number of choices for $(t_1,t_2)$ is $U^2 \cdot T'/U$ (because there are $O(T'/U)$ choices for 
$\gcd(t_1,t_2)$ associated to a primitive representative $(u_1,u_2)$ of order  $U$).
Thus, in view of \eqref{generic-stability}
and the bound $S'_{q_2}(\m,\n)\ll q_2^{3/2+\eps}$ from Remark~\ref{uniform-pointwise-bound},
we deduce that
\begin{equation*}
\frac{\Sigma_{\m\ne\0,\; n_1n_2n_3\ne 0}(I,Q_2,T,M)}{M^3}
\ll \frac{\Sigma(I,Q_2,T',M')}{(M')^3}
+ \frac{Q_2^{5/2+\eps}}{(A^2R)^{1-\eps}}.
\end{equation*}

By \eqref{degen-char}, the sum $\Sigma_{m_3=n_3=0,\; m_1m_2n_1n_2\ne 0}(I,Q_2,T,M)$,
is equal to
\begin{equation*}
\Sigma_{m_3=n_3=0,\; m_1m_2n_1n_2\ne 0}(\{1\},Q_2,T,M) \1_{1\in I}
\ll M^2 Q_2^{5/2+\eps},
\end{equation*}
since the number of pairs $(\m,\n)$ counted on the left hand side is $\ll M^2/T \ll M^2$ by the proof of the $m_3n_3 = 0$ case of Lemma~\ref{degen-counts}.

Similarly, by Lemma~\ref{degen-counts} and the trivial bound $\Xi_q\ll q^\eps$, we have
\begin{align*}
\Sigma_{n_3=0,\; m_1m_2m_3n_1n_2\ne 0}(I,Q_2,T,M) \1_{T\ge R^{1/2}}
&\ll \frac{M^3}{T} Q_1^{2+\eps} Q_2^{5/2+\eps} \1_{T\ge R^{1/2}}\\
&\ll \frac{M^3 (Q_1Q_2)^{5/2+\eps}}{R^{1/2}}.
\end{align*}
On the other hand, on 
writing $(m_1,m_2)=(t_1^2,t_2^2)h$
and $(n_1,n_2)=(t_2,-t_1)b$ and
using Lemma~\ref{non-coordinate-lattices},
we have by expanding the square in \eqref{clean-Q1Q2-hybrid}
\begin{equation*}
\Sigma_{n_3=0,\; m_1m_2m_3n_1n_2\ne 0}(I,Q_2,T,M)
= \sum_{\substack{[\mathbf{t}]\in \PP^2(\ZZ) \\ |\mathbf{t}|\sim T \\ t_1t_2\ne 0=t_3}}
\sum_{\substack{1\le |m_3|\le M
\\ 1\le |h|\le M/|\mathbf{t}^2| \\ 1\le |b|\le M/|\mathbf{t}|}}
\sum_{q_1,q'_1\in I} \Xi_{q_1}\Xi_{q'_1}
\sum_{q_2\sim Q_2} |S''_{q_2}|.
\end{equation*}
Given $\mathbf{t}$, the number of vectors $(m_3,h,b)\in \ZZ^3$
lying in a given residue class modulo $q_1q'_1q_2$ is
\begin{equation*}
\left(\frac{M}{q_1q'_1q_2} + O(1)\right)
\left(\frac{M/|\mathbf{t}^2|}{q_1q'_1q_2} + O(1)\right)
\left(\frac{M/|\mathbf{t}|}{q_1q'_1q_2} + O(1)\right)
= \frac{1}{(q_1q'_1q_2)^3} \frac{M^3}{|\mathbf{t}^2||\mathbf{t}|}
+ O\left(\frac{M^2}{T}\right),
\end{equation*}
since $M/T\gg 1$.
Summing over residue classes,
using Lemma~\ref{q-invariance},
we conclude that
\begin{align*}
&\Sigma_{n_3=0,\; m_1m_2m_3n_1n_2\ne 0}(I,Q_2,T,M)\\
&\qquad= \sum_{\substack{[\mathbf{t}]\in \PP^2(\ZZ) \\ |\mathbf{t}|\sim T \\ t_1t_2\ne 0=t_3}}
\sum_{\substack{q_1,q'_1\in I \\ q_2\sim Q_2
\\ (m_3,h,b)\bmod{q_1q'_1q_2}}}
\left(\frac{\Xi_{q_1}\Xi_{q'_1} |S''_{q_2}|}{(q_1q'_1q_2)^3}
\frac{M^3}{|\mathbf{t}^2||\mathbf{t}|}
+ O\left(q_1^\eps q_2^{3/2+\eps} \frac{M^2}{T}\right)\right).
\end{align*}
If $T\le R^{1/2}$, then $T'=T$ and we may  cancel out main terms for $M$ and $M'$ to get
% Note that M\gg T^2 and M'\gg (T')^2 = T^2.
\begin{align*}
\frac{\Sigma_{n_3=0,\; m_1m_2m_3n_1n_2\ne 0}(I,Q_2,T,M)}{M^3}
&- \frac{\Sigma_{n_3=0,\; m_1m_2m_3n_1n_2\ne 0}(I,Q_2,T',M')}{(M')^3}\\
&\ll \frac{T (Q_1^2Q_2)^{4+\eps} Q_2^{3/2}}{\min(M,M')}\\
&\ll \frac{ (Q_1Q_2)^{8+\eps} }{R^{1/2}},
\end{align*}
since $\min(M,M')/T\gg R/T\geq R^{1/2}$.
Together with our analysis for $T\ge R^{1/2}$, this implies that 
\begin{align*}
\frac{\Sigma_{n_3=0,\; m_1m_2m_3n_1n_2\ne 0}(I,Q_2,T,M)}{M^3}
\ll~& \frac{\Sigma_{n_3=0,\; m_1m_2m_3n_1n_2\ne 0}(I,Q_2,T',M')}{(M')^3}\\
&\quad + \frac{(Q_1Q_2)^{8+\eps}}{R^{1/2}},
\end{align*}
for all $T$.

Recall our convention that $[\mathbf{t}]=[1:1:1]$ if $\m\n = \0$, which implies in particular that 
$T \asymp 1$ in this case. Bearing this  in mind, 
we have
$$\Sigma_{\m\n=\0}(I,Q_2,T,M)
= \sum_{J_1\cup J_2 = \{1,2,3\}}
\Sigma_{i\in J_1\Leftrightarrow m_i=0,\; j\in J_2\Leftrightarrow n_j=0}(I,Q_2,T,M),$$
where $J_1,J_2\subseteq \{1,2,3\}$ are sets.
If $\#J_1+\#J_2\ge 4$, then
\begin{equation*}
\Sigma_{i\in J_1\Leftrightarrow m_i=0,\; j\in J_2\Leftrightarrow n_j=0}(I,Q_2,T,M)
\ll M^2 Q_1^{2+\eps} Q_2^{5/2+\eps}.
\end{equation*}
If $\#J_1+\#J_2=3$ and $\Sigma_{i\in J_1\Leftrightarrow m_i=0,\; j\in J_2\Leftrightarrow n_j=0}(I,Q_2,T,M)\ne 0$,
then $T\asymp 1$, so arguing as we did for
$\Sigma_{n_3=0,\; m_1m_2m_3n_1n_2\ne 0}(I,Q_2,T,M)$ when $T\le R^{1/2}$,
we obtain the estimate
\begin{align*}
\frac{\Sigma_{i\in J_1\Leftrightarrow m_i=0,\; j\in J_2\Leftrightarrow n_j=0}(I,Q_2,T,M)}{M^3}
&- \frac{\Sigma_{i\in J_1\Leftrightarrow m_i=0,\; j\in J_2\Leftrightarrow n_j=0}(I,Q_2,T',M')}{(M')^3}\\
&
\ll \frac{(Q_1^2Q_2)^{4+\eps} Q_2^{3/2}}{\min(M,M')}.
\end{align*}
In any case, it follows that
\begin{equation*}
\frac{\Sigma_{\m\n=\0}(I,Q_2,T,M)}{M^3}
\ll \frac{\Sigma_{\m\n=\0}(I,Q_2,T',M')}{(M')^3}
+ \frac{(Q_1Q_2)^{8+\eps}}{R}.
\end{equation*}
Plugging our work into \eqref{hybrid-strat}
yields Lemma~\ref{q_0-stability},
provided $A>16$.
\end{proof}

Let $0<\delta<1$ be fixed.
it follows from  Remark~\ref{uniform-pointwise-bound}
that 
$$
|S''_{q_2}(\m,\n)| = |S''_{q_2}(\m,\n)|^\delta |S''_{q_2}(\m,\n)|^{1-\delta} \ll (q_2^{3/2 + \ve})^\delta |S''_{q_2}(\m,\n)|^{1-\delta}.
$$ 
Hence, 
in the notation of
\eqref{clean-Q1Q2-hybrid},
it follows from 
H\"{o}lder's inequality between Lemma~\ref{first-Q1-bound} (with $A=2/\delta$) and Lemma~\ref{first-Q2-bound} that
\begin{equation*}
\begin{split}
\frac{\Sigma(I,Q_2,T,M)}{(Q_2^{3/2+\eps})^\delta}
&\ll \sum_{\substack{D(\m,\n)=0 \\
1\le |(\m,\n)|\le M \\ |\mathbf{t}|\sim T}}
\left|\sum_{q_1\in I} \Xi_{q_1}(\m,\n)\right|^2
\sum_{q_2\sim Q_2} |S''_{q_2}(\m,\n)|^{1-\delta} \\
&\ll \left((MQ_1)^\eps M^3 Q_1^{2/\delta}
\left(\frac{1}{Q_1^{1/6}} + \frac{1}{M}\right)\right)^\delta\\
&\quad \times
\left((MQ_2)^\eps Q_2^{1/2}
(MQ_2 + M^3Q_2^{1/3})\right)^{1-\delta}.
\end{split}
\end{equation*}
Thus
\begin{equation*}
\frac{\Sigma(I,Q_2,T,M)}{M^3 Q_1^2 Q_2}
\ll (MQ_1Q_2)^\eps \left(\frac{1}{Q_1^{\delta/6}} + \frac{1}{M^\delta}\right)
\left(\frac{Q_2^{1/2}}{M^{2(1-\delta)}} + \frac{1}{Q_2^{(1-4\delta)/6}}\right).
\end{equation*}
% (3/2) delta + (3/2)(1-delta) = 3/2.
% 3\delta/2 + (5/6)(1-\delta) = (4\delta+5)/6.
Replacing $(T,M)$ with $(T',M')$ and applying Lemma~\ref{q_0-stability}, we get
\begin{equation*}
\frac{\Sigma(I,Q_2,T,M)}{M^3 Q_1^2 Q_2}
\ll (Q_1Q_2)^\eps \left(\frac{1}{Q_1^{\delta/6}} + \frac{1}{(M')^\delta}\right)
\left(\frac{Q_2^{1/2}}{(M')^{2(1-\delta)}} + \frac{1}{Q_2^{(1-4\delta)/6}}\right)
+ \frac{1}{Q_1^2Q_2},
\end{equation*}
where $M' \asymp \min(T,R)^2\min(M/T^2,R) = \min(M,T^2R,R^2M/T^2,R^3)$.
We observe that  $R\ge (Q_1Q_2)^{16}$. Moreover, 
we may assume $M\gg T^2$,  by Lemma~\ref{degen-counts}, whence either $M' \asymp M$ or $M'\gg R$.
Thus
\begin{equation*}
\frac{\Sigma(I,Q_2,T,M)}{M^3 Q_1^2 Q_2}
\ll (Q_1Q_2)^\eps \left(\frac{1}{Q_1^{\delta/6}} + \frac{1}{M^\delta}\right)
\left(\frac{Q_2^{1/2}}{M^{2(1-\delta)}} + \frac{1}{Q_2^{(1-4\delta)/6}}\right),
\end{equation*}
since $R^\delta \gg Q_1^{\delta/6}$,
 $R^{2(1-\delta)} \gg Q_2^{2(1-\delta)/3} = Q_2^{1/2 + (1-4\delta)/6}$,
and $Q_1^2 Q_2 \gg Q_1^{\delta/6} Q_2^{(1-4\delta)/6}$.

Recall our dyadic decomposition of $E_3(B)$
after \eqref{global-E3-decomp}.
It now follows from   the Cauchy--Schwarz inequality that 
$E_3(B;M,T,\mathbf{Q})$ is 
\begin{align*}
&\ll \sum_{q_0\sim Q_0}
\frac{(Q_0Q_1Q_2+BM)^{-2}}{(1+M/B^{1/2})^A}
\Sigma(I,Q_2,T,M)^{1/2} \Sigma(\{1\},Q_2,T,M)^{1/2} \\
&\ll \frac{Q_0 M^3 (Q_1Q_2)^{1+\eps}}{(Q_0Q_1Q_2+BM)^2 (1+M/B^{1/2})^A}
\left(\frac{1}{Q_1^{\delta/12}} + \frac{1}{M^{\delta/2}}\right)
\left(\frac{Q_2^{1/2}}{M^{2(1-\delta)}} + \frac{1}{Q_2^{(1-4\delta)/6}}\right).
\end{align*}
Since we have  $(Q_0Q_1Q_2+BM)^2\ge (Q_0Q_1Q_2)^{(2-3\delta)/2} (BM)^{(2+3\delta)/2}$
and $(1+M/B^{1/2})^{A}\gg (1+T^2/B^{1/2})^{A}$, it follows that
\begin{align*}
\sum_{M\gg T^2}
|E_3(B;M,T,\mathbf{Q})|
&\ll \frac{Q_0^{3\delta/2} (Q_1Q_2)^{3\delta/2+\eps}}
{B^{(2+3\delta)/2} (1+T^2/B^{1/2})^{A}}\\
&~\times\left(\frac{(B^{1/2})^{\delta/2}}{Q_1^{\delta/12}} Q_2^{1/2}
+ Q_2^{1/2}
+ \frac{(B^{1/2})^{(4-3\delta)/2}}{Q_1^{\delta/12} Q_2^{(1-4\delta)/6}}
+ \frac{(B^{1/2})^{2(1-\delta)}}{Q_2^{(1-4\delta)/6}}\right) \\
&= \frac{Q_0^{3\delta/2} (B^{1/2})^{(4-3\delta)/2} (Q_1Q_2)^{3\delta/2+\eps}}
{B^{(2+3\delta)/2} (1+T^2/B^{1/2})^{A}}\\
&~\times \left(\frac{1}{Q_1^{\delta/12}} + \frac{1}{(B^{1/2})^{\delta/2}}\right)
\left(\frac{Q_2^{1/2}}{(B^{1/2})^{2(1-\delta)}} + \frac{1}{Q_2^{(1-4\delta)/6}}\right),
\end{align*}
since $3 - (2+3\delta)/2 = (4-3\delta)/2 = \delta/2 + 2(1-\delta)$.
We note that  $(2+3\delta)/2 - (4-3\delta)/4 = 9\delta/4$. Moreover, 
 $Q_1^{\delta/12} \ll B^{\delta/8}\ll B^{\delta/4}$
and $Q_2^{1/2 + (1-4\delta)/6} = Q_2^{2(1-\delta)/3} \ll B^{1-\delta}$,
since $Q_1,Q_2\ll B^{3/2}$.
% so the $B$-free terms $\frac{1}{Q_1^{\delta/12}}$ and $\frac{1}{Q_2^{(1-4\delta)/6}}$
% dominate their respective parenthetical expressions.
We conclude that 
\begin{align*}
\sum_{M\gg T^2}
|E_3(B;M,T,\mathbf{Q})|
&\ll \frac{Q_0^{3\delta/2} (Q_1Q_2)^{3\delta/2+\eps}}
{B^{9\delta/4} (1+T^2/B^{1/2})^{A}}
\frac{1}{Q_1^{\delta/12}}
\frac{1}{Q_2^{(1-4\delta)/6}}\\
&\ll \frac{(Q_0/B^{3/2})^{\min(\delta/12,(1-4\delta)/6)-\eps}}
{(1+T^2/B^{1/2})^{A}},
\end{align*}
since $Q_1^{\delta/12} Q_2^{(1-4\delta)/6} \gg (Q_1Q_2)^{\min(\delta/12,(1-4\delta)/6)}$
%and 
%$3\delta/2 > \delta/12 \ge \min(\delta/12,(1-4\delta)/6)$,
and $Q_0Q_1Q_2\ll B^{3/2}$.
The number of dyadic choices of $Q_1,Q_2\ll B^{3/2}/Q_0$  is $O((B^{3/2}/Q_0)^\eps)$. 
Thus we  may finally sum over $Q_1,Q_2\ll B^{3/2}/Q_0$
and take $\delta=2/9$ to get the following result.

\begin{lemma}
\label{absolute-tail-bound}
If $B,T,Q_0\ge 1$, then
\begin{equation*}
\sum_{M,Q_1,Q_2\ge 1} |E_3(B;M,T,\mathbf{Q})|
\ll \frac{(B^{3/2}/Q_0)^{ - 1/54}}{(1+T^2/B^{1/2})^{A}}.
\end{equation*}
\end{lemma}

\subsection*{Secondary main terms}

Returning to \eqref{global-E3-decomp},
we are now prepared to extract a new main term.
Given $B,T,Q_0\ge 1$, consider the piece
\begin{equation*}
\sum_{M,Q_1,Q_2\ge 1} E_3(B;M,T,\mathbf{Q})
= \sum_{\substack{(\m,\n)\ne \0 \\ D(\m,\n)=0 \\ |\mathbf{t}|\sim T}}
\sum_{\substack{q_0q_1q_2=q\ge 1 \\ q_0\sim Q_0}}
\frac{I_q(\m,\n)}{q^2}
\frac{\phi(q_0)}{q_0} \Xi_{q_1}(\m,\n) S''_{q_2}(\m,\n).
\end{equation*}
of $E_3(B)$.
Let us call this piece $\Sigma_1=\Sigma_1(B,T,Q_0)$. Then, 
on summing over $q_1q_2 = q/q_0$ using \eqref{dual-q0q1q2}, we get
\begin{equation*}
\Sigma_1
= \sum_{\substack{(\m,\n)\ne \0 \\ D(\m,\n)=0 \\ |\mathbf{t}|\sim T}}
\sum_{\substack{q_0q'=q\ge 1 \\ q_0\sim Q_0}}
\frac{I_q(\m,\n)}{q^2}
\frac{\phi(q_0)}{q_0} S'_{q'}(\m,\n).
\end{equation*}
Here $I_q(\m,\n)=0$ unless $q'\ll B^{3/2}/Q_0$,
by the properties of the $h$-function recorded at the start of Section \ref{sec:initial}.
We would like to approximate $\Sigma_1$ by the sum
\begin{equation*}
\Sigma_2=\Sigma_2(B,T,Q_0)\defeq \sum_{\substack{[\mathbf{t}]\in \PP^2(\ZZ) \\ |\mathbf{t}|\sim T}}
\sum_{(\m,\n)\in \Lambda^\perp(\mathbf{t})}
\sum_{\substack{q_0q'=q\ge 1 \\ q_0\sim Q_0}}
\frac{I_q(\m,\n)}{q^2}
\frac{\phi(q_0)}{q_0} S'_{q'}(\m,\n).
\end{equation*}

\begin{lemma}
\label{lattice-overcounts}
Let $M\ge 1$.
Then
\begin{equation*}
\sum_{\substack{[\mathbf{t}]\in \PP^2(\ZZ) \\ |\mathbf{t}|\sim T}}
\sum_{\substack{(\m,\n)\in \Lambda^\perp(\mathbf{t}) \\ |\m|,|\n|\ll M}}
(\1_{\m\n=\0} + \1_{n_1n_2n_3=0})
\ll T^2 (M+T)^2
+ M^3 \1_{T\asymp 1}
+ M^{2+\eps} \1_{M\gg T^2}.
\end{equation*}
\end{lemma}

\begin{proof}
We will make use of Lemma \ref{HB-lattice-bound}.
The contribution from $\m=\0$ is
$$
\ll \sum_{\mathbf{t}}\#\left\{|\n|\ll M+T:  \mathbf{n}.\mathbf{t}=0\right\}
\ll \sum_{\mathbf{t}} \frac{(M+T)^2}{T}
\ll T^2 (M+T)^2,
$$
on enlarging the range $|\n|\ll M$ into $|\n|\ll M+T$.
On the other hand, if $\m\ne\0$ and $\m\n=\0$, then $n_1n_2n_3=0$,
so by symmetry it remains to consider the contribution $N$, say,  from $\m\ne\0$, $n_3=0$.
If $t_1=t_2=0$, then $t_3=\pm 1$ and $N\ll M^3 \1_{T\asymp 1}$.
If $(t_1,t_2)\ne\0$, then
$$
N\ll \sum_{t_1,t_2,t_3\ll T} \frac{M}{T^2} \1_{M\gg T^2}
\frac{M \gcd(t_1,t_2)}{|(t_1,t_2)|},
$$
where the condition $M\gg T^2$ comes from $\m\ne\0$.
Arguing as in \eqref{gcd-sum}, we get $N\ll M^{2+\eps} \1_{M\gg T^2}$.
\end{proof}

By Lemma~\ref{non-coordinate-lattices},
% The $\m\n\ne\0$, $n_1n_2n_3\ne 0$ parts of $\Sigma_1$ and $\Sigma_2$ coincide exactly.
% the bound $S'_{q'}\ll (q')^{3/2+\eps}$ from
Remark~\ref{uniform-pointwise-bound},
and the bound $I_q\ll (1+|(\m,\n)|/B^{1/2})^{-A}$ from
Lemma~\ref{int-estimate}, we see that 
$\Sigma_1-\Sigma_2$ is 
\begin{equation*}
\ll 
\hspace{-0.3cm}
\sum_{\substack{q_0\sim Q_0 \\ q'\ll B^{3/2}/Q_0}}
\hspace{-0.3cm}
\frac{(q')^{3/2+\eps}}{Q_0^2 (q')^2}
\Bigg(\sum_{\substack{\m\n=\0 \\ |\mathbf{t}|\sim T}}
+ \sum_{\substack{\m\n\ne \0 \\ n_1n_2n_3=0 \\ D(\m,\n)=0 \\ |\mathbf{t}|\sim T}}
+ \sum_{\substack{[\mathbf{t}]\in \PP^2(\ZZ) \\ |\mathbf{t}|\sim T}}
\sum_{\substack{(\m,\n)\in \Lambda^\perp(\mathbf{t})
\\ \m\n=\0\textnormal{ or } n_1n_2n_3=0}}\Bigg)
\left(1+\frac{|(\m,\n)|}{B^{1/2}}\right)^{-A}
\hspace{-0.2cm}
.
\end{equation*}
By Lemmas~\ref{degen-counts} and~\ref{lattice-overcounts},
and the inequality $M^{2+\eps} \1_{M\gg T^2} \ll \frac{M^3}{T} \1_{M\gg T^2}$,
the inner sum is
\begin{equation*}
\ll \sum_{M\ge 1} \frac{M^3 \1_{T\asymp 1}
+ \frac{M^3}{T} \1_{M\gg T^2}
+ T^2 (M+T)^2}
{(1+M/B^{1/2})^A}
\ll B^{3/2} \1_{T\asymp 1}
+ \frac{B^{3/2}}{T}
% + \frac{B^{3/2}/T}{(1+T^2/B^{1/2})^{A/2}}
+ T^2 (B+T^2).
\end{equation*}
Therefore,
\begin{equation*}
(\Sigma_1-\Sigma_2)(B,T,Q_0)
\ll \frac{(B^{3/2}/Q_0)^{1/2+\eps} \1_{Q_0\ll B^{3/2}}}{Q_0}
\left(\frac{B^{3/2}}{T}
+ T^2 (B+T^2)\right).
\end{equation*}

Let $P=P(B,T) \defeq \min(T,B^{1/2}/T^2)^\delta$,
where $\delta\in (0,1)$ is fixed.
If $B^{3/2}/Q_0\ll P$, then
\begin{align*}
(\Sigma_1-\Sigma_2)(B,T,Q_0)
\ll \frac{P^{3/2+\eps} \1_{P\gg 1}}{T}
+ \frac{P^{1/2+\eps} \1_{P\gg 1}}{B^{3/2}/P} T^2 B
&\ll \frac{P^{3/2+\eps} \1_{P\gg 1}}{\min(T,B^{1/2}/T^2)}\\
&\ll \frac{\1_{P\gg 1}}{P^\delta},
\end{align*}
since $P^{3/2}/P^{1/\delta}\leq 1/P^{\delta}$ if $\delta$ is small enough. 
Since $\#\{Q_0: B^{3/2}/P\ll Q_0\ll B^{3/2}\} \ll P^\eps$
and $\#\{T: P(B,T) \sim P\}\ll_\delta 1$, we get
\begin{equation}
\label{large-q-P-sum}
\sum_{\substack{T\ge 1 \\ Q_0\gg B^{3/2}/P(B,T)}}
|(\Sigma_1-\Sigma_2)(B,T,Q_0)|
\ll \sum_{P\gg 1} \frac{P^\eps}{P^\delta}
\ll 1.
\end{equation}
On the other hand, by Lemma~\ref{absolute-tail-bound}, we have 
\begin{align*}
\sum_{\substack{T\ge 1 \\ Q_0\ll B^{3/2}/P(B,T)}}
|\Sigma_1(B,T,Q_0)|
&\ll \sum_{T\ge 1} \frac{P^{ - 1/54}}{(1+T^2/B^{1/2})^{A}}\\
&\ll \sum_{T\ge 1} \frac{T^{-\delta/54} + (T^2/B^{1/2})^{\delta/54}}{(1+T^2/B^{1/2})^{A}}\\
&\ll 1.
\end{align*}
Hence
\begin{equation}
\label{apt}
E_3(B)
= \sum_{\substack{T\ge 1 \\ Q_0\gg B^{3/2}/P(B,T)}}
\Sigma_2(B,T,Q_0)
+O(1).
\end{equation}

To evaluate $\Sigma_2$,
fix $S'_{q'}(\m,\n)$ using Lemma~\ref{q-invariance} to get
\begin{equation*}
\Sigma_2 = \sum_{\substack{[\mathbf{t}]\in \PP^2(\ZZ) \\ |\mathbf{t}|\sim T}}
\sum_{\substack{q_0q'=q\ge 1 \\ q_0\sim Q_0}}
\frac{\phi(q_0)}{q_0}
\sum_{(\a,\b)\in \Lambda^\perp(\mathbf{t})/q'\Lambda^\perp(\mathbf{t})}
S'_{q'}(\a,\b)
\sum_{(\m,\n)\in (\a,\b)+q'\Lambda^\perp(\mathbf{t})}
\frac{I_q(\m,\n)}{q^2}.
\end{equation*}
% (The idea is to make a multiplicative Dirichlet series approximation of
% $\sum_{q\ge 1} S_q(\m,\n) q^{-s}$
% in terms of a Riemann-zeta-like series independent of $(\m,\n)$,
% as in \cite{wang2024dual}*{(8.3)}
% or \cite{BGW2024forthcoming}*{(10.7)}.
% We then factor $q$ into a zeta modulus $q_0$ and an ``error modulus'' $q'$,
% and fix $\m,\n\bmod{q'}$.)
Let $\lambda_1(\Lambda^\perp)\leq \lambda_2(\Lambda^\perp)\leq \lambda_3(\Lambda^\perp)$
be the successive minima of $\Lambda^\perp=\Lambda^\perp(\mathbf{t})$.
Choose a matrix $\mathbf{\Lambda^\perp}\in \mathrm{Mat}_{3\times 6}$
whose three row vectors form a shortest basis of $\Lambda^\perp$.
Consider the real density
\begin{equation}
\label{EQN:real-density-linear-spaces}
\sigma_{\infty,\Lambda^\perp,W}
\defeq \lim_{\epsilon\to 0}{(2\epsilon)^{-3}
\int_{\mathbf{\Lambda^\perp}(\x,\y)\in [-\epsilon,\epsilon]^3}
W(\x,\y)}\,\mathrm{d}\x\,\mathrm{d}\y,
\end{equation}
where we view $(\x,\y)$ as a column vector in $\mathrm{Mat}_{6\times 1}$.
If
\begin{equation}
\label{stable-domain}
\frac{q_0}{B} \ge A(W) \lambda_3(\Lambda^\perp),
\end{equation}
where $A(W)>0$ is a large constant defined in terms of $W$,
then by \cite{wang2024dual}*{proof of Lemma~4.6},
% or \cite{BGW2024forthcoming}*{proof of Lemma~10.11},
which amounts to Poisson summation
% https://en.wikipedia.org/wiki/Projection-slice_theorem
% https://en.wikipedia.org/wiki/Projection-slice_theorem#The_generalized_Fourier-slice_theorem
and a Fourier-slice identity,
we have
\begin{equation}
\label{poisson-fourier-slice}
q_0^{-3}
\sum_{(\m,\n)\in (\a,\b)+q'\Lambda^\perp}
I_q(\m,\n) B^6
= \sigma_{\infty,\Lambda^\perp,W} B^3
\, h(q/Q,0),
\end{equation}
where $(q_0,q')$ corresponds to
what \cite{wang2024dual} calls $(n_1,n_0)$.
Strictly speaking, the proof requires $q_0/B$
to exceed a certain constant times the length of the shortest basis
$\mathbf{\Lambda^\perp}$ of $\Lambda^\perp$.
However, this is equivalent to \eqref{stable-domain},
by Lemma~\ref{shortest-basis}.

We can also interpret \eqref{EQN:real-density-linear-spaces}
in terms of the orthogonal lattice $\Lambda$ from \eqref{Lambda-t}.
% By Poisson summation of $w(\frac{\x,\y}{H})$
% over $(\x,\y)\in \Lambda$ (cf.~(1.7) of [W24]),
% compared with Schmidt's lattice-counting asymptotic formula, we have
% % Geometry of numbers reference: https://arxiv.org/abs/2006.02356 (and Schmidt within)
% \begin{equation}
% \sigma_{\infty,\Lambda^\perp,w}
% \ll_w \frac{1}{\mathrm{vol}(\Lambda\otimes \mathbb{R}/\Lambda)}
% \asymp \frac{1}{|\mathbf{t}|^2\cdot |\mathbf{t}|},
% \end{equation}
% since the $\y$-component lattice $\y\in \mathbf{t}\mathbb{Z}$
% has co-volume $|\mathbf{t}|$ in $\mathbf{t}\mathbb{R}$
% and the $\x$-component lattice $\x\cdot \mathbf{t}^2=0$
% has co-volume equal to that of its orthogonal complement $\mathbf{t}^2\mathbb{Z}$.
By Poisson summation of $W(\frac{\x,\y}{H})$
over $(\x,\y)\in \Lambda$,
along the lines of \cite{wang2024dual}*{Eq.~(1.7)},
we have
\begin{equation}
\label{lattice-limit}
\sigma_{\infty,\Lambda^\perp,W}
= \lim_{H\to \infty} \frac{\sum_{(\x,\y)\in \Lambda} W(\frac{\x,\y}{H})}{H^3}.
% \ll_W \frac{1}{|\mathbf{t}^2|\cdot |\mathbf{t}|},
\end{equation}
% since $\x\cdot \mathbf{t}^2=0$
% may be viewed as a congruence modulo the largest coordinate of $\mathbf{t}^2$.
The actual interpretation we will use is \eqref{general-formula},
a hybrid between \eqref{EQN:real-density-linear-spaces} and \eqref{lattice-limit}.
%It should also be possible to interpret
%$\sigma_{\infty,\Lambda^\perp(\mathbf{t}),W}$
%as a Fourier integral,
%as in \cite{BHBDuke}.

\begin{lemma}
\label{separable-formula}
If $|t_3|\asymp |\mathbf{t}|$, say, then
\begin{equation}
\label{general-formula}
\sigma_{\infty,\Lambda^\perp(\mathbf{t}),W}
= \int_{\RR}
\int_{\RR^2} W\left(x_1,x_2,-\frac{x_1t_1^2+x_2t_2^2}{t_3^2},s\mathbf{t}\right)
\,\frac{\mathrm{d}x_1\,\mathrm{d}x_2}{t_3^2}\, \mathrm{d}s.
\end{equation}
Thus $\sigma_{\infty,\Lambda^\perp(\alpha\mathbf{t}),W}
= |\alpha|^{-3} \sigma_{\infty,\Lambda^\perp(\mathbf{t}),W}$
for all $\alpha\in \RR^\times$.
Moreover,
$$\sigma_{\infty,\Lambda^\perp(\mathbf{t}),W} \ll |\mathbf{t}|^{-3},
\qquad
\nabla_{\mathbf{t}}(\sigma_{\infty,\Lambda^\perp(\mathbf{t}),W}) \ll |\mathbf{t}|^{-4}.$$
\end{lemma}

\begin{proof}
We may express $\sigma_{\infty,\Lambda^\perp(\mathbf{t}),W}$
as a continuous function of $W\in C_c(\RR^6)$,
by \eqref{EQN:real-density-linear-spaces},
after choosing a Leray form on
the smooth complete intersection $\mathbf{\Lambda^\perp}(\x,\y) = \0$ in $\RR^6$.
By continuity,
since $C^\infty_c(\RR^3) \otimes C^\infty_c(\RR^3)$
is dense in $C_c(\RR^6)$ for the $\sup$ norm,
% https://en.wikipedia.org/wiki/Uniform_norm
it thus suffices to prove \eqref{general-formula}
assuming that $W(\x,\y) = W_1(\x)W_2(\y)$,
in which case the following three observations imply \eqref{general-formula}.
First, \eqref{EQN:real-density-linear-spaces}
factors as a product of two densities,
one associated to the two-dimensional lattice $\x\cdot \mathbf{t}^2=0$,
and the other associated the one-dimensional lattice $\y\in \mathbf{t}\ZZ$.
Second,
$\frac{\mathrm{d}x_1\,\mathrm{d} x_2}{t_3^2}$ is a Leray form on $\x\cdot \mathbf{t}^2=0$.
Third,
$\frac1H \sum_{\y\in \mathbf{t}\ZZ} W_2(\frac{\y}{H})
\to \int_{\RR} W_2(s\mathbf{t})\, \mathrm{d}s$
as $H\to \infty$,
because if $f(s)\defeq W_2(s\mathbf{t})$
then $\frac1H \sum_{h\in \ZZ} f(\frac{h}{H}) \to \int_\RR f(s)\, \mathrm{d}s$.
% \VW{This is assuming $w(\x,\y)=w_1(\x)w_2(\y)$.
% % We can try to do more general weights by Fourier inversion and/or other methods for separation of variables, if desired.
% The general case probably follows by continuity.
% Alternatively, methods from the main term analysis of \cite{BHBDuke} might be useful.}
The rest of Lemma~\ref{separable-formula} follows from \eqref{general-formula}.
\end{proof}

Among the four vectors
$(\mathbf{t}^2,\0)$,
$(\0,(0,t_3,-t_2))$,
$(\0,(t_3,0,-t_1))$,
$(\0,(t_2,-t_1,0))$
in $\Lambda^\perp$,
some three of them are linearly independent (e.g.~the first three, if $t_3\ne 0$). Thus
\begin{equation}
\label{bound-lambda3}
\lambda_3(\Lambda^\perp)\ll |\mathbf{t}|^2,
\end{equation}
which is optimal in view of the condition $\m\in \mathbf{t}^2\mathbb{Z}$ in \eqref{Lambda-perp-t}.

If $C_1B^{3/2}/P(B,T)\le Q_0\le C_2B^{3/2}$,
then $B^{1/2}/T^2 \ge P^{1/\delta} \ge (C_1/C_2)^{1/\delta}$, so
\begin{equation*}
\frac{Q_0}{B} \ge \frac{C_1B^{1/2}}{P}
\ge C_1P^{(1-\delta)/\delta}T^2
\ge C_1(C_1/C_2)^{(1-\delta)/\delta}T^2.
\end{equation*}
Hence,  if $C_1$ is large enough in terms of $W$,
then by \eqref{bound-lambda3} the condition \eqref{stable-domain} is satisfied
for all $q_0\sim Q_0$.
Thus by \eqref{poisson-fourier-slice}
\begin{equation*}
\Sigma_2 = \sum_{\substack{[\mathbf{t}]\in \PP^2(\ZZ) \\ |\mathbf{t}|\sim T}}
\sum_{\substack{q_0q'=q\ge 1 \\ q_0\sim Q_0}}
\frac{\phi(q_0)}{q_0}
\frac{q_0^3 \sigma_{\infty,\Lambda^\perp,W} h(q/Q,0)}{q^2 B^3}
\sum_{(\a,\b)\in \Lambda^\perp(\mathbf{t})/q'\Lambda^\perp(\mathbf{t})}
S'_{q'}(\a,\b).
\end{equation*}
Arguing as in
\cite{wang2024dual}*{proofs of Propositions~4.7 and~8.4},
we find that the average value of $q^{-4} S_q(\m,\n)$
over $(\m,\n)\in \Lambda^\perp(\mathbf{t})/q\Lambda^\perp(\mathbf{t})$ is
\begin{equation*}
q^{-4} \sum_{a\in (\ZZ/q\ZZ)^\times}
\sum_{(\x,\y)\in \Lambda(\mathbf{t})/q\Lambda(\mathbf{t})}
e_q(aF(\x,\y))
= q^{-4} \sum_{a\in (\ZZ/q\ZZ)^\times}
\sum_{(\x,\y)\in \Lambda(\mathbf{t})/q\Lambda(\mathbf{t})}
1
= q^{-1} \phi(q),
\end{equation*}
and thus that the average value of $S'_q(\m,\n)$
over $(\m,\n)\in \Lambda^\perp(\mathbf{t})/q\Lambda^\perp(\mathbf{t})$ is
$\1_{q=1}$, by \eqref{dual-q0q1q2}.
Therefore,
\begin{equation*}
\Sigma_2 = \sum_{\substack{[\mathbf{t}]\in \PP^2(\ZZ) \\ |\mathbf{t}|\sim T}}
\sum_{q\sim Q_0}
\phi(q) 
\frac{\sigma_{\infty,\Lambda^\perp(\mathbf{t}),W} h(q/Q,0)}{B^3}.
\end{equation*}
We may sum over $q$ using the $h$-function identity
$$
\sum_{q\geq Q/L} \phi(q) h(q/Q,0)
= Q^2(1 + O(L^{-1})) \1_{L\ge 1},
% = Q^2(1 + O(L^{-1})),
$$
which follows from  \cite{wang2024dual}*{Proposition~8.11}
if $L\ge 1$,
and from  the identity $h(q/Q,0) \1_{q\ge Q} = 0$
in \cite{HB}*{Lemma~4}
if $L\le 1$.
Plugging our work into \eqref{apt},
and bounding the total error term using the
inequality $\sigma_{\infty,\Lambda^\perp(\mathbf{t}),W} \ll |\mathbf{t}|^{-3}$,
we get
\begin{equation*}
\begin{split}
E_3(B)
- \sum_{\substack{T\ge 1 \\ P(B,T)\gg 1}}
\sum_{\substack{[\mathbf{t}]\in \PP^2(\ZZ) \\ |\mathbf{t}|\sim T}}
\sigma_{\infty,\Lambda^\perp,W}
&\ll 1
+ \sum_{\substack{T\ge 1 \\ P(B,T)\gg 1}}
\sum_{\substack{[\mathbf{t}]\in \PP^2(\ZZ) \\ |\mathbf{t}|\sim T}}
\frac{|\sigma_{\infty,\Lambda^\perp,W}|}{P(B,T)} \\
&\ll 1
+ \sum_{\substack{T\ge 1 \\ P(B,T)\gg 1}} \frac{1}{P(B,T)}
\ll 1,
\end{split}
\end{equation*}
similarly to \eqref{large-q-P-sum}.

% Since $f(\mathbf{t})\defeq |\mathbf{t}|^3 \sigma_{\infty,\Lambda^\perp(\mathbf{t}),w}$
% is a continuous function of $\mathbf{t}\in \RR^3 \setminus \{\0\}$
% that is invariant under scaling by $\RR^\times$
% (i.e.~$f(\alpha\mathbf{t})=f(\mathbf{t})$ for all $\alpha\in \RR^\times$),
% sieving in annuli gives

For any integer $1\le d\ll T$ we have, by Lemma~\ref{separable-formula},
\begin{equation*}
\sum_{\substack{\mathbf{t}\in d\ZZ^3 \\ |\mathbf{t}|\sim T}}
\sigma_{\infty,\Lambda^\perp(\mathbf{t}),W}
= \sum_{\substack{\mathbf{u}\in \ZZ^3 \\ |\mathbf{u}|\sim T/d}}
\frac{\sigma_{\infty,\Lambda^\perp(\mathbf{u}),W}}{d^3}
= \frac{1}{d^3}
\left(\theta(T/d)
+ \int_{\substack{\mathbf{u}\in \RR^3 \\ |\mathbf{u}|\asymp T/d}}
O(|\mathbf{u}|^{-4})\, \mathrm{d}\mathbf{u}\right),
\end{equation*}
where for any $U>0$ we let
\begin{equation*}
\theta(U)\defeq \int_{\substack{\mathbf{u}\in \RR^3 \\ |\mathbf{u}|\sim U}}
\sigma_{\infty,\Lambda^\perp(\mathbf{u}),W}\, \mathrm{d}\mathbf{u}.
\end{equation*}
Moreover,
% if we call the first integral $\theta(T/d)$, then
$\theta(U) = \theta(1)$
by scale-invariance of the volume form
$\sigma_{\infty,\Lambda^\perp(\mathbf{u}),W}\, d\mathbf{u}$ on $\RR^3$.
Thus, by M\"{o}bius inversion,
\begin{equation*}
2 \sum_{\substack{[\mathbf{t}]\in \PP^2(\ZZ) \\ |\mathbf{t}|\sim T}}
\sigma_{\infty,\Lambda^\perp,W}
= \sum_{d\ll T} \frac{\mu(d)}{d^3} (\theta(1) + O(\frac{d}{T}))
= \frac{\theta(1)}{\zeta(3)} + O(T^{-1}).
\end{equation*}

After writing the condition $P(B,T)\gg 1$ in the form $1\ll T\ll B^{1/4}$, we get
\begin{equation*}
E_3(B)
= \sum_{1\ll T\ll B^{1/4}} \frac{\theta(1)}{2\zeta(3)}
+ O(1)
= \frac{\theta(1)}{8\zeta(3)} \frac{\log{B}}{\log{2}}
+ O(1).
\end{equation*}
Plugging Lemma~\ref{separable-formula} into the definition of $\theta(1)$,
% \begin{equation*}
% \theta(1)
% = \int_{\substack{\mathbf{t}\in \RR^3 \\ |\mathbf{t}|\sim 1}}
% \int_{\RR^2}
% \int_{\RR}
% W(x_1,x_2,-\frac{x_1t_1^2+x_2t_2^2}{t_3^2},s\mathbf{t})
% \, \mathrm{d}s
% \,\frac{\mathrm{d}x_1\,\mathrm{d}x_2}{t_3^2}
% \, \mathrm{d}\mathbf{t}.
% \end{equation*}
and changing variables from $\mathbf{t}$ to $\y\defeq s\mathbf{t}$,
then integrating over $|s|\sim |\y|$,
where $s$ may lie in either component of $\RR^\times$,
we get
\begin{equation*}
\begin{split}
\theta(1)
&= \int_{\RR^3}
\int_{\RR^2}
\int_{|s|\sim |\y|}
W(x_1,x_2,-\frac{x_1y_1^2+x_2y_2^2}{y_3^2},\y)
\, \mathrm{d}s
\,\frac{\mathrm{d}x_1\,\mathrm{d}x_2}{(y_3/s)^2}
\, \frac{\mathrm{d}\y}{|s|^3} \\
&= (2\log{2}) \int_{\RR^3}
\int_{\RR^2} W(x_1,x_2,-\frac{x_1y_1^2+x_2y_2^2}{y_3^2},\y)
\,\frac{\mathrm{d}x_1\,\mathrm{d}x_2}{y_3^2}
\, \mathrm{d}\y\\
&= (2\log{2}) \sigma_\infty,
\end{split}
\end{equation*}
since $\frac{\mathrm{d}x_1\,\mathrm{d}x_2\,\mathrm{d}\y}{y_3^2}$ is a Leray form on the hypersurface $F(\x,\y)=0$.
This completes the proof of Proposition~\ref{final-E3-estimate}.

\appendix

\section{Another example of \texorpdfstring{$(\diamond)$}{(diamond)}}
\label{elementary}

Let $F(\x,\y) = \sum_{i=1}^{3} x_iy_i^2 + x_1x_2x_3$,
 let  $p\ge 3$ be a prime and let $\psi:\FF_p\to \CC^*$ be a non-trivial additive character.
Then by fixing $\x\in \FF_p^3$, we get
\begin{equation*}
\begin{split}
S_p(\m,\n)
&\defeq \sum_{a\ne 0} \sum_{\x,\y} \psi(aF(\x,\y) + \m\cdot\x+\n\cdot\y) \\
&= \sum_{a\ne 0} \sum_{\x} \psi(ax_1x_2x_3 + \m\cdot\x)
\prod_{i=1}^{3} G(ax_i,n_i),
\end{split}
\end{equation*}
where
$$
G(a,b)
\defeq \sum_{y\in \FF_p} \psi(ay^2+by)
= \psi\left(-\frac{b^2}{4a}\right) \left(\frac{a}{p}\right) G(1,0) \1_{a\ne 0} + p \1_{a=b=0}.
$$
Evaluating the Gauss sums over $\y\in \FF_p^3$,
we therefore find that if $n_1n_2n_3\ne 0$ in $\FF_p$, then
\begin{equation*}
\begin{split}
S_p(\m,\n)
&= \sum_{ax_1x_2x_3\ne 0} \psi(ax_1x_2x_3 + \m\cdot\x)
\prod_{i=1}^{3} \psi\left(-\frac{n_i^2}{4ax_i}\right) \left(\frac{ax_i}{p}\right) G(1,0) \\
&= G(1,0)^3 \sum_{ax_1x_2x_3\ne 0} \left(\frac{ax_1x_2x_3}{p}\right)
\psi\left(ax_1x_2x_3 + \m\cdot\x - \sum_{i=1}^{3} \frac{n_i^2}{4ax_i}\right).
\end{split}
\end{equation*}

By the Sali\'{e} sum formula
in  \cite[Lemma~12.4]{IK}, we obtain
% https://blogs.ethz.ch/kowalski/2010/03/07/the-computation-of-salie-sums/
% https://arxiv.org/abs/1107.5282
\begin{align*}
S(m,n;p) &\defeq \sum_{a\ne 0} \left(\frac{a}{p}\right) \psi(ma + na^{-1})\\
&= \left(\frac{m}{p}\right) \sum_{v^2\equiv mn\bmod{p}} \psi(2v) G(1,0) \1_{m\ne 0}
+ \left(\frac{n}{p}\right) G(1,0) \1_{m=0}.
\end{align*}
Applying this with $m = x_1x_2x_3 \ne 0$ and $n = -\sum_{i=1}^{3} \frac{n_i^2}{4x_i}$,
we get
\begin{equation*}
\begin{split}
S_p(\m,\n)
&= G(1,0)^4 \sum_{x_1x_2x_3\ne 0} 
\left(\frac{x_1x_2x_3}{p}\right) \psi(\m\cdot\x)
\left(\frac{x_1x_2x_3}{p}\right) 
\hspace{-0.2cm}
\sum_{v^2\equiv -\frac14 \sum_{i=1}^{3} n_i^2x_jx_k \bmod{p}} 
\hspace{-0.3cm}
\psi(2v) \\
&= p^2 \sum_{x_1x_2x_3\ne 0}
\sum_{u^2 = -\sum_{i=1}^{3} n_i^2x_jx_k} \psi(\m\cdot\x + u),
\end{split}
\end{equation*}
where $u \defeq 2v$.
Now, replacing $(\x,u)$ with $(t\x,tu)$
and averaging over $t\ne 0$, we get
\begin{equation*}
\begin{split}
S_p(\m,\n)
&= p^2 \sum_{x_1x_2x_3\ne 0}
\sum_{u^2 = -\sum_{i=1}^{3} n_i^2x_jx_k} \frac{p \1_{\m\cdot\x + u=0} - 1}{p-1} \\
&= p^2 \frac{p (N_1-N_2) - (N_3-N_4)}{p-1},
\end{split}
\end{equation*}
where 
\begin{align*}
N_1 &= \sum_{\x} \1_{(\m\cdot\x)^2 = -\sum_{i=1}^{3} n_i^2x_jx_k},
\qquad N_2 = \sum_{x_1x_2x_3=0} \1_{(\m\cdot\x)^2 = -\sum_{i=1}^{3} n_i^2x_jx_k},\\
N_3 &= \sum_{\x,u} \1_{u^2 = -\sum_{i=1}^{3} n_i^2x_jx_k},
\hspace{1.35cm} N_4 = \sum_{x_1x_2x_3=0} \sum_{u} \1_{u^2 = -\sum_{i=1}^{3} n_i^2x_jx_k}.
\end{align*}
If the projective conic in $N_1$ is smooth (a condition defined by the non-vanishing of a $3\times 3$ determinant, which is a sextic form in $(\m,\n)$),
then 
$$
\frac{N_1-1}{p-1} = \#\PP^1(\FF_p) = p+1,
$$
so $N_1 = p^2$.
Typically, the projective locus in $N_2$ will be zero-dimensional, so that $N_2\ll p$.
The projective quadric surface in $N_3$ has determinant $n_1^2n_2^2n_3^2/4$,
so if $n_1n_2n_3\ne 0$ then $$\frac{N_3-1}{p-1} = \#\PP^2(\FF_p) + p,
$$
so $N_3 = p^3 + (p^2-p) = p^3 + O(p^2)$.
The projective locus in $N_4$ is one-dimensional, so $N_4\ll p^2$.
Thus, typically,
\begin{equation*}
S_p(\m,\n)
\ll p^2 \frac{p^2}{p-1}
\ll p^3.	
\end{equation*}
This justifies the claim for $(c_1,c_2) = (-1,0)$ 
in Example~\ref{prelim-examples}(4).

\end{document}